\documentclass[english,a4paper]{amsart}

\newcommand{\scr}[1]{\mathscr{#1}}
\usepackage{hyperref}
\usepackage{xy}
\input xy
\xyoption{all}
\usepackage{verbatim}

\newcommand{\triplearrows}{\begin{smallmatrix} \to \\ \to \\ 
\to \end{smallmatrix} }

\hypersetup{
 hidelinks,colorlinks=false,breaklinks=true,
 pdfkeywords={equivariant homotopy theory, Artin's theorem, Brauer's theorem, induction, spectral sequences, K-theory, topological modular forms, tensor triangulated categories, Quillen's F-isomorphism theorem, group cohomology},
 pdfauthor={Akhil Mathew, Niko Naumann, Justin Noel},
}
\usepackage{xr}
\usepackage{Definitions}
\usepackage{Environments2}
\usepackage{PageSetup}
\usepackage{amssymb}
\usepackage{enumitem} 
\usepackage{geometry} 
\usepackage[toc,page]{appendix} 
\usepackage{mathtools} 

\usepackage[
disable=true, 
color=orange!80, 
bordercolor=black,
textwidth=.8in,
textsize=small]
{todonotes}

\usepackage{caption} 
\usepackage{subcaption}
\DeclareCaptionSubType[Alph]{figure}
\usepackage[bb=ams, cal=euler, scr=rsfso]{mathalfa}

\usepackage{tabu}
\makeatletter \providecommand\@dotsep{5}
\makeatother

\newcommand{\Desc}{\mathrm{Desc}}
\newcommand{\fun}{\mathrm{Fun}}
\newcommand{\Fun}{\mathrm{Fun}}

\newcommand{\Res}{\mathrm{Res}}
\renewcommand{\Top}{{\mathcal{S}}}

\usepackage{url} 
\renewcommand{\ho}{\mathrm{Ho}\,}
\newcommand{\Spec}{\mathrm{Spec}}
\newcommand{\SpG}{\mathrm{Sp}_G}
\newcommand{\SpH}{\mathrm{Sp}_H}
\newcommand{\SpK}{\mathrm{Sp}_K}
\newcommand{\Spe}{\mathrm{Sp}}
\theoremstyle{definition}
\newtheorem{thmA}{Theorem}

\renewcommand\Im{\mathrm{Im}}
\Crefname{thmA}{Theorem}{Theorems}
\newcommand{\sOG}{\sO(G)}

\newcommand{\ZOG}{\bZ\sO(G)}
\newcommand{\ZOGF}{\bZ\sO(G)_{\sF}}
\newcommand{\ZOGC}{\bZ\sO(G)_{\sC}}

\newcommand{\sOGF}{\sO_{\sF}(G)}

\newcommand{\sOGEp}{\sO(G)_{\sE_p}}

\newcommand{\sOGC}{\sO(G)_{\sC}}

\newcommand{\bb}[1]{\underline{#1}}
\newcommand{\sAll}{\sA\ell\ell}
\newcommand{\sAb}{\sA}
\newcommand{\sAbone}{\sAb^1}
\newcommand{\sAbn}{\sAb^n}
\newcommand{\sAbp}{\sAb_{(p)}}
\newcommand{\sTriv}{\sT}
\newcommand{\Tow}{\mathrm{Tow}}
\newcommand{\Townil}{\Tow^{\textrm{nil}}}
\newcommand{\Towfast}{\Tow^{\textrm{fast}}}

\newcommand{\sEp}{\sE_{(p)}}

\newcommand{\sAllp}{\sAll_{(p)}}

\newcommand{\sFNilPr}{{\sF^{\textrm{Nil}}}^\prime}

\newcommand{\sFNil}{\sF^{\textrm{Nil}}}

\newcommand{\sAbpn}{\sAb_{(p)}^n}
\newcommand{\sAbpnNil}{\sA_{(p)}^{n,\textrm{Nil}}}
\newcommand{\GSpec}{\mathrm{Sp}_G}

\newcommand{\KG}{KU}
\newcommand{\koG}{ko}
\newcommand{\KOG}{\mathit{KO}}
\newcommand{\kG}{ku}
\newcommand{\MUG}{\mathit{MU}}
\newcommand{\MOG}{\mathit{MO}}
\newcommand{\MUR}{M\bR}

\newcommand{\ko}{\bb{\mathit{ko}}}
\newcommand{\ku}{\bb{\mathit{ku}}}
\newcommand{\KO}{\bb{\mathit{KO}}}
\newcommand{\KU}{\bb{\mathit{KU}}}
\newcommand{\tmf}{\bb{\mathit{tmf}}}
\newcommand{\Tmf}{\bb{\mathit{Tmf}}}
\newcommand{\TMF}{\bb{\mathit{TMF}}}
\newcommand{\MO}{\bb{\mathit{MO}}}
\newcommand{\MSO}{\bb{\mathit{MSO}}}
\newcommand{\MSpin}{\bb{\mathit{MSpin}}}
\newcommand{\MString}{\mathit{MString}}
\newcommand{\MSp}{\mathit{MSp}}

\newcommand{\Kn}{\bb{K(n)}}
\newcommand{\Gpi}{\pi_*^{(-)}}
\newcommand{\Gpit}{\pi_t^{(-)}}
\renewcommand{\mod}{\mathrm{Mod}}
\newcommand{\ring}{\mathrm{Ring}}
\renewcommand{\theenumi}{\arabic{enumi}}

\makeatletter
\renewcommand{\p@enumii}{\theenumi.}
\makeatother
\newtheorem{cons}[thm]{Construction}
\begin{document}
	\title{Derived induction and restriction theory}

	\author{Akhil Mathew}
	\address{University of Chicago, Chicago, IL USA}
	\email{amathew@math.uchicago.edu}
	\urladdr{http://math.uchicago.edu/~amathew/}

	\author{Niko Naumann}
	\address{University of Regensburg\\
	NWF I - Mathematik; Regensburg, Germany}
	\email{Niko.Naumann@mathematik.uni-regensburg.de}
	\urladdr{http://homepages.uni-regensburg.de/~nan25776/}

	\author{Justin Noel}
	\address{University of Regensburg\\
	NWF I - Mathematik; Regensburg, Germany}
	\email{justin.noel@mathematik.uni-regensburg.de}
	\urladdr{http://nullplug.org}
	\thanks{The first author was supported by the NSF Graduate Research
	Fellowship under grant DGE-110640, and the work was finished while the
	first author
	was a Clay Research Fellow. 
	The second author was partially supported by the SFB 1085 - Higher Invariants, Regensburg.
	The third author was partially supported by the DFG grants: NO 1175/1-1 and
	SFB 1085 - Higher Invariants, Regensburg. }

\date{\today}

\begin{abstract}
	Let $G$ be a finite group.  To any family $\sF$ of subgroups of $G$, we
	associate a thick $\otimes$-ideal $\sFNil$ of the category of $G$-spectra
	with the property that every $G$-spectrum in $\sFNil$ (which we call
	$\sF$-nilpotent) can be reconstructed from its underlying $H$-spectra as $H$
	varies over $\sF$. A similar result holds for calculating $G$-equivariant
	homotopy classes of maps into such spectra via an appropriate homotopy limit
	spectral sequence. In general, the condition $E\in \sFNil$ implies strong
	collapse results for this spectral sequence as well as its dual homotopy
	colimit spectral sequence. As applications, we obtain Artin and Brauer type
	induction theorems for $G$-equivariant $E$-homology and
	cohomology, and generalizations of Quillen's $\cF_p$-isomorphism theorem when $E$ is a homotopy commutative $G$-ring spectrum.

	We show that the subcategory   $\sFNil$ contains many $G$-spectra of interest for
	relatively small families $\sF$. These include $G$-equivariant real and complex $K$-theory as well as the Borel-equivariant cohomology theories associated to complex oriented ring spectra, the $L_n$-local sphere, the classical bordism theories, connective real $K$-theory, and any of the standard variants of topological modular forms. In each of these cases we identify the minimal family such that these results hold.
\end{abstract}

\maketitle
 
\tableofcontents

\section{Introduction}
\subsection{Motivation and overview}
Let $G$ be a finite group and $R(G)$ the Grothendieck ring of
finite-dimensional complex representations of $G$.  One can ask if
$R(G)$ is determined by the representation rings $R(H)$ as $H$ varies over some
set $\cC$ of subgroups of $G$. For example, every $G$-representation $V$ has an
underlying, or restricted, $H$-representation $\Res_H^G V$, and we can ask if
the product of the restriction maps \[  \Res^G_{\cC}: R(G)\longrightarrow
\prod_{H\in \cC} R(H)\] is injective. By elementary character theory, this holds if $\cC$ contains the cyclic subgroups of $G$.

Alternatively, associated to every $H$-representation $W$ is an induced $G$-representation $\Ind_H^G W$ and we can ask if the direct sum of the induction maps 
\begin{equation*}
\Ind^G_{\cC}  : \bigoplus_{H\in \cC} R(H)\longrightarrow R(G)
 \end{equation*}
  is surjective. This holds if $\cC$ contains the Brauer elementary subgroups of $G$, i.e., subgroups of the form $C\times P$, where $P$ is a $p$-group and $C$ is a cyclic group of order prime to $p$ \cite[\S 10.5, Thm.~19]{Ser77}. 

In general,  there are strong  
restrictions on elements in the image of the restriction homomorphism: 
for example,
an element $\{W_H\}\in \prod_{H\in \cC} R(H)$ can only be in the image of $\Res^G_{\cC}$ if 
\begin{enumerate}
 \item\label{it:res-compat} for every pair of subgroups $H_1, H_2\in \cC$ such that $H_1\leq H_2$, $\Res_{H_1}^{H_2}W_{H_2}=W_{H_1}$ and
 \item\label{it:conj-compat} for every pair of subgroups $H_1, H_2\in \cC$ and
 $g \in G$ such that $gH_1g^{-1}=H_2$, $W_{H_2}$ is the image of $W_{H_1}$ under the isomorphism $R(H_1)\xrightarrow{\sim} R(H_2)$  induced by conjugating $H_1$ by $g$.
 \end{enumerate}

 In this paper, we consider only the case where $\cC = \sF$ is a \emph{family}
 of subgroups, that is, a nonempty collection of subgroups closed under
 subconjugation.\footnote{There is also a rich literature in the more general
 case where $\cC$ is only closed under conjugation; see
 subsection~\ref{relatedwork} for some references.} Then one can consider the
 subset of the product $\prod_{H\in \sF} R(H)$ consisting of those elements
 which satisfy conditions \eqref{it:res-compat} and \eqref{it:conj-compat}.
 This subset can be identified with a certain limit, $\lim_{\sOGF^\op} R(H)$,
 indexed over a subcategory $\sOGF$ of the orbit category of $G$,  and the restriction map naturally lifts to this limit. 

 We can apply a dual construction for the induction homomorphism to obtain maps which factor through the induction and restriction maps above:
 \begin{equation}\label{eq:ind-res}
  \bigoplus_{G/H\in \sOGF}R(H)\twoheadrightarrow \colim_{\sOGF} R(H)\xrightarrow{\Ind_{\sF}^G} R(G)\xrightarrow{\Res_{\sF}^G} \lim_{\sOGF^\op} R(H)\hookrightarrow \prod_{G/H\in \sOGF^\op}R(H).
  \end{equation}

 If $\sF$ is a family of subgroups which contains the Brauer elementary
 subgroups, then both $\Ind_{\sF}^G$ and $\Res_{\sF}^G$ are
 isomorphisms.\footnote{In fact, these maps are isomorphisms if and only if
 $\sF$ contains the Brauer elementary subgroups \cite[\S 11.3, Thm.~23]{Ser77}.} If
 instead we set $\sF$ to be the generally smaller family of cyclic subgroups, these maps
 are isomorphisms after inverting the order of $G$. We can regard these two
 results as forms of the induction/restriction theorems of Brauer and Artin respectively \cite[Chaps.~9-10]{Ser77}.

A formally analogous result occurs in the theory of group cohomology. Let $A$ be
a $\mathbb{Z}[G]$-module which is $p$-power torsion. 
Then we can consider the group cohomology $H^*(H; A)$ for each subgroup $H \leq
G$; under restriction (and conjugation) of group cohomology classes, we obtain a presheaf of abelian groups on $\sOG$. 
If $\sF$ is a family of subgroups of $G$ which contains the $p$-subgroups, then
the natural map 
\begin{equation} \label{CEformula} H^*(G; A) \to \varprojlim_{\sOGF^{op}} H^*(H;
A)  \end{equation}
is an isomorphism; this is a restatement of the classical Cartan-Eilenberg
\emph{stable elements formula} \cite[Ch. XII, Thm. 10.1]{CEHA}. Using transfer
operations in group cohomology, one also can obtain a colimit decomposition of
$H^*(G; A)$ in terms of the cohomology of the $p$-subgroups of $G$.

The discussion above formally extends to the study of Mackey functors of $G$. A
\emph{Mackey functor} $M$ assigns an abelian group $M(H)$ to each subgroup $H\leq G$.
These abelian groups are related by induction, restriction, and
conjugation maps satisfying certain identities.  In the 
theory of Mackey functors, one aims to find the smallest family $\sF$ of
subgroups of $G$ for a given $M$ such that we can reconstruct $M(G)$ from
$M(H)$ as $H$ varies over $\sF$ as in Brauer's theorem \cite{Dre73b}. Such a
family is called the \emph{defect base of $M$}. 

Recall that Mackey functors naturally occur as the homotopy groups of (genuine)
$G$-spectra.  For
example, $R(G)$ is the zeroth homotopy group of the $G$-fixed point spectrum of
equivariant $K$-theory, \( R(G) \cong \pi_0^G \KG.\) 
Given a $G$-spectrum $M$ and a subgroup $H \leq G$, we associate the
$G$-spectra \[G/H_+ \wedge M \simeq F(G/H_+, M);\] we have
$\pi_0^G (  G/H_+\wedge M )\cong \pi_0^H M\cong \pi_0^G F(G/H_+, M)$. As $G/H$ varies over the orbit category
of $G$, the covariant (resp.~contravariant) functoriality of $G/H_+\wedge M$
(resp.~$F(G/H_+,M)$) gives the induction (resp.~ restriction) maps in the Mackey functor $\pi_0^{(-)} M$.

By taking \emph{homotopy} colimits and limits instead, we can obtain \emph{derived} analogues of the maps in \eqref{eq:ind-res} for a $G$-spectrum $M$:
\begin{equation}\label{eq:derived-ind-res}
  \hocolim_{\sOGF} G/H_+\wedge M \xrightarrow{\Ind_{\sF}^G} M \xrightarrow{\Res_{\sF}^G} \holim_{\sOGF^\op} F(G/H_+,M).
\end{equation}
Here the map $\Ind_{\sF}^G$ is the homotopy colimit of the maps
$G/H_+ \wedge M \to M$ obtained from the projections $G/H_+ \to S^0$ by
smashing with $M$. Similarly,
the map $\Res_{\sF}^G$ is the homotopy limit of the maps $M \to F(G/H_+, M)$
obtained from the projections $G/H_+ \to S^0$ by applying $F(\cdot, M)$. 

Note that the homotopy colimit $\hocolim_{\sOGF} G/H_+$ is the
suspension spectrum of the  classifying space $E \sF$ of the
family of subgroups $\sF$ (cf. \cref{sec:esf}); this is a $G$-space whose non-equivariant homotopy
type is contractible but whose equivariant homotopy type is more subtle. Thus, the induction map $\Ind_{\sF}^G$ in \eqref{eq:derived-ind-res} is 
a type of \emph{assembly map} for 
$M$ and the restriction map $\Res_{\sF}^G$ a type of co-assembly map. 

We can now ask when $\Ind_{\sF}^G$ and $\Res_{\sF}^G$ are equivalences of
$G$-spectra. Below
we will study a stronger condition on $M$, namely that it should be  \emph{$\sF$-nilpotent.} This will insure that not only are these maps equivalences, but also
that the corresponding homotopy colimit and limit spectral sequences collapse in
a strong way:
with a horizontal vanishing line at some finite stage. On homotopy groups, this will imply an analogue of
Artin's theorem (see \Cref{thm:gen-artin}). 

Now if $M=R$ is a homotopy commutative $G$-ring spectrum, then the restriction
maps are maps of ring spectra such that the lift $\Res_{\sF}^G$ is a ring
homomorphism, and we get a corresponding map of graded commutative rings after applying $\pi_*^G$. For example, if $R=\underline{H\bF_p}$ is the $G$-spectrum representing mod-$p$ Borel-equivariant cohomology, then we obtain a ring homomorphism
\[ \pi_{-*}^G\underline{H\bF_p}\cong H^*(BG;\bF_p)\xrightarrow{\Res_{\sF}^G}
\lim_{\sOGF^\op} H^*(BH;\bF_p).\] A celebrated result of Quillen
\cite[Thm.~7.1]{Qui71b} states that this map is a uniform $\cF_p$-isomorphism
when $\sF=\sEp$ is the family of elementary abelian $p$-subgroups of $G$,
i.e., subgroups of the form $C_p^{\times n}$ for some non-negative integer
$n$. Recall that a ring map $f\colon A\rightarrow B$ is a uniform
$\cF_{p}$-isomorphism if there are integers $m>0$ and $n\geq 0$ such that if
$x\in \ker f$ and $y\in B$ then $x^m=0$ and $y^{p^n}\in \Im f$. We will see
that $\underline{H\bF_p}$ is $\sEp$-nilpotent and that our collapse results
for the homotopy limit spectral sequence imply Quillen's theorem as well as
analogs for every homotopy commutative $\sF$-nilpotent $G$-ring spectrum (see \Cref{thm:gen-fiso}). 
\subsection{Main results}
 Throughout this paper, $G$ will denote a finite group and $\sF$ a family of
 subgroups of $G$. We will work with the
homotopy theory  of $G$-spectra. For our purposes, we will use the 
 stable presentable $\infty$-category of
 $G$-spectra $\SpG$ equipped with its symmetric monoidal smash
 product; see for instance 
 \cite[Sec. 5]{MNNa} for a brief account in this language. 
 
In most of this paper, the language of $\infty$-categories is used lightly; 
 if the reader prefers, they can recast our work in the setting of model
 category descriptions of $G$-spectra such as equivariant orthogonal spectra, as in \cite{MandellMay} or \cite{Man04}. 
In fact, the condition of $\sF$-nilpotence depends only on the \emph{homotopy}
category of $G$-spectra. 
The main translation would be that all limits and colimits 
occurring in this paper (in the $\infty$-categorical sense) need to be replaced
by homotopy limits and colimits in the respective model category, so all constructions are
appropriately derived. 
 However, one will still need a theory of $\infty$-categories, as developed in \cite{Lur09,Lur14}, for descent applications such as \cite[Thm.~\ref{S-decompFnilpmodules}]{MNNa}.

 The focus of this paper is the following subcategory of $G$-spectra. 
\begin{defn}[{Cf.~\cite[Def.~\ref{S-Fnildef}]{MNNa}}]\label{def:fnil}
	Let $\sFNil$, the $\infty$-category of $\sF$\emph{-nilpotent}  $G$-spectra, be the smallest thick $\otimes$-ideal in $\SpG$ containing $\{G/H_+\}_{H\in \sF}$. In other words, 
	$\sFNil$ is the smallest full subcategory of $\SpG$ such that:
	\begin{enumerate}
		\item For each subgroup $H\in \sF$, the suspension $G$-spectrum $G/H_+$ is $\sF$-nilpotent.
		\item For $E,F\in \SpG$ and $f\in \SpG(E,F)$, let $Cf$ denote the cofiber
		of $f$. If any two of $\{E,F,Cf\}$ are $\sF$-nilpotent, then all three of them are $\sF$-nilpotent.
		\item If $E\in \SpG$ is a retract of an $\sF$-nilpotent $G$-spectrum, then $E$ is $\sF$-nilpotent.
		\item If $E\in \SpG$ and $F$ is $\sF$-nilpotent, then $E\wedge F$ is $\sF$-nilpotent.
	\end{enumerate}

Let $\sF_1$ and $\sF_2$ be two families of subgroups of $G$ and let $\sF_1\cap
\sF_2$ denote their intersection. 
Then $\sF_1^{\mathrm{Nil}} \cap \sF_2^{\mathrm{Nil}} = (\sF_1 \cap
\sF_2)^{\mathrm{Nil}}$ by
\cite[Prop.~\ref{S-fnilintersect}]{MNNa}. 
For any $G$-spectrum $M$, there is thus  a {minimal} family $\sF$ such that $M$ is $\sF$-nilpotent; we will call this minimal family the \emph{derived defect base of $M$}.
\end{defn}

\begin{remark} 
The terminology ``derived defect base'' in Definition~\ref{def:fnil} is motivated
by the result that $G$-spectra can be viewed as 
\emph{spectral Mackey
functors} as in \cite{GuillouMay, SMackI, SMackII, Nardin}. 
In particular, the notion of derived defect base is an extension of the notion
of defect base from ordinary Mackey functors (valued in abelian groups) to
spectral ones. In particular, it does not refer to the use of ``derived'' techniques such as
derived functors of inverse limits, which is a standard technique in this
context, e.g., in the theory of homology decompositions.
\end{remark}

Although the above definition is simple, it is generally difficult to
determine
the derived defect base directly. We will provide several alternative
characterizations of $\sFNil$ shortly. First we recall some notation. 

For a real orthogonal representation $V$ of $G$, let $S^V=V\cup\{ \infty \}$ denote its
one-point compactification, considered as a pointed $G$-space with $\infty$
as basepoint. The inclusion $0\subset V$ induces an
equivariant map $e_V\colon S^0\rightarrow S^V$ called the \emph{Euler class} of $V$.
We consider in particular the case $V = \tilde{\rho}_G$,
the reduced regular representation of $G$. 

\begin{thmA}[See \Cref{thm:identify-fnil} and \Cref{thm:fnil-three-prime}] \label{thm:fnil-three}
	Let $M\in \SpG$. The following three conditions on $M$ are equivalent:
	\begin{enumerate}
		\item\label{it:fnil-def} The $G$-spectrum $M$ is $\sF$-nilpotent.
		\item\label{it:fnil-reg} For each subgroup $K\not\in \sF$,
		$e_{\tilde{\rho}_K}$ is a nilpotent endomorphism of the
		$K$-spectrum $\Res_K^GM$. In other words, there exists $n \geq 0$ such that the map $e_{n\tilde{\rho}_K}\simeq e_{\tilde{\rho}_K}^n$ is null-homotopic after smashing with $\Res_K^G M$. 

		\item\label{it:fnil-ss} The map of $G$-spectra \(
	\Res_\sF^G:	M\to \holim_{\sOGF^\op} F(G/H_+, M) \) of 
	\eqref{eq:derived-ind-res}	
		is an equivalence and  there is an integer $n\geq 0$ such that for every $G$-spectrum $X$, the $\sF$-homotopy limit spectral sequence:
		\[ E_2^{s,t}=\sideset{}{^s}\lim_{\sOGF^\op}\pi_{t}^H F(X,M) \Longrightarrow \pi_{t-s}^G F(X, \holim_{\sOGF^\op} F(G/H_+, M))\cong M_G^{s-t}(X)\]
		has a horizontal vanishing line of height $n$ on the $E_{n+1}$-page. In other words, we have $E_{n+1}^{k,*}=0$ for all $k>n$.  
		\end{enumerate}
\end{thmA}

\Cref{thm:fnil-three} is fundamental to this paper. 
Condition~\eqref{it:fnil-reg} is often easy to check in practice, especially
in the presence of \emph{Thom isomorphisms} for representation spheres (see
\Cref{sec:oriented} for some examples); these will lead to most of our
examples of $\sF$-nilpotence.

The equivalence between Conditions \eqref{it:fnil-def} and \eqref{it:fnil-ss}
has several computational consequences which we will now list.
\begin{thmA}[See \Cref{thm:gen-gen-artin}]\label{thm:gen-artin}
	Let $M$ and $X$ be $G$-spectra. Suppose that $M$ is $\sF$-nilpotent. Then each of the following maps
	\[ \colim_{\sOGF} M^*_H(X)\xrightarrow{\Ind_\sF^G} M^*_G(X)\xrightarrow{\Res_\sF^G} \lim_{\sOGF^\op} M^*_H(X)\]
	\[ \colim_{\sOGF} M_*^H(X)\xrightarrow{\Ind_\sF^G} M_*^G(X)\xrightarrow{\Res_\sF^G} \lim_{\sOGF^\op} M_*^H(X)\]
	becomes an isomorphism after inverting $|G|$. 
\end{thmA}

We next state our general analog of Quillen's $\cF_p$-isomorphism theorem. 
\begin{thmA}[See \Cref{thm:gen-gen-fiso}]\label{thm:gen-fiso}
	Let $R$ be a homotopy commutative $G$-ring spectrum and let $X$ be a $G$-space. Suppose that $R$ is $\sF$-nilpotent. Then the canonical map
	\[ R^*_G(X)\xrightarrow{\Res_\sF^G} \lim_{\sOGF^\op} R^*_H(X)\]
	is a uniform $\cN$-isomorphism\footnote{We believe this term was first coined in \cite[p.~88]{Hop87}.}: there are positive integers $m,n$ such that if $x\in \ker \Res_\sF^G$ and $y\in \lim_{\sOGF^\op} R^*_H(X)$ then $x^m=0$ and $y^n\in \Im\ \Res_{\sF}^G$. Moreover, after localizing at a prime $p$, $\Res_{\sF}^G$ is a uniform $\cF_{p}$-isomorphism.
\end{thmA}

Both 
Theorems~\ref{thm:gen-artin} 
and 
\ref{thm:gen-fiso} are consequences of 
the horizontal vanishing line and
a transfer argument which implies that the elements in positive filtration degree in the hocolim and holim spectral sequences are $|G|$-torsion.

\begin{cor}[Compare \Cref{prop:spec-homeo}] \label{cor:varieties}
	Under the hypotheses of \Cref{thm:gen-fiso}, the map 
	of commutative rings
	\(\Res_\sF^G\colon R^0_G(X)\rightarrow \lim_{\sOGF^\op} R^0_H(X) \)
	 induces a homeomorphism between the associated Zariski spaces\footnote{Under additional finiteness hypotheses (see \Cref{prop:stratification}), there is a further identification: $\colim_{\sOGF}\Spec( R^0_H(X))\cong \Spec(\lim_{\sOGF^\op} R^0_H(X))$.}:
	 \[ \Spec\left(\lim_{\sOGF^\op} R^0_H(X)\right)\to \Spec(R^0_G(X)).\]
\end{cor}

For $M\in\sF^{\mathrm{Nil}}$, the minimal integer $n$ satisfying
\Cref{thm:fnil-three}\eqref{it:fnil-ss} is called the \emph{$\sF$-exponent} of $M$.
We include various characterizations of this numerical invariant below. 

We also prove an analog of \Cref{thm:gen-artin} and \Cref{thm:gen-fiso} which 
involves an \emph{end} rather than an inverse limit over $\sOGF^{op}$.
This recovers the original formulas of  Quillen \cite{Qui71b} and Hopkins-Kuhn-Ravenel
\cite{HKR00}, and is nontrivial even when $\sF$ is the family of all subgroups. For $H \leq G$ and $M$ a $G$-spectrum, recall that we write $M^H$ to denote the
$H$-fixed point spectrum of $M$, i.e., the spectrum of equivariant maps $G/H_+ \to M$. 

\begin{thmA}[See \Cref{endformula}] 
\label{thm:thmendinentro}
Let $R$ be a homotopy commutative $G$-ring spectrum and 
$X$ a finite $G$-CW complex. Assume that $R$ is $\sF$-nilpotent.
Then the natural map 
\[ \phi_{\sF} \colon R_G^*(X) \to  \int_{\sOGF^{op}} (R^H)^*( X^H)  \]
has the following two properties: 
\begin{enumerate}
\item  $\phi_{\sF} \otimes_{\mathbb{Z}} \mathbb{Z}[1/|G|]$ is an isomorphism.
\item 
The map $\phi_{\sF}$ is a uniform $\cN$-isomorphism 
 and 
 for any prime number $p$, $(\phi_{\sF})_{(p)}$ is a uniform
 $\cF_p$-isomorphism.
\end{enumerate}
\end{thmA} 

In \cite[Theorem 6.2]{Qui71b}, rather than
assuming that $X$ is a finite $G$-CW complex, Quillen assumes more generally that $X$ is
compact. In addition, in the end diagram, Quillen replaces $X^H$ with the
discrete space $\pi_0(X^H)$; since Quillen works with mod $p$ cohomology this
does not change (1) or (2) above.

We can identify the derived defect bases for many $G$-equivariant ring spectra
of interest. These are listed in \Cref{fig:examples}, where we set the notation
for the relevant families of subgroups in \Cref{fig:families}. Many of these
examples arise from non-equivariant ring spectra by taking their associated
Borel theories as in \cite[\S 6.3]{MNNa}. There, as above, we are letting $\bb{M}$ denote the \emph{Borel-equivariant $G$-spectrum associated to a spectrum} $M$ with a $G$-action. All of the examples above come from spectra with trivial $G$-actions, in which case this equivariant cohomology theory theory is defined so that, for a $G$-spectrum $X$,
\[ \bb{M}^*_G(X)=M^*(EG_+\wedge_G X).\] 

In  \Cref{fig:examples}, note first that if $R\in \SpG$ is $\sF$-nilpotent, then its Borel completion $\bb{R}$ is automatically $\bb{\sF}$-nilpotent (i.e., we only need to consider the $p$-groups in $\sF$, as $p$ varies over the primes dividing $|G|$). The notation respects localization in the following sense: if $\bb{R}$ is $\bb{\sF}$-nilpotent, then $\bb{R_{(p)}}$ (resp.\ $\bb{R[1/n]}$) will automatically be $\sF_{(p)}=\bb{\sF}_{(p)}$-nilpotent (resp.\ $\bb{\sF}[1/n]$-nilpotent).  These results are immediate consequences of \Cref{prop:borel-sphere} and allow one to determine derived defect bases for Borel-equivariant $G$-spectra via arithmetic fracture square arguments. 

Finally, we demonstrate a connection (displayed in \Cref{fig:examples})
between the `chromatic complexity' of a $G$-spectrum $E$ and the complexity of
$E$'s derived defect base. More precisely, we show in \Cref{prop:ln-local} that
if a spectrum $E$ is $L_n$-local, then the Borel spectrum $\bb{E}$ is $\sAbpn$-nilpotent. This
result relies on 
the `character theory' of Hopkins-Kuhn-Ravenel
\cite{HKR00} and the Hopkins-Ravenel smash product theorem. 


%

\begin{figure}
\centering 
\begin{subfigure}{0.4\textwidth}
\centering
\scalebox{0.8}{
\tabulinesep = 5pt
\begin{tabu}{ X[1, c]  X[3, c] }
\hline
Notation & Definition of family\\
\hline
 $\sAll$ & All subgroups \\
$\sP$ & Proper subgroups \\ 
$\sTriv$ & Only the trivial subgroup \\ 
$\sAb$ & Abelian subgroups\\
$\sAbn$ & Abelian subgroups which can be generated by $n$ elements \\
$\sC=\sAbone$ & Cyclic subgroups\\
$\sE$ & Subgroups of the form $C_p^{\times n}$ for some prime $p$ and some $n$\\
$\sF(K)$ & Subgroups in $\sF$ which are subconjugate to $K\leq G$\\
$\bb{\sF}$ & Subgroups $H$ in $\sF$ such that $|H|=p^n$ for some prime $p$ and some $n$\\
$\sF_{(p)}$ & Subgroups $H$ in $\sF$ such that $|H|=p^n$ for some $n$ \\
$\sF[1/n]$ & Subgroups $H$ in $\sF$ such that $n\nmid |H|$ \\ 
$\sF_1\cup \sF_2$ & Subgroups $H$ in either $\sF_1$ or $\sF_2$ \\
\hline
\end{tabu}
}
	\subcaption{\label{fig:families} Families of subgroups.}
\end{subfigure}%
\begin{subfigure}{0.55\textwidth}
\centering
\scalebox{0.8}{
\renewcommand{\arraystretch}{1.2}
\begin{tabu}{ c  c  c }
\hline
$G$-Spectrum $R$ & Derived defect base & Proof of claim\\
\hline
$S, S\otimes \bQ$ & $\sAll$ & \Cref{prop:sphere}\\
$K\bR\ (G=C_2)$ & $\sTriv$  & \Cref{prop:kr-defect-base}\\
$\MUR\ (G=C_2)$ & $\sAll$  & \Cref{prop:do-not-descend}\\
$\MUG, \MOG$ & $\sAll$ & \Cref{prop:do-not-descend}\\
$H\mathbb{Z}$  & $\bb{\sAll}$ & \Cref{prop:constant-examples}\\
$H\mathbb{Q}$  & $\sTriv$ & \Cref{prop:constant-examples}\\
$\KOG,\KG$ & $\sC$ &  \Cref{prop:kg-is-in-cnil}\\
$\koG,\kG$ & $\sC\cup \sE$ & \Cref{prop:greenlees-kg-nil} \\
\hline
$\bb{S}$ & $\bb{\sAll}$ & \Cref{prop:borel-sphere}\\ 
$\bb{S\otimes \bQ}$ & $\sTriv$ & \Cref{prop:borel-sphere}\\ 
$\bb{MU}$ & $\bb{\sAb}$ &  \Cref{prop:mu-is-anil}\\
$\bb{H\bF_p}$ & $\sEp$ &  \Cref{prop:hfp-is-epnil}\\
$\bb{H\bZ}$ & $\sE$ &  \Cref{prop:HZ-is-enil} \\
$\ku$ & $\sE\cup \bb{\sC}$ & \Cref{prop:ku-is-in-secnil}\\
$\bb{BP\langle n\rangle}$ &  $\sEp \cup \sAbpn$ & \Cref{prop:BPn}\\
$\Kn$ & $\sTriv$ & \Cref{prop:kn-defect-base}\\
$\bb{T(n)}$ & $\sTriv$ & \Cref{prop:kn-defect-base}\\
$\bb{E_n}$ &  $\sAbpn$ & \Cref{prop:JW-theory}\\
$\bb{L_n S}$ & $\sAbpn$ &  \Cref{prop:ln-local}\\ 
$\ko$ & $\sE\cup \bb{\sC}$ & \Cref{prop:KO-is-in-cnil}\\
$\bb{KO},\bb{KU}$ & $\bb{\sC}$ & \Cref{prop:KO-is-in-cnil}\\
$\Tmf, \TMF$ &  $\bb{\sAb^2}$ & \Cref{prop:non-connective-variants-of-tmf}\\
$\tmf$  & $\sE\cup \bb{\sAb^2}$ &  \Cref{prop:connective-tmf}\\
$\MO$ &  $\sE_{(2)}$ & \Cref{cor:mo-is-e2-nil}\\
$\MSO$ &  $\sE_{(2)} \cup \bb{\sAb}[1/2]$ & \Cref{prop:mso-nilpotence}\\
$\bb{\MSp[1/2]}$ & $\bb{\sAb}[1/2]$ & \Cref{cor:msp-nilpotence}\\
$\MSpin$ &  $\sE_{(2)} \cup \sC_{(2)} \cup \bb{\sAb}[1/2]$ & \Cref{prop:spin-bordism}\\
$\bb{MO\langle n\rangle [1/2]} \ (n\geq 2)$ & $\bb{\sAb}[1/2]$ & \Cref{prop:MOn}\\
$\bb{MU\langle n\rangle}$ &  $\bb{\sAb}$ & \Cref{prop:MUn}\\
\noalign{\smallskip}
\hline
\end{tabu}
}
	\subcaption{\label{fig:examples} Derived defect bases for some $G$-ring spectra.}
\end{subfigure}
\caption[]{}
\end{figure}


\subsection{Related work}
\label{relatedwork}
There is a large body of work 
around questions of recovering equivariant cohomology theories from
suitable subgroups; we summarize some of it below. 

\newcommand{\scg}{\mathscr{O}_{\cC}(G)}
In this paper, we only consider \emph{families} of subgroups. 
One can instead work more generally with \emph{collections} $\cC$ of subgroups of a finite
group $G$, which by
definition are only required to be closed under conjugation. The question of
decomposing homology and cohomology in terms of collections has been extensively
studied starting with \cite{Dw97, Dw98}. 
Given a collection $\cC$, one defines the $\cC$-orbit category $\scg$
analogously and one has a 
map of $G$-spaces
\begin{equation} \label{collmap} \hocolim_{G/H \in \scg} G/H \to \ast.
\end{equation}
The collection $\cC$ is said to be \emph{ample} if the induced map on
$G$-homotopy orbits induces an equivalence after applying singular chains, i.e., 
\begin{equation} \label{collmap2} \hocolim_{G/H \in \scg} C_*(BH;
\mathbb{F}_p) \to C_*(BG; \mathbb{F}_p);   \end{equation}
note that when $\cC$ is a family 
then
both
\eqref{collmap}, \eqref{collmap2} are automatically equivalences. 
For an ample collection $\cC$, we obtain colimit spectral sequences for the homology of
$BG$ from \eqref{collmap2}, and  $\cC$
is said to be \emph{subgroup-sharp} if it collapses at $E_2$ on the zero-line.
In particular, in this case one obtains an exact description of the homology (or
cohomology) of $G$ in terms of the homology of $H \in \cC$, i.e.,
$\varinjlim_{G/H \in \scg} H_*(BH; \mathbb{F}_p) \simeq H_*(BG; \mathbb{F}_p)$. 

The collection of $p$-subgroups is subgroup-sharp, essentially by the
Cartan-Eilenberg stable elements formula \eqref{CEformula}. 
There are many examples of 
subgroup-sharp collections $\cC$ which are strictly contained in the collection
of $p$-subgroups. 
These ideas originated in \cite{JM92, Dw97, Dw98} and have since extended further and
improved; see
\cite{GS06, Gro02}. 

Our setting differs from the theory of homology decompositions in the following
ways. 
\begin{enumerate}
\item  
First, we work only with families (rather than collections) of
subgroups. Thus, there is no analog of the condition of ampleness.
\item
In the setting of 
\emph{sharp} homology decompositions, the colimit spectral sequences (as in \eqref{collmap2}
or variants) collapses at $E_2$ at
the zero-line.
Therefore, one obtains precise decompositions of the homology or cohomology of
$BG$. 
Sometimes one also considers more general settings (see \cite[Theorem 1.1 and
Remark 3.11]{Gro02}) where one has a horizontal vanishing line at $E_2$. 

In our setting, by contrast, the limit and colimit spectral sequences are often
very infinite at $E_2$ (see Appendix B for an example), but are only required to collapse at some finite stage. For this reason, at
the level of equivariant homology and cohomology, we do not obtain
exact
decompositions, but rather 
$\cN$-isomorphisms.
This is  a fundamental feature of our setup. 
\item 
 Finally, the theory of homology decompositions usually
relates $H^*(BG; \mathbb{F}_p)$ to the cohomology of various $p$-subgroups of
$G$, thereby providing strong refinements of the Cartan-Eilenberg stable
elements formula \eqref{CEformula}. By contrast, our results apply  when $G$ is a  
$p$-group (in fact, for Borel-equivariant theories, all questions can be reduced
to ones involving $p$-groups thanks to \Cref{prop:borel-sphere} below).

\end{enumerate}

In particular, we emphasize that many of the ideas that occur in this paper
(such as the use of the homotopy limit and colimit spectral sequences) are far
from new in this context. The main idea we use that is new here (although not in
other contexts, such as chromatic homotopy theory) is the theory of \emph{nilpotence}
(which we discuss at length in the companion paper \cite{MNNa} in an axiomatic
setting).

Many other authors have considered the setting of families of subgroups, and 
for more general equivariant homology theories. In particular, 
results similar to \Cref{thm:gen-artin,thm:gen-fiso,thm:thmendinentro} have been established by various authors:
\begin{itemize}

	\item  Segal proves the analog of \Cref{thm:gen-fiso} for $G$-equivariant
	$K$-theory for a general compact Lie group when $X$ is a point and $\sF$ is
	the family of topologically cyclic subgroups with finite Weyl groups
	\cite[Prop.~3.5]{Seg68b}. Segal also proves an analog of Brauer's theorem in this setting \cite[Prop.~3.11]{Seg68b}. 

	\item The most celebrated form of \Cref{thm:gen-fiso}
and \Cref{thm:thmendinentro}	
	is \cite[Thm.~6.2]{Qui71b}. There, Quillen proves this result in the case $M=\bb{H\bF_p}$, $\sF=\sEp$,  $G$ is a compact Lie group, and $X$ is a $G$-CW complex of finite mod-$p$ cohomological dimension. 
	In the case $X=*$, Quillen also proves this result in the case $G$ is a
	compact Hausdorff topological group with only  finitely many conjugacy classes of
	elementary abelian subgroups \cite[Prop.~13.4]{Qui71b}, along with  an extension to the case where $G$ is a discrete subgroup with a finite index subgroup $H$ of finite mod-$p$ cohomological dimension \cite[Thm.~14.1]{Qui71b}.  

	Quillen's seminal work underlies all of the following research in this direction including our own. This paper owes a tremendous debt to him. 

	\item Bojanowska and Jackowski prove \Cref{thm:gen-fiso} in the case $M=\KG$, $\sF=\sC$, and $X$ is a finite $G$-CW complex. They also prove that the homotopy limit spectral sequence has the desired abutment \cite{BoJ80}. 

	\item Greenlees and Strickland prove a result similar to 
\Cref{thm:thmendinentro} and 	
	\Cref{cor:varieties} in the case that $M=\bb{E}$, $E$ is a complex oriented ring spectrum with formal properties similar to $E_n$, $X$ is a finite $G$-CW complex, and $\sF=\sAbpn$ \cite[Thm.~3.5]{GS99}.  They also obtain suitable extensions when $G$ is a compact Lie group \cite[App.~C]{GS99}. 

	\item Hopkins, Kuhn, and Ravenel prove \Cref{thm:gen-artin}
and the $\mathbb{Z}[1/|G|]$-local part of 
\Cref{thm:thmendinentro}
	in the case where  $M=\bb{E}$, $E$ is a complex oriented ring spectrum, $\sF=\bb{\sAb}$, and $X$ is a finite $G$-CW complex \cite[Thm.~A and Rem.~3.5]{HKR00}.

	\item In \cite{Fau08} Fausk shows that \cite[Thm.~A]{HKR00} can be generalized in several ways if one makes some additional assumptions. First, Fausk proves the analogue of \Cref{thm:gen-artin} when $M=\KG$ and
	$G$ is a compact Lie group. Moreover, Fausk proves \Cref{thm:gen-artin} when
	$M=\bb{E_n}$ (or a closely related ring spectrum), $G$ is a finite group, $\sF=\sAbpn$, and $\pi_*^G M$ is torsion-free (e.g., when $G$ is a good group in the sense of \cite{HKR00}). Fausk also obtains generalized Brauer induction theorems in these contexts. Fausk's results do not require a finiteness assumption on $X$.
\end{itemize}



\subsection*{Organization}\label{sec:organization}
In \Cref{sec:thick}, we will analyze the class $\sFNil$ of $G$-spectra and
prove \Cref{thm:fnil-three}. We break this proof into two parts. In
\Cref{sec:fnil-euler}, we prove the equivalence of Conditions
\eqref{it:fnil-def} and \eqref{it:fnil-reg} of \Cref{thm:fnil-three}
(\Cref{thm:identify-fnil}) as well as some immediate consequences. In
\Cref{sec:fnil-holim}, we prove the equivalence of Conditions
\eqref{it:fnil-def} and \eqref{it:fnil-ss} (\Cref{thm:fnil-three-prime}).

In \Cref{sec:ss}, we will analyze the homotopy colimit and homotopy limit
spectral sequences. This will lead to proofs of \Cref{thm:gen-artin},
\Cref{thm:gen-fiso}, and \Cref{cor:varieties} in \Cref{sec:calc-thms}.  Along the way we will prove \Cref{prop:stratification}, which is the appropriate analogue of Quillen's stratification theorem \cite[Thm.~8.10]{Qui71b} in this context. 
In \Cref{sec:endthm}, we will prove \Cref{thm:thmendinentro}, which will
require some additional work.


In \Cref{sec:split} we show that derived induction and restriction theory
generalizes classical induction and restriction theory and reduces to it exactly for 
$\sF$-nilpotent spectra of exponent at most one. 
We show that one can use the calculation of the derived defect base of a $G$-ring spectrum to put an upper bound on its defect base (\Cref{prop:hyper-induction}). As applications, we obtain a generalized hyperelementary induction theorem similar to Brauer's theorem (\Cref{thm:split-brauer}) and triangulated descent results in the sense of Balmer (\Cref{prop:balmer}).

In the last two sections of the main body of hte text, we prove all of the remaining claims
in \Cref{fig:examples}. In \Cref{sec:oriented}, we show how the existence of
Thom isomorphisms can be used to show a $G$-ring spectrum is $\sF$-nilpotent.
We then combine these results with non-equivariant thick subcategory arguments to determine the derived defect bases of  the remaining examples.

In the appendices we gather several auxiliary results for working with the $\sF$-homotopy limit spectral sequences and work through a nontrivial example for equivariant topological $K$-theory.

\subsection*{Acknowledgments} 
The authors would like to thank John Greenlees, Jesper Grodal, Hans-Werner Henn, Mike Hopkins, Peter May, 
Charles Rezk, and Nat Stapleton for helpful comments related to this project. We would also like to thank Koen van Woerden for proofreading an earlier draft of this paper. Finally we would like to thank the referee for providing helpful remarks on this paper.

\subsection*{Conventions}
Throughout this paper, $G$ will denote a finite group and $\sF$ a family of subgroups of $G$.
For two $G$-spectra $X$ and $Y$ we will let $F(X,Y)\in \SpG$ denote the internal function $G$-spectrum. Unless we believe it to be helpful to  the reader, we will generally suppress the functors $\Sigma^\infty$ and $\Res_K^G$ from our notation.


A $G$-ring spectrum $R$ will always be a $G$-spectrum equipped with a homotopy associative and unital multiplication, i.e., an associative algebra in $\ho(\SpG)$. We will say that $R$ is homotopy commutative if the multiplication is commutative in $\ho(\SpG)$. An $R$-module will be an object of $\ho(\SpG)$ equipped with a left action of $R$ satisfying the standard associativity and unit conditions. 
We will use the adjective \emph{structured} when we want to talk about the 
$\infty$-categorical or model categorical notion of module. 


\section{The thick $\otimes$-ideal $\sFNil$}\label{sec:thick}
In this section, we give the main characterizations of $\sF$-nilpotence and
prove \Cref{thm:fnil-three}. 
\subsection{The characterization of $\sFNil$ in terms of Euler classes}\label{sec:fnil-euler}\ \\

In this subsection, we will prove the equivalence of Conditions
\eqref{it:fnil-def} and \eqref{it:fnil-reg} from \Cref{thm:fnil-three} in
\Cref{thm:identify-fnil} below. First, we will require some elementary properties of representation spheres. 

\begin{defn}
	For a finite-dimensional
orthogonal representation $V$ of $G$, we let
$nV$ denote $V^{\oplus n}$.
We let  $S(V)$ denote the unit $G$-sphere of $V$ and 
		$S^V$ denote the pointed $G$-space obtained as the one-point compactification of $V$, where we take the point at $\infty$ to be the basepoint. 
Finally, we let $e_V$, the \emph{Euler class} of $V$, denote the pointed $G$-map
		\[ e_V\colon S^0\rightarrow S^V\]
		induced by the inclusion $0\rightarrow V$.
\end{defn}

We now recall the following standard results.
\begin{prop}\label{prop:props-of-reps}
	Let $V$ be a finite-dimensional orthogonal representation of $G$. Then: 
	\begin{enumerate}
		\item If $V$ contains a trivial summand, then $e_V$ is $G$-equivariantly homotopic to the trivial map.
		\item The $G$-space $S^V$ is the cofiber of the nontrivial map $S(V)_+\rightarrow S^0$.
		\item\label{it:isotropy} The $G$-space $S(V)$ admits a finite $G$-CW
		structure constructed from  cells of the form \[G/H\times S^{n-1}\rightarrow
		G/H\times D^{n}\] where $H$ is a subgroup such that $V^H\neq\{ 0\}$
		and $n<\dim V^H$. Compare \cite[Exer.~ II.1..~ II.1.10]{Die87a}
		and \cite{Ill83}.
		\item For every $n\ge 0$, we have $e_V^{n}\simeq e_{nV}$.
	\end{enumerate}
\end{prop}

We now prove the main characterization of $\sF$-nilpotence (see \Cref{def:fnil}) in terms of Euler
classes. 
 \begin{thm}\label{thm:identify-fnil}
 	A $G$-spectrum $M$ is $\sF$-nilpotent 
if and only if, for all subgroups $K \leq G$ with $K \notin \sF$,
there exists an integer $n$
such that the Euler class $e_{n\tilde{\rho}_K}\colon S^0\rightarrow S^{n\tilde{\rho}_K}$ is null-homotopic after smashing with $\Res_K^G M$. 
	\end{thm}
 \begin{proof}
	Let $\sFNilPr\subseteq \SpG$ denote the full subcategory spanned by the
	$M\in \SpG$ satisfying the Euler class condition of the theorem. 
	It is easy to see that 
	$\sFNilPr$ is a thick $\otimes$-ideal.
 	We need to show that $\sFNil = \sFNilPr$. 

	For a subgroup $H \leq G$, let $\mathcal{P}_H$ denote the family of
	proper subgroups of $H$.
Observe that $M \in \sFNilPr$ if and only if, for every 	$H \leq G$
not in $\sF$, we have $\Res^G_H M \in \mathcal{P}_H^{\mathrm{Nil}'}$.
Moreover, one has a similar statement for $\sF$-nilpotence: by
\cite[Prop.~\ref{S-reducetopropersubgroups}]{MNNa}, $M \in \sFNil$ if and only if
for every subgroup $H \notin \sF$, $\Res^G_H M \in
\mathcal{P}_H^{\mathrm{Nil}}$.

It thus suffices to consider the case where $\sF = \mathcal{P}_G$. 
In other words, we need to show that the thick $\otimes$-ideal generated by 
$\left\{G/H_+\right\}_{H < G}$ is equal to $\mathcal{P}_G^{\mathrm{Nil}'}$.
Observe first that the Euler class $e_{\widetilde{\rho}_G}$ becomes
null-homotopic after smashing with $G/H_+$ for any $H < G$. This follows
because for any $H < G$, $\Res^G_H e_{\widetilde{\rho}_G}$ is null-homotopic as
the $H$-representation $\Res^G_H \widetilde{\rho}_G$ contains a trivial summand.
Here we use the relationship between smashing with $G/H_+$ and restricting to
$H$-spectra; compare \cite[Thm.~1.1]{BAS} and \cite[Thm.~\ref{S-restensoring}]{MNNa}. Therefore, we get $G/H_+ \in
\mathcal{P}_G^{\mathrm{Nil}'}$, so that $\mathcal{P}_G^{\mathrm{Nil}} \subset
\mathcal{P}_G^{\mathrm{Nil}'}$. 

We now prove the opposite inclusion. Suppose $M \in
\mathcal{P}_G^{\mathrm{Nil}'}$. Then there exists $n$ such that $\mathrm{id}_M \wedge e_{n
\widetilde{\rho}_G}$ is null-homotopic, and the cofiber sequence
\[ S(n \widetilde{ \rho}_G)_+ \wedge M \to M  
\xrightarrow{\mathrm{id}_M \wedge e_{n
\widetilde{\rho}_G}} 
M   \wedge S^{n \widetilde{\rho}_G}  \]
shows that $M$ is a retract of 
$S(n \widetilde{ \rho}_G)_+ \wedge M$. Since $\widetilde{\rho }_G$ has no non-trivial fixed
points, 
$S(n \widetilde{ \rho}_G)_+  \in \mathcal{P}_G^{\mathrm{Nil}}$ in view of the
cell decomposition 
given in \Cref{prop:props-of-reps}.
Therefore, 
$S(n \widetilde{ \rho}_G)_+ \wedge M$ is $\mathcal{P}_G$-nilpotent, and thus its retract
$M$ is too. 
 \end{proof}
\begin{remark}\label{rem:endomorphism-rings}
	If we regard $e_{\tilde{\rho}_K}$ as an element of the `$RO(K)$-graded homotopy groups' \cite[\S 6]{Ada84}, $\pi_\star^K S$, of the sphere spectrum, then after smashing $e_{\tilde{\rho}_K}$ with $M$ we obtain an element in $\pi_\star^K F(M,M)$, the `$RO(K)$-graded homotopy groups' of the endomorphism ring of $M$. This element can also be identified with the image of $e_{\tilde{\rho}_K}$ under the unit map $S\rightarrow F(M,M)$. 

	Identifying $e_{\tilde{\rho}_K}$ with its image, we can now restate the null-homotopy condition of \Cref{thm:identify-fnil} in either of the following equivalent ways:
	\begin{enumerate}
		\item $e_{\tilde{\rho}_K}\in \pi_\star^K F(M,M)$ is nilpotent, or 
		\item $F(M,M)[e_{\tilde{\rho}_K}^{-1}]\simeq *\in \SpK$.
	\end{enumerate}

	While $M\in\sFNil$ implies that $M[e_{\tilde{\rho_K}}^{-1}]$ is
	contractible for each $K\not \in \sF$, the converse does not hold. The
	contractibility of $M[e_{\tilde{\rho}_K}^{-1}]$ is equivalent to knowing
	that every element $x\in \pi_\star^K M$ is annihilated by \emph{some}
	power, possibly depending on $x$, of $e_{\tilde{\rho}_K}$. The condition
	$M\in\sFNil$ tells us that there is a \emph{fixed} power of
	$e_{\tilde{\rho}_K}$ which annihilates all of  $\pi_\star^K M$. 

	On the other hand, when $M=R$ is a $G$-ring spectrum, the two conditions are equivalent because the power of $e_{\tilde{\rho}_K}$ annihilating $1\in \pi_*^K R$ annihilates all of $\pi_*^K R$. 

\end{remark}
\begin{cor}\label{prop:ring-criteria-fnil}
	Suppose that $R$ is a $G$-ring spectrum. Then the following are equivalent:
	\begin{enumerate}
		\item\label{it:prop-ring-1} The $G$-spectrum $R$ is $\sF$-nilpotent.
		\item\label{it:prop-ring-2} For each subgroup $H\not\in \sF$, the image of $e_{\tilde\rho_H}\in \pi_{\star}^H S$ under the unit map $S\rightarrow R$ is nilpotent.
	\end{enumerate}
\end{cor}

\subsection{The $\sF$-homotopy limits and colimits}\label{sec:apps-of-first-equivalence}
We will now  precisely define the homotopy colimits and limits mentioned in
the introduction in \eqref{eq:derived-ind-res} and prove they are
equivalences when $M$ is $\sF$-nilpotent. 

We denote the category of $G$-spaces by $\Top_G$. As usual, let $\sOG\subseteq
\Top_G$ (the \emph{orbit category}) denote the full subcategory of $\Top_G$ 
spanned by the transitive $G$-sets. To a family $\sF$ we associate the full
subcategory $\sOGF \subset \sOG$  spanned by the transitive $G$-sets whose isotropy lies in $\sF$.

Let $i\colon \sOGF\rightarrow \Top_G$ denote the inclusion. We associate to
$\sF$ a $G$-space $E\sF:=\hocolim_{\sOGF} i$. We also define a pointed
$G$-space $\wt{E}\sF$ as the homotopy cofiber of the unique nontrivial map $E\sF_+\rightarrow S^0$. These $G$-spaces are determined up to canonical equivalence by the following properties (see \Cref{sec:esf} and \cite[Def.~II.2.10]{LMS86}):
\begin{equation}
E\sF^K  \simeq 
	\begin{cases} 
		* & \mbox{if } K\in \sF\\
		\emptyset & \mbox{otherwise}
	\end{cases}\label{eq:EF-univ-prop}
\quad \quad \quad \wt{E}\sF^K  \simeq 
	\begin{cases} 
		* & \mbox{if } K\in \sF\\
		S^0 & \mbox{otherwise.}
	\end{cases}	
\end{equation}

For the family $\sP$ of all proper subgroups of $G$, these spaces 
admit a particularly simple construction.

\begin{prop}\label{prop:model-for-proper-subgroups}
	There are canonical equivalences  \[E\sP\simeq \hocolim_n
	S(n\tilde{\rho}_G) \quad  \mathrm{and} \quad \wt{E}\sP\simeq \hocolim_n S^{n\tilde{\rho}_G}\simeq S[e_{\tilde{\rho}_G}^{-1}].\] Here the homotopy colimits are indexed over the maps induced by the inclusions $n\tilde{\rho}_G\rightarrow (n+1)\tilde{\rho}_G$.
\end{prop}
\begin{proof}
	We just need to check that the homotopy colimits have the correct fixed
	points. Since fixed points commute with homotopy colimits, this follows from \Cref{prop:props-of-reps} and the following observation: $\tilde{\rho}_G^K$ is 0-dimensional if and only if $K=G$.
\end{proof}

We recall the significance of the objects $E \sF_+$ and $ \widetilde{E} {\sF}$ in the
general theory, cf. \cite[\S
\ref{S-ss:familiespaper1}]{MNNa}.
Let $\mathrm{Loc}_{\sF}$ denote the \emph{localizing} subcategory of $\GSpec$
generated by the $\left\{G/H_+\right\}_{H \in \sF}$. It is equivalently
the localizing tensor-ideal
generated by 
the commutative algebra object
\( A_{\sF} := \prod_{H \in \sF} F(G/H_+, S) \in \GSpec, \)
which we call \emph{$A_{\sF}$-torsion objects} in  \cite[Def~\ref{S-def:torsionobj}]{MNNa}
as it extends ideas of \cite{DwG02} in the case of module categories. 
The inclusion $\mathrm{Loc}_{\sF} \subset \GSpec$ admits a right adjoint given
by \emph{$\sF$-colocalization} \cite[Construction 3.2]{MNNa}; the right adjoint is given explicitly by $X
\mapsto E \sF_+ \wedge X$. In particular, $X \in \mathrm{Loc}_{\sF}$ if and
only if 
the natural map
\[  E \sF_+ \wedge X \to X \]
is an equivalence. We also have the subcategory of \emph{$\sF$-complete}
$G$-spectra, i.e., those $G$-spectra complete 
with respect to the
algebra object $A_{\sF}$ \cite[Sec.~\ref{S-sec:axiomatic}]{MNNa}. The
$G$-space $E\sF$ also controls the theory of $\sF$-completeness: a
$G$-spectrum $X$ is $\sF$-complete if and only if the natural map
\[ X \to F(E\sF_+, X)   \]
is an equivalence.

We consider finally (cf.\ \cite[Sec.~\ref{S-sec:A-1local}]{MNNa})
the subcategory $\GSpec[\sF^{-1}]$ of those $G$-spectra $Y$ such that $F(X, Y)
\simeq \ast$ for any $X \in \mathrm{Loc}_{\sF}$. 
Then $\GSpec[\sF^{-1}]$ is a \emph{localization} of $\GSpec$, and the
localization is given by the functor $X \mapsto \widetilde{E} \sF \wedge X$. 
The localization functor annihilates \emph{precisely} the localizing
subcategory $\mathrm{Loc}_{\sF}$. 
Note that, by definition \cite[Def.\  \ref{S-Fnildef}]{MNNa}, a $G$-spectrum is $\sF$-nilpotent if and
only if it is $A_{\sF}$-nilpotent.  

Using the general theory of torsion, complete, and nilpotent objects with
respect to a \emph{dualizable} algebra object, we now record the  following
list of properties of $\sFNil$. 
\begin{prop}\label{prop:Fnil-implies-acyclic}
\label{prop:genlist}
\label{prop:end-m}
\begin{enumerate}
\item 
	If $M$ is an $\sF$-nilpotent $G$-spectrum, then $\wt{E}\sF\wedge M$ is
	contractible, and thus the map $M \wedge E \sF_+ \to M$ is an equivalence. 
	Similarly, the map $M \to F(E\sF_+, M)$ is an equivalence.
	\item 
If $M$ is  a $G$-ring spectrum with $\widetilde{E} \sF \wedge M$
	contractible, 
	then $M$ is $\sF$-nilpotent.
	\item 
	Let $X$ and $M$ be $G$-spectra. If $M$ is $\sF$-nilpotent, then so is $F(X,M)$.
\item A $G$-spectrum $M$ is $\sF$-nilpotent if and only if the endomorphism $G$-ring spectrum $F(M,M)$ is $\sF$-nilpotent.
	\end{enumerate}
\end{prop}
\begin{proof}
As above, a $G$-spectrum $M$
belongs to the localizing subcategory $\mathrm{Loc}_{\sF}$ generated by the $\left\{G/H_+\right\}_{H
\in \sF}$ if and only if $M \wedge \widetilde{E} \sF$ is contractible (or
equivalently if $ M \wedge E\sF_+ \simeq M$). If $M$
is $\sF$-nilpotent, this is certainly the case. 
If $M \in \sFNil$, then $M$ is also complete with respect to the algebra object
$A_{\sF}$ so that the $\sF$-completion map $M \to F(E\sF_+, M)$ is an equivalence. 

Conversely, if $M$ is a $G$-ring spectrum, then the $\sF^{-1}$-localization of
$M$, i.e., $\widetilde{E} \sF \wedge M$, vanishes  if and only $M $ is $\sF$-nilpotent by
\cite[Thm.~\ref{S-ringobjectnilp}]{MNNa}.

We refer to \cite[Cor.~\ref{S-nil:closedundercotensor}]{MNNa} for the (general)
argument that $\sFNil$ is closed under cotensors. If $M \in \GSpec$ and $F(M,
M) \in \sFNil$, then $M$, as a module over $F(M, M)$, also belongs to $\sFNil$.
This verifies the third and fourth claims. 
\end{proof}


We now construct the derived restriction and induction maps
\eqref{eq:derived-ind-res} in terms of the space
$E \sF$, as $\sF$-colocalization and completion respectively. 
\begin{cons}
We consider now the $\sF$-colocalization map
\( E \sF_+ \wedge M \to M;  \)
since $E \sF= \hocolim_{\sOGF} G/H_+$ and smash products commute with homotopy
colimits, we can write this map as
\begin{equation}  
\label{eq:dind}
\Ind_{\sF}^G \colon \hocolim_{\sOGF} (G/H_+ \wedge M) = E\sF_+\wedge M \to M. 
\end{equation} 
Similarly, we can identify the $\sF$-completion map $M \to F(E \sF_+, M)$;
with the map:
\begin{equation} 
\label{eq:dres}
\Res^G_{\sF}\colon M \to \holim_{\sOGF^{op}} F(G/H_+, M).
\end{equation} 
\end{cons}
\begin{prop}\label{prop:relevance}
	If $M$ is $\sF$-nilpotent, then 
the derived induction and restriction maps 	
\eqref{eq:dind} and \eqref{eq:dres} are equivalences.
\end{prop}

\begin{proof}
This now follows from 
\Cref{prop:Fnil-implies-acyclic}. 
\end{proof}

We round out this subsection with a few basic examples of derived defect bases. 
We remark also that this technique is essentially \cite[Sec. 10]{HHR14}. 
\begin{prop}\label{rem:borel-tate}
	Let $\sF=\sTriv$ be the trivial family of subgroups. Suppose  $\bb{R}$ is a Borel-equivariant $G$-ring spectrum. Then $\bb{R}$ is $\sTriv$-nilpotent if and only if the $G$-Tate construction $( \widetilde{E} \sTriv \wedge \bb{R})^G$ of $R$ is contractible.
\end{prop}
\begin{proof} 
We know that $\bb{R}$ is $\sTriv$-nilpotent if and only if $ \widetilde{E} \sTriv
\wedge \bb{R}$ is contractible by \Cref{prop:Fnil-implies-acyclic}. Since this
is a ring object, it is contractible if and only if its fixed point spectrum is
contractible. 
\end{proof}

\begin{prop}\label{prop:kr-defect-base}
	The derived defect base of $C_2$-equivariant $K\bR$-theory \cite{Ati66a} is $\sTriv$, the trivial family.
\end{prop}
\begin{proof}
We need to show that $K\bR$ is
$\sTriv=\sP$-nilpotent. In view of
\cite[Thm.~\ref{S-geofixedpointscontractible}]{MNNa}, it
suffices to show that the geometric fixed point spectrum $\Phi^{C_2}K\bR=(\wt{E}\sP\wedge R)^{C_2}$ is contractible. In this language the relevant calculation appears in
\cite[Thm.\ 5.2]{Faj95}
and in
\cite[\S 7.3]{HHR11}; however the result follows from \cite[Prop.\ 3.2 and Lem.\ 3.7]{Ati66a}.
In fact, in the proof of \cite[Thm.\ 5.2]{Faj95}, it is observed that the cube of
the Euler class of the reduced regular representation of $C_2$ vanishes in $K\bR$.
\end{proof}

Let $\MOG$ and $\MUG$ denote the genuine $G$-equivariant real and complex cobordism spectra of tom Dieck \cite{toD70,BrH72}. When $G=C_2$, let $\MUR$ denote the real $G$-equivariant complex cobordism spectrum of Landweber \cite{Lan68}.

\begin{prop}\label{prop:do-not-descend}
 	The derived defect base of any of $\MOG$, $\MUG$, and $\MUR$ is $\sAll$, the family of all subgroups of $G$. 
\end{prop}
\begin{proof}
We need to show that there is no proper family $\sF$ such that any of these
$G$-spectra is $\sF$-nilpotent. 
By
	\Cref{prop:ring-criteria-fnil}, to prove this for a
	$G$-ring spectrum $R$, it suffices to show that 
	
	\[ 0\neq\pi_*\Phi^G R\left( \cong \pi_*^G \wt{E}\sP\wedge R\cong \pi_*^GR[e_{\tilde{\rho}_G}^{-1}]\right) .\] 

	In each of the stated cases this is known. The results for $\MOG$ and $\MUG$ are due to tom Dieck and can be found in \cite[\S XV Lem.~3.1]{May96} and \cite[Lem.~2.2]{toD70} respectively. For $\MUR$ this is \cite[Cor.~3.4]{Lan68}. 
\end{proof}

\subsection{The class $\sFNil$ and the homotopy limit spectral sequence}\label{sec:fnil-holim}

Before proving the equivalence of Conditions \eqref{it:fnil-def} and
\eqref{it:fnil-ss} from \Cref{thm:fnil-three} in \Cref{thm:fnil-three-prime}
below, we will give an alternate construction of $E\sF$ and the $\sF$-homotopy
limit spectral sequence, following \cite[Sec. 21]{GrM95}. 

First, we describe another model for $E \sF$.
For a space $Z$, let $d_0\colon Z^{\bul+1}\rightarrow *$ denote the standard augmented simplicial space which in degree $n$ is the $(n+1)$-fold product of $Z$. 


When $Z \neq \emptyset$, we can pick a point in $Z$ to define a section $s_{-1}$ of $d_0$. This section defines an additional degeneracy in each degree, or equivalently a retraction diagram of simplicial spaces
\begin{equation}\label{eq:contracting-homotopy}
 *\xrightarrow{s_{-1}} Z^{\bul+1}\xrightarrow{d_0} *
\end{equation}
with a simplicial homotopy $s_{-1}d_0\simeq \mathrm{Id}$ \cite[\S
III.5]{GoJ99}. We will call an augmented simplicial space admitting extra
degeneracies \emph{split}. 

When $Z$ is a $G$-space, it is necessary and sufficient for $Z$ to have a
$G$-fixed point to split $Z^{\bullet + 1}$ as a simplicial $G$-space. More
generally, if $Z^H \neq \emptyset$ for $ H \leq G$, then $\Res_H^G(Z^{\bul+1})\simeq (\Res_H^G Z)^{\bul+1}$ is split
as a simplicial $H$-space.  This implies that \(G/H\times Z^{\bul+1}\simeq
\Ind_H^G \Res_H^G Z^{\bul+1}\) is split as an augmented simplicial $G$-space.  

\begin{prop}[cf.~{\cite[p.~119]{GrM95}}]\label{prop:retraction}
	Let $\sF$ be a family of subgroups of $G$ and consider the $G$-space
	  \(X=\coprod_{H\in I} G/H,\) 
	  where $I \subset \sF$ contains a representative from each conjugacy class of maximal subgroups in $\sF$.  

	  Then there is an equivalence \[ |X^{\bul +1}|\simeq E\sF .\] Moreover, if
	  $H \in \sF$ then $G/H\times X^{\bul+1}$ is  split.
\end{prop}
\begin{proof}
	This follows easily from the observations above and the characterization of
	$E\sF$ from \eqref{eq:EF-univ-prop}, since taking fixed points commutes with geometric realizations and products. In particular, $|X^{\bul+1}|^H$ is contractible when $X$ has an $H$-fixed point and is empty otherwise. 
\end{proof}


The geometric realization of a simplicial $G$-space, $Z=|W_\bul|$, admits two standard increasing filtrations by $G$-CW subcomplexes. The first is the filtration by dimension: 
\[ F_{-1}Z=\emptyset \subseteq F_0 Z\subseteq \cdots \subseteq F_\infty Z= Z\] and depends on a choice of $G$-CW structure on $Z$. The second arises from the skeletal filtration on $Z$: \[
F^\prime_{-1} Z=\emptyset \subseteq F^\prime_0 Z\subseteq \cdots \subseteq F_\infty^\prime Z = Z.
\]
Here $F^\prime_n Z:=\hocolim_{\Delta^{\op}_{\leq n}} W_\bul$ is the $n$-skeleton of $Z$ and depends on the presentation of $Z$ as the geometric realization of a simplicial $G$-space. 

Fixing a $G$-spectrum $M$ and applying $F(-_+, M)$ to these two filtrations, we
obtain two towers of $G$-spectra, $ \{F(F_n Z_+, M)\}_{n\geq 0}$ and
$\{F(F_n^\prime Z_+, M)\}_{n \geq 0}$. In general, if we apply $\pi_*^G$ to a bounded below tower we obtain an exact couple and an associated spectral sequence conditionally converging to the homotopy groups of the homotopy inverse limit of the tower\footnote{More generally, one can apply $\pi_\star^{(-)}$ to obtain a Mackey functor-valued, $RO(G)$-graded spectral sequence. This variant, although useful, will not be required for this paper.} \cite[\S 7]{Boa99}.

 In the case of the first tower, we are using a $G$-CW filtration on $Z$
 which satisfies \[F_n Z/F_{n-1}Z\simeq \bigvee_{i\in I_n} {G/H_i}_+\wedge
 S^n,\] where  $I_n$ is the set of orbits of $n$-cells of $Z$. The
 $E_1$-complex associated to the tower $\{F(F_n Z, M)\}_{n\geq 0}$ is 
\begin{equation}\label{eq:AHSS-E1}
 E_1^{s,t}=\pi_{t-s}^G F(F_s Z/F_{s-1}Z, M)\cong \prod_{i\in I_s} \pi_t^{H_i} M, 
 \end{equation}
 where the $d_1$-differential is induced by the attaching maps. This yields the
 equivariant analogue of the Atiyah-Hirzebruch spectral sequence whose
 $E_2$-term is, by definition, the \emph{Bredon cohomology} of $Z$ with coefficients in $\Gpi M$:
\begin{equation}\label{eq:AHSS-EF}
	 H^s_G(Z;\pi_t^{(-)} M)\Longrightarrow \pi_{t-s}^G F(Z_+,M).
\end{equation}
If $M$ is a $G$-ring spectrum then \eqref{eq:AHSS-EF} is a spectral sequence of algebras \cite[App.~B]{GrM95}.

For the second tower we obtain the equivariant analogue of the Bousfield-Kan spectral sequence \cite[\S 3]{Bou89}:
\begin{equation}\label{eq:BKSS}
	\pi^s\pi_t^G(F(W_{\bul +};M))\Longrightarrow \pi_{t-s}^G F(Z_+; M).
\end{equation}
Here, the $E_2$-term is the cohomology of the graded cosimplicial abelian
group $\pi_*^G(F(W_{\bul +},M))$.
In \Cref{sec:ss} we will discuss the $E_2$-terms of these two spectral sequences further. 

\begin{prop}\label{prop:AHSS-equals-BKSS}
	Suppose that $Z$ is the geometric realization of a simplicial $G$-space $W_\bul$ and $M$ is a $G$-spectrum. If $W_n$ is discrete for each $n$, then the two spectral sequences \eqref{eq:AHSS-EF} and \eqref{eq:BKSS} are isomorphic from the $E_2$-page on. 
\end{prop}
\begin{proof}
	To compare \eqref{eq:AHSS-EF} and \eqref{eq:BKSS} we would like a map
	between the associated towers. We do have an equivalence $F_\infty Z\simeq
	F_\infty^\prime  Z$, but in general this equivalence need not respect the
	filtrations.  However, when $W_\bul$ is degree-wise discrete, then for each
	$n$,  $F_n^\prime Z$ is $n$-dimensional and the skeletal filtration on $Z$
	is just the dimension filtration for a different choice of $G$-CW structure
	on $Z$. In this case, we can find an equivalence $s\colon F_\infty
	Z\rightarrow F_\infty ^\prime Z$ which respects the filtrations
	\cite[Cor.~3.5]{May96} and hence induces a map from the spectral sequence in
	\eqref{eq:BKSS} to the spectral sequence of \eqref{eq:AHSS-EF}. Applying the
	same argument to an inverse equivalence $t\colon F_\infty^\prime
	Z\rightarrow F_\infty Z$ and to a homotopy $ts\simeq \mathrm{Id}$, we obtain a homotopy equivalence of $E_1$-complexes and hence an isomorphism at $E_2$.
\end{proof}

We now turn our attention to $E\sF=\hocolim_{\sOGF} i$ for $i: \sOGF \subset
\Top_G$ the inclusion. We will model this homotopy colimit as the geometric realization of the standard two-sided bar construction (see \Cref{sec:esf} for further details):
\begin{equation}\label{eq:bar-construction}
	E\sF\simeq |B_\bul(*,\sOGF, i)|. 
\end{equation}
\begin{defn}\label{def:sf-homotopy-limit-ss}
	Let $M$ be a $G$-spectrum. The $\sF$\emph{-homotopy limit spectral sequence
	associated to $M$} is the homotopy spectral spectral sequence associated to
	the tower \[ \{F(\sk_n E\sF_+, M)\}_{n\geq 0},\] where $E\sF$ is equipped with the simplicial structure of \eqref{eq:bar-construction}.
\end{defn}


\begin{prop}\label{prop:three-ss-are-equal} 
	Let $N$ be a $G$-spectrum. Then there is an isomorphism, from $E_2$ on, between:
	\begin{enumerate}
		\item The $\sF$-homotopy limit spectral sequence:
		\[
		 	\pi^s\pi_t^G(F(B_\bul(*,\sOGF,i)_+,N))\Longrightarrow \pi_{t-s}^G
			F(E\sF_+, N)\cong \pi_{t-s}^G \holim_{\sOGF^\op} F(G/H_+, N)
	 	\]
	 	from \Cref{def:sf-homotopy-limit-ss},
		\item  the Bousfield-Kan spectral sequence 
		\[
		 	\pi^s\pi_t^G(F(X^{\bul+1}_+,N))\Longrightarrow \pi_{t-s}^G F(E\sF_+, N)\cong \pi_{t-s}^G \holim_{\sOGF^\op} F(G/H_+, N) \\
	 	\]
		associated to a simplicial presentation of $E\sF$ from \Cref{prop:retraction}, and 

	\item the equivariant Atiyah-Hirzebruch spectral sequence 
	\[H^s_G(E\sF;\pi_t^{(-)} N)\Longrightarrow \pi_{t-s}^G F(E\sF_+, N)\cong \pi_{t-s}^G\holim_{\sOGF^\op}F(G/H_+,N).\]
	\end{enumerate}

	Moreover, when $N=F(Y,M)$, for two $G$-spectra $Y$ and $M$ such that $M$ is $\sF$-nilpotent, the above spectral sequences converge to $M^*_G(Y)$.\footnote{As a consequence of \Cref{thm:fnil-three-prime}, \Cref{enum:three} below, they actually converge strongly to their abutment.}
\end{prop}

\begin{proof} 
	When $G$ is discrete, both $X^{\bul+1}$ and $B_\bul(*,\sOGF, i)$ are degree-wise discrete. So it follows from  \Cref{prop:AHSS-equals-BKSS} that
	all three spectral sequences are forms of the Atiyah-Hirzebruch spectral sequence for $E\sF$ and hence isomorphic.
	The final claim follows from \Cref{prop:relevance} and the 
	isomorphism $\pi_{t-s}^G F(Y_+,M)\cong M^{s-t}_G(Y).$
\end{proof}


To proceed, we will need to recall some results on towers of $G$-spectra from
\cite[\S 3]{Mat15}. We denote by $\Tow(\SpG)=\Fun((\bZ_{\geq 0})^\op,\SpG)$ the
$\infty$-category of towers in $\SpG$. Inside this $\infty$-category is
$\Townil(\SpG)\subset\Tow(\SpG)$, the full subcategory of nilpotent towers,
i.e., those towers $\{X_n\}_{n\ge 0}$ such that for some $N\ge 0$ and all $k\ge
0$, the map $X_{N+k}\to X_k$ is zero. 
We denote by $\Towfast(\SpG)\subset\Tow(\SpG)$ the full subcategory of
\emph{quickly converging} towers, i.e., those towers $\{ X_n\}_{n\ge 0}$ such that the cofiber of the canonical map of towers $\{\holim X_n\}\to\{X_n\}_{n\ge 0}$ is contained in $\Townil(\SpG)$. 
It follows from the definitions that $\Towfast(\SpG)\subset\Tow(\SpG)$ is a thick subcategory, and that exact endofunctors of $\SpG$ preserve $\Towfast(\SpG)$. 


We can now formulate the main result of this subsection, which in particular establishes the equivalences between Conditions \eqref{it:fnil-def} and \eqref{it:fnil-ss} from \Cref{thm:fnil-three}.

\begin{thm}\label{thm:fnil-three-prime}
	The following three conditions on a $G$-spectrum $M$ are equivalent:
	\begin{enumerate}
		\item\label{enum:one} The $G$-spectrum $M$ is $\sF$-nilpotent.
		\item\label{enum:two} The restriction map 
		$\Res_\sF^G: M\longrightarrow
		\holim_{\sOGF^\op} F(G/H_+,M)\simeq F(E\sF_+,M)$
		is an equivalence and the associated tower		$\{ F(\sk_n E\sF_+, M)\}_{n\ge 0}$ converges quickly.
		\item\label{enum:three} The map
		$M \to F(E \sF_+, M)$
 is an equivalence and there are integers $m$ and $n\geq 2$ such that for every $G$-spectrum $Y$, the $\sF$-homotopy limit spectral sequence:
		\[ E_2^{s,t}=H^s(E\sF;\pi_{t}^{(-)} F(Y,M)) \Longrightarrow M_G^{s-t}(Y)\]
		has a horizontal vanishing line of height $m$ on the $E_n$-page. In other words, $E_k^{s,*}=0$ for all $s>m$ and $k\geq n$.  
	\end{enumerate}
\end{thm}

\begin{proof} 
	The equivalence \eqref{enum:two}$\iff$\eqref{enum:three} is \cite[Prop.~3.12]{Mat15} combined with the identification of the $\sF$-spectral sequence from \Cref{prop:three-ss-are-equal}.

	We will now show \eqref{enum:one}$\iff$\eqref{enum:two}. 
	Let $A_{\sF} = \prod_{H \in \sF} F(G/H_+, S)$, so that a $G$-spectrum is
	$\sF$-nilpotent if and only if it is $A_{\sF}$-nilpotent. 
	Write $E \sF = |X^{\bullet + 1}|$ for $X = \bigsqcup_{H \in \sF} G/H$.
Then the tower $\left\{F(\mathrm{sk}_n E \sF_+, M\right\})$ is the  
$\mathrm{Tot}$ tower of the $A_{\sF}$-cobar complex of $M$. This is a quickly
converging tower with homotopy limit $M$ if and only if the $A_{\sF}$-Adams tower
\cite[Construction~\ref{S-adamstowerdef}]{MNNa} is nilpotent
(note that the $A_{\sF}$-Adams tower is the cofiber of the map of towers
$\left\{M\right\} \to \left\{F(\mathrm{sk}_n E\sF_+, M)\right\}$ by
\cite[Prop.~\ref{S-adamscobar}]{MNNa}).  Furthermore, that holds if and only if
and $M$ is $A_{\sF}$-nilpotent \cite[Prop.~\ref{S-properties:nilpotent}]{MNNa}.
\end{proof}

Recall also that we can quantify nilpotence, leading  to the  notion of the {\em$\sF$-exponent} of an $\sF$-nilpotent
$G$-spectrum $M$, denoted $\exp_{\sF}(M)$ \cite[Def.~\ref{S-Fnildef}]{MNNa}.  Recall again that, associated to
$G$ and $\sF$, there is the commutative algebra $A_{\sF}:=\prod_{H\in\sF}
F(G/H_+,S)$ in $\SpG$. The fiber $I$ of the canonical map $S\to A_{\sF}$ is
a non-unital algebra, and the $\sF$-exponent of $M \in \sFNil$ is the minimum
number $n\ge 0$ such that $(I^{\wedge n}\to S)\wedge M$ is zero. For $Y\in\SpG$,
we will denote by $E^{*,*}_*(Y)$ the $\sF$-homotopy limit spectral sequence
converging to $M^*_G(Y)$. We can then formulate the following alternate descriptions of the $\sF$-exponent.

\begin{prop}\label{prop:exponent}
For a nontrivial $\sF$-nilpotent spectrum $M$, the following integers are equal:
	\begin{itemize}
		\item the $\sF$-exponent $\exp_{\sF}(M)$,
		\item the minimal $n$ such that the canonical map $M\simeq
		F(E\sF_+,M)\longrightarrow F(\sk_{n-1} E\sF_+,M)$ in $\SpG$ admits a retraction,
		\item the minimal $n^\prime$ such that $M$ is a retract of an $F(Z_+,M)$ for an $(n^\prime-1)$-dimensional $G$-CW complex $Z$ with isotropy in $\sF$, 
		\item and the minimum $s\geq 0$ such that for all $Y\in\SpG$ and $k\geq
		s$, $E^{k,*}_{s+1}(Y)=E^{k,*}_\infty(Y)=0$. 
	\end{itemize}
\end{prop}


\begin{proof}
This follows easily from results in \cite{MNNa}. Fix the
$G$-space $X:=\coprod_{H\in\sF} G/H$ and the associated simplicial $G$-space
$X^{\bullet+1}$ which realizes to $E\sF$. One sees that the identification
$A_{\sF}\simeq F(X_+,S)$ generalizes to an identification of cosimplicial
commutative algebras in $\SpG$, namely the \emph{cobar construction}
$\mathrm{CB}^{\bullet}(A_{\sF})$ (cf.\ \cite[Sec.~2.1]{MNNa}) is equivalent to $F(X^{\bullet+1}_+,S)$. 
In view of this, the equality of the first two integers follows from
\cite[Prop.~\ref{S-prop:basiconexp}]{MNNa}.
To compare $n^\prime$ and $n$ we first note that by setting $Z=\sk_{n-1} E\sF$ we see that $n^\prime\leq n$. The other inequality follows 
because $F(Z_+, M)$, for an $(n'-1)$-dimensional $G$-CW complex $Z$ with isotropy
in $\sF$ and for any $G$-spectrum $M$, has $\sF$-exponent $\leq n'$.

Finally, we show $n = s$. 
Using
$\mathrm{CB}^{\bullet}(A_{\sF})\simeq F(X^{\bullet+1}_+,S)$ again, one sees
that our $\sF$-homotopy limit spectral sequence can be identified with the
$A_{\sF}$-based Adams spectral sequence as in \cite{Gre92}, and it is
well-known that the Adams filtration of a map $f\colon \Sigma^{-*} Y\to M$ in $M^*_G(Y)$ is exactly the maximum $q$ such that $f$ factors through $I^{\wedge q}\wedge M\to M$. 
It follows that $(n-1)$ is (precisely) the maximum $A_{\sF}$-Adams filtration of any map
into $M$, which implies that $E_{\infty}^{\ast, k}(Y) = 0$ for $k \geq n$ and
for any $G$-spectrum $Y$; moreover, $n$ is minimal with respect to this
property.

It remains to show that 
the $\sF$-spectral sequence degenerates at $E_{n+1}$, or equivalently that
$d_i = 0$ for $i \geq n+1$. This is a very general assertion about these types
of generalized Adams spectral sequences. For simplicity of notation, we assume that $Y
= S^0$. 
The $E_1^{p,q}$-page of the spectral sequence gives the homotopy groups 
$\pi_p( \mathrm{fib}\left( \mathrm{Tot}_q \to \mathrm{Tot}_{q-1}\right)) $ for the
cosimplicial object $ M \otimes \mathrm{CB}^\bullet(A_{\sF})$. 
By \cite[Prop.~\ref{S-adamscobar}]{MNNa}, we have
\[  \mathrm{fib}\left( \mathrm{Tot}_q \to \mathrm{Tot}_{q-1}\right) = 
\mathrm{cofib}( I^{\wedge q+1} \to I^{\wedge q} ) \wedge M  = I^{\wedge
q}/I^{\wedge q+1} \wedge M.
\]

If a class survives to $E_{n+1}$, then it can be lifted to $$\mathrm{fib}\left(
\mathrm{Tot}_{q+n} \to \mathrm{Tot}_{q-1}\right) = 
I^{\wedge q}/I^{\wedge q + n + 1} \wedge M,$$
by \cite[Prop.~\ref{S-adamscobar}]{MNNa} again. 
Consider now the diagram
\[ \xymatrix{
&  I^{\wedge q}/I^{\wedge q + n + 1} \wedge M \ar[d]^{\psi} \ar@{-->}[ld]
 \ar[r] &   \Sigma I^{\wedge q + n +
1} \wedge M \ar[d]^{\phi}
 \\
I^{\wedge q} \wedge M\ar[r] &  I^{\wedge q }/I^{\wedge q+1} \wedge M
\ar[r]^{\partial} & 
\Sigma I^{\wedge q+ 1} \wedge M.
}.\]
We claim that, under the hypotheses, there exists a dotted arrow making the diagram commute.
Therefore, our class can be in fact lifted to $\mathrm{fib}( \mathrm{Tot} \to
\mathrm{Tot}_{q-1})$ and so is a permanent cycle in the $\sF$-spectral sequence. 
To see this, we need to argue that the composite $\partial \circ \psi$ is
null-homotopic.
However, this follows from the fact that the diagram commutes and that 
$\phi$ is null-homotopic by hypothesis on $M$. 
\end{proof}

The proof of \Cref{thm:fnil-three} is now complete except for the identification of the $E_2$-term of the homotopy limit spectral sequence, and this will be completed in \Cref{sec:bredon}.

\begin{remark}\label{rem:dual-ss}
	One can dualize \cite[\S 3]{Mat15} since the notion of a stable $\infty$-category is self-dual. We thus obtain inside $\Fun(\mathbb{Z}_{\ge 0},\SpG)$ the nilpotent and quickly converging directed systems. The latter subcategory is thick and stable under exact endofunctors of $\SpG$. The exact couples associated to such directed systems once again define homological-type spectral sequences with horizontal vanishing lines. For example, when $M$ is $\sF$-nilpotent, $\{ \sk_n E\sF_+\wedge M\}_{n\ge 0}$ is a quickly converging directed system. It follows that for arbitrary $X\in\SpG$, the $\sF$-homotopy colimit spectral sequence
	\[ E^2_{s,t}=H_s^G(E\sF;\pi_t^{(-)} F(X,M))\cong \sideset{}{_s}\colim_{\sOGF} M_t^H(X)\Longrightarrow M_{t+s}^G(X)\]
	has a horizontal vanishing line at a finite page.

	Coupling this with the analogous result for the homotopy limit spectral sequence forces the generalized $\sF$-Tate spectral sequence of \cite[\S 22]{GrM95} to collapse to zero at some finite stage. Indeed, the positive degree terms of this spectral sequence are a quotient of the positive degree terms in the $\sF$-homotopy limit spectral sequence while the terms in degrees less than $-1$ are a subset of the positive degree terms in the $\sF$-homotopy colimit spectral sequence (cf.~\eqref{eq:tate}). Our vanishing results now imply the collapse of the $\sF$-Tate spectral sequence at a finite stage. By \Cref{prop:Fnil-implies-acyclic} this spectral sequence converges to 0.
\end{remark}
\section{Analysis of the spectral sequences}\label{sec:ss}


Let $G$ be a finite group and $\sF$ a family of subgroups. Let $X=\coprod_{H\in\sF}G/H$ be as in 
\Cref{prop:retraction}. 
 As observed in the previous section, the $\sF$-homotopy limit spectral
 sequence can be viewed as the Bousfield-Kan spectral sequence
 \cite[Ch.~X]{BoK72} associated to the cosimplicial $G$-spectrum
 $F(X^{\bul+1}_+, M)$ or as an equivariant Atiyah-Hirzebruch spectral sequence
 with $E_2$-term \[H^*_G(|X^{\bul+1}|_+;\Gpi M)\cong H^*_G(E\sF_+;\Gpi M).\] In \Cref{sec:bredon} we recall that this $E_2$-term can be identified with the derived functors $\lim^*_{\sOGF^\op}\Gpi M$. 

There is also an $\sF$-homotopy colimit spectral sequence and the chain complexes calculating the $E_2$-terms of the $\sF$-homotopy colimit and limit spectral sequences can be glued together to form the associated Amitsur-Dress-Tate cohomology groups $\widehat{H}_\sF^*(\Gpi M)$. In \Cref{sec:tate} we will review this construction and recall a few vanishing results. These results play a critical role in the proofs of \Cref{thm:gen-artin,thm:gen-fiso,cor:varieties} in \Cref{sec:calc-thms}.  They will also be used in the proof of the generalized hyperelementary induction theorem \Cref{prop:hyper-induction}, in \Cref{sec:split}. We conclude this section with a form of Quillen's stratification theorem (\Cref{prop:stratification}).

\subsection{Bredon (co)homology and derived functors}\label{sec:bredon}

In this subsection we review some classical results about coefficient systems,
and relate the $\sF$-homotopy limit spectral sequence to Bredon cohomology. 
Let $\sC$ be a small category and $\bZ\sC$ the category of contravariant
functors from $\sC$ to abelian groups; $\bZ \sC$ is an abelian category
with kernels and cokernels calculated object-wise, which admits enough
projectives and injectives. 



Now let $\underline{\bZ}$ denote the constant functor $c\mapsto \bZ$.
Then we have  
\begin{equation} \bZ\sC(\underline{\bZ}, M)\cong \lim_{\sC^{\op}} M,
\label{eq:derived-lim} \quad 
	\sideset{}{^*}\lim_{\sC^{\op}}(M) \cong \Ext^{*}_{\bZ\sC}(\underline{\bZ}, M),
\end{equation}
i.e., we recover the derived functors of the inverse limit. 

We now specialize to the primary case of interest for us. 
\begin{defn}[{Cf. \cite[Sec. I.4]{Bre67}}]\label{def:coefficient-system}
	The category of \emph{coefficient systems} (on a finite group $G$) is the category $\bZ\sOG$ of contravariant functors from $\sOG$ to abelian groups.
\end{defn}

\begin{examples}\label{exs:coefficient-systems}
\begin{enumerate}
\item Associated to any $G$-set $X$ we obtain a coefficient system $\bZ[X]$ defined by \[ \bZ[X]\colon G/H\mapsto \bZ\{\ho\Top_G(G/H, X)\}\cong \bZ[X^H].\] 
When $X = G/H$, $\bZ[X]$ is the projective functor $\bZ\{\sOG(-,G/H)\}$ considered above. 

\item Let $X$ be a $G$-CW complex and for each $n\geq 0$ let $X_n$ be the
$G$-set of $n$-cells in $X$. The attaching maps define a chain complex of
coefficient systems $C_*(X):=\bZ[X_*]$. 

\item Let $\underline{\bZ[\sF]}$ denote the coefficient system 
\[ \underline{\bZ[\sF]}\colon G/K\mapsto H_*(E\sF^K;\bZ)=H_0(E\sF ^K;\bZ).\]
By \eqref{eq:EF-univ-prop}, we see that $\underline{\bZ[\sF]}(G/K)=\bZ$, when $K\in\sF$, and is zero otherwise.

\item A $G$-spectrum $M$ defines a graded coefficient system $\Gpi M$ by \[\Gpi
M\colon G/H\mapsto \pi_*^G F(G/H_+, M)\cong\pi_*^H M.\] 
\end{enumerate}
\end{examples}

We now quote the following classical relationship between the Bredon cohomology
of $E \sF$ and the higher limits of $C$ over $\sOGF^{op}$. See also \cite[Prop.
2.10]{Gro02} for a treatment and many applications. 

\begin{prop}[{Cf. \cite[Prop. 4.2]{Wa84} or \cite[Prop. V.4.8]{May96}}] 
Let $C\in \mathbb{Z}\mathcal{O}(G)$ be a coefficient system. 
Then there is an identification between the Bredon cohomology 
$H^s_G( E\sF; C)$ and the derived functors $\lim^s_{\sOGF^{op}} C$.
\end{prop} 
\begin{cor}\label{prop:e2-identification}
Fix a $G$-spectrum $M$.
Let $E_2^{s,t}$ denote the $E_2$-term of the $\sF$-homotopy limit spectral sequence.  Then there is a chain of isomorphisms: 
\begin{align*}
	E_2^{s,t} & \cong H^s_G(E\sF;\pi_t^{(-)}M)\\
	 & \cong \Ext^{s,t}_{\ZOG}(\underline{\bZ[\sF]}, \Gpi M)\\
	 & \cong \sideset{}{^s}\lim_{\sOGF^\op} \pi_t^H M\\
	 & \cong \Ext^{s,t}_{\ZOGF}(\underline{\bZ}, \Gpi M).
 \end{align*}
In particular, the 0-line is $\lim_{\sOGF^\op} \pi_*^H M$.
\end{cor}
\begin{proof}

	The identification of the $E_2$-term as the derived functors of the limit is
	due to Bousfield and Kan \cite[Ch.~XI]{BoK72} and the remaining isomorphisms
	are consequences of the above discussion.
\end{proof}

The above results and identifications dualize, cf. \cite[Sec. V.4]{May96}.  A $G$-spectrum $M$ defines a \emph{covariant} functor $\Gpi M$ from $\sOG$ to (graded) abelian groups by 
\[(\Gpi M)(G/H)=\pi_*^G (G/H_+\wedge M)\cong \pi_*^H M.\] 
Now the skeletal filtration on $E\sF$ defines a homological Atiyah-Hirzebruch
spectral sequence with the following $E^2$-identifications \cite[Prop.
V.4.8]{May96}:
\begin{align*}
	E^2_{s,t}&\cong  H_{s}^G(E\sF;\Gpit M) \cong \Tor_{s,t}^{\ZOG}(\underline{\bZ[\sF]},\Gpi M) \cong \Tor_{s,t}^{\ZOGF}(\underline{\bZ},\Gpi M)\\
\\ & \cong \sideset{}{_s}\colim_{\sOGF} \pi_t^H M \Longrightarrow \pi_{t+s}^G \hocolim_{\sOGF} G/H_+\wedge M.
\end{align*}
Here, for a $G$-space $X$, the Bredon homology $H_*^G(X; \Gpit M)$ 
(cf.\ \cite[\S 1.4]{May96})
is defined
to be the homology of the chain complex \[ C_*^G(X;\Gpit
M):=C_*(X)\otimes_{\ZOG} \Gpit M\]  formed from the tensor product of graded functors.

\subsection{Amitsur-Dress-Tate cohomology}\label{sec:tate}

Let $C \in \mathbb{Z}\mathcal{O}(G)$ and  consider the Bredon
cohomology $H^s_G( E \sF; C) = \sideset{}{^s}
\lim_{\sOGF^{op}}C$ as in the previous subsection. 
In this subsection, we recall the following (\Cref{prop:torsion}): 
when $C$ comes from a \emph{Mackey functor} on $G$
(e.g., as the homotopy groups of a $G$-spectrum),
these groups are forced to be $|G|$-torsion for $s> 0$. This will be
fundamental for our computational applications of $\sF$-nilpotence. 
The property follows essentially from a transfer argument (a generalization of
the fact that for a finite group $G$, the group cohomology $H^s(G; \mathbb{Z})$
is $|G|$-torsion for $s > 0$) and appears, for instance, as \cite[Cor.
5.16]{JM92}. 

In this subsection, we
will review some of the theory of Amitsur-Dress-Tate cohomology \cite[\S
21]{GrM95}, which we will
use to prove these results.
The rest of this paper depends on the present section only through the $|G|$-torsion result from \cite{JM92}.

For notational simplicity, we will always assume that our Mackey functor is
given to us as the homotopy groups of a $G$-spectrum $M$. 

\begin{cons}
We can splice together the $E_1$-pages of the homological and cohomological spectral sequences from the previous section to define Amitsur-Dress-Tate cohomology. 

For this purpose let $C_*(E\sF;\Gpi M )$ and $C^*(E\sF;\Gpi M)$ denote the Bredon cellular chains and cochain complexes on $E\sF$ with coefficients in $\Gpi M$. 
These complexes have degree zero (co)homology given by $\colim_{\sOGF} \pi_*^H M$ and 
$\lim_{\sOGF^{op}} \pi_*^H M$, respectively, and we obtain a natural norm
map (cf.\ \eqref{eq:ind-res})
\begin{equation} \label{Findagain}\mathrm{N} \colon \colim_{\sOGF} \pi_*^H M \to  \lim_{\sOGF^{op}} \pi_*^H
M.\end{equation}
As a result, we obtain a map of complexes
\begin{equation} \label{Fnormcx}
C_*(E\sF;\Gpi M )\to C^*(E\sF;\Gpi M)
\end{equation}
 determined by the condition that it induce \eqref{Findagain} in $\pi_0$. 
We define the \emph{Amitsur-Dress-Tate complex} 
$\widehat{C}^*(E\sF;\Gpi M))$ to be the cofiber of the above map. 
\end{cons}
\begin{defn}[{\cite[Def.~21.1]{GrM95}}]\label{def:adt-cohomology}
	The \emph{Amitsur-Dress-Tate cohomology} groups of $\sF$ with coefficients in $\Gpi M$ are defined by \[\widehat{H}_\sF^{*}(\Gpi M) := H^*(\widehat{C}^*(E\sF;\Gpi M)).\] 
\end{defn} 


We immediately obtain the following identification of the Amitsur-Dress-Tate
cohomology  in terms of \eqref{Findagain}:
\begin{equation}\label{eq:tate}
	\widehat{H}^{s}_{\sF}(\Gpi M)\cong
	\begin{cases} 
	H^{s,*}_G(E\sF; \Gpi M) & \mbox{if }s>0\\
	H_{-s-1,*}^G(E\sF; \Gpi M) & \mbox{if }s<-1\\
	\coker \mathrm{N} & \mbox{if }s=0\\
	\ker \mathrm{N} & \mbox{if }s=-1
	\end{cases}
\end{equation}

We will now record some basic properties of Amitsur-Dress-Tate cohomology.
\begin{prop}\label{prop:torsion}
	Suppose that $R$ is a $G$-ring spectrum and $M$ is an $R$-module, then:
	\begin{enumerate}
		\item The Amitsur-Dress-Tate cohomology groups $\widehat{H}^{*}_{\sF}(\Gpi R)$ have an induced graded $\pi_*^G R$-algebra structure and $\widehat{H}^{*}_{\sF}(\Gpi M)$ is a graded module over $\widehat{H}^{*}_{\sF}(\Gpi R)$ such that the isomorphisms in \eqref{eq:tate} respect this structure. \label{it:tate-multiplicative}
		\item If $x=\Ind_H^G y \in \pi_*^G R$ for some $H\in \sF$ and $y \in
		\pi_*^H R$, then $x\cdot \widehat{H}^{*}_{\sF}(\Gpi R)=0$. \label{it:ind-dies}
		\item\label{item:index}  The commutative ring
		$\widehat{H}^0_{\sF}(\pi_0^{(-)} S)$ is annihilated by $|G|$. 
	We let $n(\sF)$ be the minimal positive integer which vanishes in
	$\widehat{H}^0_{\sF}(\pi_0^{(-)} S)$, so that $n(\sF) \mid |G|$. 	
		\item The number $n(\sF)$ from \eqref{item:index} is the minimal positive
		integer $n$ such that $n\cdot \widehat{H}^{*}_{\sF}(\Gpi M)=0$, for
		\emph{every} $R$ and $M$.
	\end{enumerate}

	In particular, if $i>0$ then $H_i(E\sF;\Gpi M)$ and $H^i(E\sF;\Gpi M)$ are $n(\sF)$-torsion. 
\end{prop}

\begin{definition} 
For a finite group $G$ and a family $\sF$ of subgroups of $G$, the integer $n(\sF)$ in \Cref{prop:torsion}, \eqref{item:index} is called {\em the index of 
the family} $\sF$ (of subgroups of $G$).
\end{definition}

\begin{proof}[Proof of \Cref{prop:torsion}]

	The first claim is a graded form of \cite[Prop.~2.3]{Dre73b}. It follows that $\widehat{H}^{*}_{\sF}(\Gpi R)$ is a module over 
	\begin{equation}\label{eq:tatering}
	\widehat{H}^{0}_\sF(\Gpi R)\cong \lim_{\sOGF^\op} \pi_*^{(-)} R/\Im\left(
	\mathrm{Ind}^G_{\sF}\right)=\lim_{\sOGF^\op} \pi_*^{(-)} R/(\sum_{H\in\sF} \Im \Ind_H^G (\pi_*^H R)).
	\end{equation}
	
	  This immediately implies the second claim. 
The fourth claim is clear because every $\widehat{H}^*_{\sF}(\Gpi M)$ is a module over $\widehat{H}^0_{\sF}(\pi_0^{(-)} S)$,
and the third claim will be addressed in the lemma below.

\end{proof}


Recall that we have $\pi_0^G S\simeq A(G)$, the Burnside ring of $G$. Jointly
with \eqref{eq:tatering} applied to $R=S$ and $*=0$, this yields a description of the commutative ring $\widehat{H}^0_{\sF}(\pi_0^{(-)} S)$ in terms of the Burnside rings $A(H)$ for certain subgroups $H\leq G$, and shows that claim \eqref{item:index} of \Cref{prop:torsion} is equivalent to the following result.

\begin{lemma}\label{lem:exists-index}
	There is a minimal positive integer $n(\sF)$ such that there exists $x\in \Im \Ind_\sF^G\subseteq A(G)$ and 
			\[
				y\in \ker\left(A(G)\xrightarrow{\Res_\sF^G} \lim_{\sOGF^\op}A(-) \right)
			\]
			such that $n(\sF)=x+y$.
	Furthermore, the integer $n(\sF)$ divides the group order $|G|$.
\end{lemma}
\begin{proof} 
See \cite[Prop.~21.3 and Cor.~21.4]{GrM95}. 
\end{proof} 

\begin{remark}

The existence proof of $n(\sF)$ is constructive. In fact, computing $n(\sF)$ is a linear algebra problem involving the table of marks of $G$ which can be calculated by a computer algebra package such as GAP.
\end{remark}

\begin{examples}\label{exs:indices}
	\begin{enumerate}
		\item When $G=A_5$ we have calculated the indices of various families in \Cref{fig:indices} using the table of marks in \Cref{fig:table-of-marks}.
		\item A prime $p$ divides $n(\sP)$ if and only if there is a nontrivial homomorphism $G\rightarrow C_p$ or, equivalently, $H^1(BG;\bF_p)\neq 0$ \cite[Ex.~21.5.(iii)]{GrM95}. In particular, if $G$ is perfect, then we have $n(\sP)=1$. \label{it:exs-indices-1}

\begin{figure}
\centering

\begin{subfigure}{.25\textwidth}
\centering
\scalebox{0.65}{
\setlength{\tabcolsep}{4pt}
\begin{tabu}{ c  c }
\hline
Family $\sF$ & $n(\sF)$\\
\hline
$\sTriv$ & 60 \\ 
$\sC_{(2)}$ & 30\\
$\sAb_{(2)}$ & 15\\
$\sC_{(3)}=\sAb_{(3)}$ & 20\\
$\sC_{(5)}=\sAb_{(5)}$ & 12\\
$\sC$ & 2\\
$\bb{\sAll}=\sAb=\sAb^2=\sE$ & 6 \\
$\sP$ & 1 \\ 
$\sAll$ & 1 \\
\hline
\end{tabu}
}
	\subcaption{\label{fig:indices} Indices of families of subgroups of $A_5$.}
\end{subfigure}%
\begin{subfigure}{.75\textwidth}
\centering
\scalebox{0.65}{
\setlength{\tabcolsep}{3pt}
\begin{tabu}{ c   c  c  c  c  c  c  c  c  c  }
\hline
 & $[A_5/e]$ & $[A_5/C_2]$ & $[A_5/C_3]$ & $[A_5/C_2\times C_2]$  & $[A_5/C_5]$ & $[A_5/\Sigma_3]$ & $[A_5/D_{10}]$ & $[A_5/A_4]$  & $[A_5/A_5]$ \\
\hline
$e$                       & 60 & 30 & 20 & 15 & 12 & 10 & 6 & 5 & 1 \\
$C_2$                  & 0   & 2   & 0   & 3   & 0   &  2  & 2 & 1 & 1 \\
$C_3$                  & 0   & 0   & 2   & 0   & 0   &  1  & 0 & 2 & 1 \\
$C_2\times C_2$ & 0   & 0   & 0   & 3   & 0   &  0  & 0 & 1 & 1 \\
$C_5$                  & 0   & 0   & 0   & 0   & 2   &  0  & 1 & 0 & 1 \\
$\Sigma_3$         & 0   & 0   & 0   & 0   & 0   &  1  & 0 & 0 & 1 \\
$D_{10}$             & 0   & 0   & 0   & 0   & 0   &  0  & 1 & 0 & 1 \\
$A_4$                 & 0   & 0   & 0   & 0   & 0   &  0  & 0 & 1 & 1 \\
$A_5$                 & 0   & 0   & 0   & 0   & 0   &  0  & 0 & 0 & 1 \\
\hline
\end{tabu}
}
	\subcaption{\label{fig:table-of-marks} Table of marks for $A_5$.}
\end{subfigure}
	\caption[]{}
\end{figure}
	\end{enumerate}
\end{examples}
\subsection{Artin induction and $\mathcal{N}$-isomorphism
theorems}\label{sec:calc-thms}

\Cref{prop:torsion} immediately implies the following more precise form of \Cref{thm:gen-artin}:
\begin{thm}\label{thm:gen-gen-artin}
	Let $M$ and $X$ be $G$-spectra and $\sF$ a family of subgroups such that $M$ is $\sF$-nilpotent. Then each of the following maps
	\begin{equation}\label{eq:cohomology}
	 \colim_{\sOGF} M^*_H(X)\xrightarrow{\Ind_\sF^G} M^*_G(X)\xrightarrow{\Res_\sF^G} \lim_{\sOGF^\op} M^*_H(X)
	\end{equation}
	\begin{equation}\label{eq:homology}
	 \colim_{\sOGF} M_*^H(X)\xrightarrow{\Ind_\sF^G} M_*^G(X)\xrightarrow{\Res_\sF^G} \lim_{\sOGF^\op} M_*^H(X)
	\end{equation}
	is an isomorphism after inverting $n(\sF)$, the index of the family $\sF$.
\end{thm}
\begin{proof}
	Since $\sFNil$ is closed under tensors and cotensors, it suffices to prove
	the theorem in the case $X=S^0$. Since \(\Gpi
	M=M_*^{(-)}(S^0)=M^{-*}_{(-)}(S^0),\) we see that \eqref{eq:cohomology} and \eqref{eq:homology} both reduce to statements about homotopy groups.  

	Set $n=n(\sF)$. Since $M$ is $\sF$-nilpotent, the $\sF$-homotopy limit spectral sequence converges strongly and has a horizontal vanishing line, say of height $m$ at the $N$th page. We will now analyze the following composition of maps
	\[ \ker \Res_\sF^G \hookrightarrow \pi_*^G M \twoheadrightarrow E_\infty^{0,*}\hookrightarrow E_2^{0,*}= \lim_{\sOGF^\op} \pi_*^H M\] where the composition of the latter two maps is $\Res_\sF^G$. Now $\ker \Res_\sF^G$ consists of those elements in $\pi_*^G M$ detected in positive filtration. The associated graded of this filtration on $\pi_*^G M$ is $\bigoplus_{s\geq 0}E_\infty^{s,*}$. These groups are $n$-torsion for $s>0$ by \Cref{prop:torsion} and $0$ for $s>m$. So if $x$ is detected in $E_2^{s,*}$ for $s>0$ then $nx$ is zero \emph{modulo higher filtration}. Since the groups in filtration degree greater than $m$ are zero we see that $n^{m}\cdot \ker \Res_\sF^G=0$. So $\Res_\sF^G$ is an injection after inverting $n$.

	Now suppose that $x\in E_2^{0,*}$ is not in the image of $\Res_\sF^G$. Since the spectral sequence converges strongly $x$ must support a differential. Suppose that $d_2x=y\neq 0$. Since $y$ is in positive filtration it is $n$-torsion and hence $d_2(nx)=ny=0$. So $nx$ survives to $E_3$. Inductively we see that $n^{k} x$ survives to the $E_{2+k}$ page. Using the horizontal vanishing line we see that there must be a fixed $k$ such that $n^k E_2^{0,*}\subset E_\infty^{0,*}=\Im \Res_\sF^G$. It follows that $\Res_\sF^G$ is an isomorphism after inverting $n$.

	The claim for $\Ind_\sF^G$ is easier and only requires that the map $\Ind_\sF^G\colon E\sF_+\wedge M\rightarrow M$ be an equivalence,
	i.e., that $M$ be $\sF$-torsion, rather than that $M$ be actually $\sF$-nilpotent. Since inverting $n$ commutes with homotopy colimits and ordinary colimits, if we tensor the $\sF$-homotopy colimit spectral sequence 
	\[
		\sideset{}{_s}\colim_{\sOGF} (\pi_*^HM)\Longrightarrow \pi_*^G M
	\]
	with $\bZ[n^{-1}]$ we obtain the homotopy colimit spectral sequence for $M[n^{-1}]$.
	This spectral sequence collapses at $E_2$ onto the zero line by \Cref{prop:torsion} and the claim for $\Ind_\sF^G$ follows.
\end{proof}

\begin{thm}\label{thm:gen-gen-fiso}
	Let $R$ be an $\sF$-nilpotent $G$-ring spectrum and let $X$ be a $G$-space.
	Suppose further that for each $H\in \sF$, the graded ring $R^*_H(X)$ is graded-commutative. Then the canonical map
	\[ 
	\Res_\sF^G : R^*_G(X)\longrightarrow \lim_{\sOGF^\op} R^*_H(X)
	\]
	is a uniform $\cN$-isomorphism, i.e., there are positive integers $K$ and $L$ such that if $x\in \ker \Res_\sF^G$ and $y\in \lim_{\sOGF^\op} R^*_H(X)$ then $x^K=0$ and $y^L\in \Im\ \Res_{\sF}^G$.
	Moreover, after localizing at a prime $p$, $\Res_{\sF}^G$ is a uniform $\cF_{p}$-isomorphism.
\end{thm}
\begin{proof}
	Suppose that $x\in \ker \Res_\sF^G$. It follows from the strong convergence of the $\sF$-homotopy limit spectral sequence 
	\begin{equation}\label{eq:f-iso-holim}
		 E_2^{s,-t}=\sideset{}{^s}\lim_{\sOGF^\op} R_H^t(X) \Longrightarrow R^{t+s}_G(X)
	\end{equation}
	that $x$ is detected in positive filtration. This spectral sequence has a horizontal vanishing line at the $E_\infty$-page. More precisely,  we know that for $K=\exp_{\sF}(R)$, $E_\infty^{s,*}=0$ when $s\geq K$. It follows that $x^K=0$. 

	Now suppose that $y\in \lim_{\sOGF^\op} R^*_H(X)$ is not in the image of $\Res_\sF^G$. Convergence of the $\sF$-homotopy limit spectral sequence implies that such an element must support a nontrivial differential, say $d_n(y)=z\neq 0$. Since $z$ is in positive filtration, it is $N=n(\sF)$-torsion by \Cref{prop:torsion}. Replacing $y$ with its square if necessary, we can assume that $y$ is in even degrees. Now since $R_H^*(X)$ is a graded-commutative ring functorially in $H\in \sF$, $\lim_{\sOGF^\op}R_H^*(X)$ is a graded-commutative ring. It now follows from the Leibniz rule that $d_n(y^{N})=Ny^{N-1}z=0$ and that $y^{N}$ survives to the $E_{n+1}$-page. We can now argue by induction and, since the spectral sequence collapses at the $E_{K+1}$-page, it follows that for every $y\in \lim_{\sOGF^\op} R_H^*(X)$, $y^{2N^{K-1}}$ survives the spectral sequence. Setting $L=2N^{K-1}$ we that $\Res_\sF^G$ is a uniform $\cN$-isomorphism as described above. 

	To see that $\bZ_{(p)}\otimes \Res_\sF^G$ is a uniform $\cF_p$-isomorphism, 
observe first that since the kernel of $\Res_{\sF}^G$ is nilpotent, so is the kernel of 
$\bZ_{(p)}\otimes \Res_\sF^G$. 
	Consider now the $p$-localization of the $\sF$-homotopy limit spectral
	sequence in \eqref{eq:f-iso-holim}. 
Since the spectral sequence collapses with a horizontal vanishing line at a
finite stage, we can pass $p$-localization through the spectral sequence 	
and obtain a spectral sequence converging to $F(X_+, R)_{(p)}$. 	
Since everything above the zero-line at $E_2$ is now $p$-power torsion, it
follows that for any element $x \in E_2^{0, t}$, we have that $x^{p^k}$ is a
permanent cycle for $k \gg 0$. This shows that 
$\bZ_{(p)}\otimes \Res_\sF^G$ has image containing all $p^k$-powers for $k \gg
0$.

\end{proof}
\begin{remark} \label{rem:lowerboundexp}
The horizontal vanishing line in fact implies \(\ker(\Res_\sF^G)^{\exp_{\sF}(R)}=0.\)
\end{remark}

We conclude this section with several applications of \Cref{thm:gen-gen-fiso}
including \Cref{prop:stratification}, a form of Quillen's stratification
theorem. First we will prove the following two elementary propositions, which were known to Quillen, which
demonstrate how \Cref{thm:gen-gen-fiso} implies \Cref{cor:varieties}. 

\begin{prop}\label{prop:spec-homeo}
	If $f\colon A\to B$ is an $\cN$-isomorphism of commutative rings, then $\Spec(f)$ is a homeomorphism. 

 \end{prop}
\begin{proof} 
	We factor $f$ as $A\to A/\ker(f)\to B$. The first map induces a homeomorphism
	on $\Spec$ by \cite[\S 1, ex.~21.iv)]{AtM69} and the second one a closed continuous surjection
	by \cite[\S 1, ex.~21, v) and \S 5, ex.~1]{AtM69}. It remains to see that
	\(\Spec(B)\to \Spec(A/\ker(f))\) is injective. 
	If $p_1,p_2\subseteq B$ are prime ideals with the same contraction to $A/\ker(f)$, and $x\in p_1$ is given, then for some $n\ge 0$ we have \[x^n\in p_1\cap (A/\ker(f))\subseteq p_2,\] hence $x\in p_2$ and $p_1\subseteq p_2$. By symmetry this gives $p_1=p_2$, as desired.

\end{proof}

\begin{prop}\label{prop:v-iso}
	Suppose that $f\colon A\rightarrow B$ is a map of commutative rings such that:
	\begin{enumerate}
		\item $f\otimes \bQ$ is an isomorphism and 
		\item for every prime $p$, $f\otimes \bZ_{(p)}$ is an $\cF_p$-isomorphism.
	\end{enumerate}
	Then the natural transformation of functors of rings
	\[f^*\colon \ring(B,-)\rightarrow  \ring(A, -)\] is an isomorphism
	on algebraically closed fields. In other words, $f$ is a $\cV$-isomorphism \cite[Defn.~A.3]{GS99}.

\end{prop}

\begin{proof}
The first condition implies that  $f^*$ is an isomorphism after restricting to fields of
	characteristic 0.  For algebraically closed fields of characteristic $p$, since $f\otimes \bZ_{(p)}$ is an $\cF_p$-isomorphism by assumption and any $\cF_p$-isomorphism between two $\bF_p$-algebras is a $\cV$-isomorphism by \cite[Prop.~B.8]{Qui71b}, we just need to verify that reducing a $\cF_p$-isomorphism mod $p$ induces a $\cF_p$-isomorphism. 

	In other words, we need to show that if
	$f^\prime=f\otimes \bZ_{(p)}$ is an $\cF_p$-isomorphism then so is
	$\overline{f} = f \otimes_{\mathbb{Z}} \mathbb{F}_p$.
	Suppose that $\overline{x}\in \ker(\overline{f})$, which we lift to $x\in
	A_{(p)}$. Now $f^\prime(x)=pz$ for some $z\in B_{(p)}$. Since $f^\prime$ is
	an $\cF_p$-isomorphism, there exists an $m\geq 0$ and $y\in A_{(p)}$ such
	that $z^{p^m}=f^\prime(y)$. Now set $w=(x^{p^m}-p^{p^m}y)$ so
	\[f^\prime(w)=p^{p^m}z^{p^m}-p^{p^m}z^{p^m}=0.\] Since $f^\prime$ is an
	$\cF_p$-isomorphism, $w$ is nilpotent; however, $w$ reduces to $\overline{x}^{p^m}$ so $\overline{x}$ is nilpotent. 

	Now consider some $\overline{z}^\prime \in B\otimes \bF_p$ and choose a lift
	$z^\prime \in B_{(p)}$. Since $f^\prime$ is an $\cF_p$-isomorphism there is
	a non-negative integer $m^\prime$ and $y^\prime \in A_{(p)}$ such that
	$f^\prime(y^\prime)=(z^\prime)^{p^{m^\prime}}$. Reducing $y^\prime$ mod $p$, we see that $\overline{f}$ is an $\cF_p$-isomorphism.

\end{proof}

We will now combine the above results with the work of Quillen to obtain the following stratification result:
\begin{thm}\label{prop:stratification}
	Suppose that $R$ is a homotopy commutative $\sF$-nilpotent ring spectrum.
	Suppose further that $\pi_0^G R$ is noetherian and for every $H\in \sF$, $\pi_0^H R$ is finite over $\pi_0^G R$ via $\Res_H^G$. Then the canonical natural transformations 
	of functors of rings
	\[ \colim_{\sOGF} \ring(\pi_0^H R, -) \rightarrow \ring(\lim_{\sOGF^\op} \pi_0^H R, -)\xleftarrow{{\Res_\sF^G}^*} \ring(\pi_0^G R, -)\]
	are isomorphisms when the input  is an algebraically closed field. 
	Similarly, the canonical maps between Zariski spaces
	\[ \colim_{\sOGF} \Spec(\pi_0^H R) \rightarrow \Spec(\lim_{\sOGF^\op} \pi_0^H R)\xleftarrow{{\Res_\sF^G}^*} \Spec(\pi_0^G R)\]
	are homeomorphisms.
\end{thm}
\begin{proof}
First, the map $\Res^G_{\sF}: \pi_0^G R \to \varprojlim_{\sOGF^{op}} \pi_0^H R$ 
becomes an isomorphism after
rationalizing (\Cref{thm:gen-gen-artin}) and an $\cF_p$-isomorphism after localizing at $p$
(\Cref{thm:gen-gen-fiso}); in addition, it is an $\cN$-isomorphism. 
It follows that the map of Zariski spectra
\( 
\spec( \varprojlim_{\sOGF^{op}} \pi_0^H R ) \to \spec( \pi_0^G R)
\)
is a homeomorphism (\Cref{prop:spec-homeo}). 
Furthermore, for an algebraically closed field $L$ we have an isomorphism
$\mathrm{Ring}( \pi_0^G R, L) \simeq
\mathrm{Ring}( \varprojlim_{\sOGF^{op}} \pi_0^H R, L)$
via \Cref{prop:v-iso}. 
This shows that both the natural transformations directed to the left are
isomorphisms.

	Next, the natural transformations directed to the right are isomorphisms by
	results of \cite{Qui71b}. 
	Here we use the finiteness of each $\pi_0^H R$ as a $\pi_0^G
	R$-module to guarantee that each map $\spec( \pi_0^H R)  \to \spec( \pi_0^G
	R)$ is a closed map. 
	Consider the category of finite $\pi_0^G R$-algebras. 
By \cite[Cor. B.7]{Qui71b}, the $\mathrm{Spec}$ functor (to topological spaces)
carries finite limits of 
finite $\pi_0^G R$-algebras to colimits of topological spaces. 
Therefore, $\varinjlim_{\sOGF} \spec( \pi_0^H R) \to 
\spec( \varprojlim_{\sOGF^{op}}  \pi_0^H R)$ is a homeomorphism, as desired. 
Finally, by \cite[Lem.~8.11]{Qui71b}, the map 
$\mathrm{colim}_{\sOGF} \mathrm{Ring}( \pi_0^H R, L) \to \mathrm{Ring}(
\varprojlim_{\sOGF^{op}} \pi_0^H R, L)$ is an isomorphism for each algebraically
closed field $L$. 
Combining this with the previous paragraph, the theorem follows. 
	\end{proof}

\begin{example}
Suppose $n(\sF) = 1$. 
For instance, this occurs if 	$G$ is a perfect group and $\sF=\sP$
is the family of all proper subgroups of $G$ (Example~\ref{exs:indices}). 

In this case, the 
idempotent $e_{\sF}$ belongs to the Burnside ring $A(G)$. We obtain a
decomposition of the symmetric monoidal $\infty$-category $\SpG$ as 
\[ \SpG \simeq \mathcal{C}_1 \times \mathcal{C}_2,  \]
where $\mathcal{C}_1$ consists of those $G$-spectra on  which $e_{\sF}$ is the
identity (equivalently, is a self-equivalence), and $\mathcal{C}_2$ consists of those $G$-spectra on which $e_{\sF}$
is null (cf.\ \cite{Bar09}).

We claim that $\mathcal{C}_1$ is equal to the subcategories of $\sF$-nilpotent,
$\sF$-complete, and $\sF$-torsion $G$-spectra (which therefore all coincide). In
particular, $\sF$-nilpotence is a purely algebraic condition on
the homotopy groups of a $G$-spectrum in this case. 

We start by showing that every $\sF$-complete $G$-spectrum belongs to $\sFNil$. 
In fact, 
this follows from \Cref{thm:fnil-three-prime}, since our assumptions imply that
the associated $\sF$-spectral sequence has a horizontal vanishing line at $E_2$. 
We now invoke \cite[Prop~\ref{S-idempotentsplitnilp}]{MNNa} to obtain that 
the subcategories of $\sF$-nilpotent, $\sF$-complete, and $\sF$-torsion objects
in $\SpG$ all coincide and that 
there
is a splitting (of symmetric monoidal $\infty$-categories) of $\SpG \simeq
\mathcal{C}_1' \times \mathcal{C}_2'$ where $\mathcal{C}_1'$ consists of the
$\sF$-nilpotent objects and $\mathcal{C}_2'$ consists of the $\sF^{-1}$-local
objects. 

It remains to show that the two splittings of $\SpG$ coincide. 
To see this, observe that $e_{\sF}$ restricts to $1$ in $A(H)$ for $H \in \sF$.
As a result, $e_{\sF}$ acts as the identity on $\left\{G/H_+\right\}_{H \in
\sF}$ and therefore on the localizing subcategory they generate. It follows that 
$\mathcal{C}_1$ contains the $\sF$-torsion $G$-spectra, i.e., $\mathcal{C}_1
\supset \mathcal{C}_1'$. 
Conversely, if $X \in \SpG$ is $\sF^{-1}$-local, then its restriction to $\SpH$
for $H \in \sF$ is contractible; therefore the class $e_{\sF} \in A(G)$, as a
sum of classes induced from subgroups in $\sF$, acts by zero. Therefore,
$\mathcal{C}_2\supset \mathcal{C}_2'$. It now follows that $\mathcal{C}_1 =
\mathcal{C}_1', \mathcal{C}_2 = \mathcal{C}_2'$ as desired. 
	\end{example}

\subsection{The end formula}
\label{sec:endthm}

In this subsection, we will explain how to deduce from our methods rational and
$\cF_p$-isomorphism results as in \cite{Qui71b,HKR00,GS99}. 
These results will require studying a different homotopy limit spectral sequence and will require some additional machinery, which we will review.
The convergence properties of this new homotopy limit spectral sequence will be more subtle even when considering the family of \emph{all} subgroups; in particular these results will require as input a \emph{finite} 
$G$-CW complex (cf.\ \Cref{ex:badconv}).

Let $X$ be a $G$-space and let $E$ be a $G$-spectrum. 
Then we have a bifunctor
\( T \colon \mathcal{O}(G) \times \mathcal{O}(G)^{op} \to \mathrm{Ab}_* \)
(where $\mathrm{Ab}_*$ denotes the category of graded abelian groups)
given by the formula
\begin{equation} \label{bifunctorformula}  T(G/H, G/K) = (E^K)^* ( X^H).  \end{equation}
Here $E^K = \Hom_{\GSpec}(G/K_+, E)$ denotes the $K$-fixed point spectrum of $E$ and $X^H$ denotes the subspace of $H$-fixed points of $X$.
Note that the $E^K$ (resp.\ $X^H$) are non-equivariant spectra (resp.\ spaces) here. 

Fix a family of subgroups $\sF$ of $G$.
We can form the \emph{end} $\int_{\sOGF^{op}} T$ of this bifunctor, consisting precisely of
those tuples of elements $\left\{x_H \in (E^H)^*(X^H)\right\}_{H \in \sF}$ which have
the following property: whenever we are given a map $G/H \to G/H'$, the natural images of $x_H$ and
$x_{H'}$ in $(E^H)^*(X^{H'})$ agree. Such classes naturally arise from the following construction. Given a map of $G$-spectra $X_+ \to \Sigma^* E$ and an $H\in \sF$, we obtain a map of spectra $X^H_+\to \Sigma^* E^H$. Altogether these define a natural comparison map
\begin{equation} \label{endcompmap}  E_G^*(X) \to \int_{\sOGF^{op}} (E^H)^* (X^H).  \end{equation}

The main result of this section is the following. 
\begin{thm} 
\label{endformula}
Suppose $X$ is a finite 
 $G$-CW complex. 
 Let $E$ be an $\sF$-nilpotent $G$-spectrum.  Then 
 \eqref{endcompmap} becomes an isomorphism after inverting $|G|$. 
 In fact, the kernel and cokernel of \eqref{endcompmap} are both annihilated by
 a power of $|G|$.

 Suppose in addition that $E$ is an $A_\infty$-algebra in $\GSpec$ which is
 homotopy commutative. Then 
 the map
 \eqref{endcompmap} is a uniform $\cN$-isomorphism and 
 for any prime $p$, its $p$-localization   is a uniform $\cF_p$-isomorphism. 
\end{thm} 


\begin{remark}\label{ex:badconv}
	We note that unlike \Cref{thm:gen-gen-artin,thm:gen-gen-fiso}, \Cref{endformula} can fail if $X$ is not assumed to have the equivariant homotopy type of a finite 
$G$-CW complex.  
	
For instance, consider $G = C_2$ and consider Borel-equivariant
$\mathbb{F}_2$-cohomology, i.e., the $C_2$-spectrum $\underline{H \mathbb{F}_2}$.
Take the $C_2$-space $X = E C_2$ and the family $\sF = \sAll$. 
Since $X^{C_2} = \emptyset$, the description of the end becomes
	\[\left( \int_{G/H\in \sO(C_2)^\op} (\underline{H
	\mathbb{F}_2}^{H})^\ast(X^{H})\right)\cong H \mathbb{F}_2	^\ast( \ast)^{C_2} = \mathbb{F}_2 ,\]
while $( \underline{H \mathbb{F}_2})^*_{C_2}(X)\simeq H \mathbb{F}_2^*(BC_2) \simeq \mathbb{F}_2[e]$ for $|e| =
1$. 
Here the kernel of \eqref{endcompmap} is not nilpotent.

As another example, let us take $G=C_2$, $X=EC_2$, $\sF=\sAll$, and $E=\KG$, $C_2$-equivariant complex $K$-theory. So $\KG_{C_2}^0(X)\cong \KU_{C_2}^0(\ast)\cong KU^0(BC_2)\cong \bZ\oplus \bZ_2$ where $\mathbb{Z}_2$ is topologically generated by a non-nilpotent element. Since $X^{C_2}=\emptyset$, the calculation of the end simplifies:
	\[\left( \int_{G/H\in \sO(C_2)^\op} (\KG^{H})^0(X^{H})\right)\cong \mathrm{eq}\left(KU^0(X)\rightrightarrows \prod_{C_2} KU^0(X)\right)\cong KU^0(\ast)\cong \mathbb{Z}.\]
	We easily identify the comparison map in \eqref{endcompmap} as the augmentation map $\mathbb{Z}\oplus \mathbb{Z}_2\to \mathbb{Z}$.  The kernel of this map is the $\mathbb{Z}_2$-summand, which is neither torsion nor nilpotent.

\end{remark}

\begin{remark} 
Quillen's argument for \Cref{endformula} (for mod $p$ cohomology) in
\cite[Thm.~6.2]{Qui71b} is based on different techniques. 
Given a $G$-space $X$, 
the strategy is to consider the homotopy 
orbits $X_{hG}$ as a space mapping to the strict orbits $X/G$ and the Leray
spectral sequence (in sheaf cohomology) for this map. This shows that
\eqref{endcompmap} is an $\cF_p$-isomorphism for a family of subgroups containing the 
isotropy of $X$. To get to the family of elementary abelian groups, Quillen uses
an argument (which has since become standard) involving the flag variety of a
faithful representation, which will also appear in a different form later in
this paper. 

By contrast, we will set up another spectral sequence for 
$E_G^*(X)$ and show that \eqref{endcompmap} arises as an edge homomorphism, and
then show that the spectral sequence collapses with a horizontal vanishing line
at a finite stage. 
Our approach follows the strategy of \cite[Sec. 2--3]{HKR00}, who consider the
case where $|G|$ is inverted. 

Indeed,  the $\mathbb{Z}[1/|G|]$-local version of \Cref{endformula} for complex-oriented Borel-equivariant
 theories appears in \cite[Thm.~3.3]{HKR00}. For certain Borel-equivariant theories, variants of the second assertion  appear in \cite{GS99}.
\end{remark} 

We will require some preliminaries first. 
Recall that $\mathcal{S}_G$ denotes the (ordinary) category of $G$-spaces and
$\mathcal{S}$ denotes the category of spaces. 

Let $\sF$ be a family of subgroups of $G$. 
For every $G$-space $X$ with isotropy in $\sF$, we consider
the $G$-space $M(X) \stackrel{\mathrm{def}}{=}\bigsqcup_{H \in \sF} G/H \times X^H $
and the natural map $M(X) \to X$ of $G$-spaces which on the summand $G/H \times
X^H \to X$ induces the inclusion of spaces $X^H \to X$. 

We now observe that the construction $X \mapsto M(X)$ arises from an adjunction. 
We have functors
\[ 
(L, R): \prod_{H \in \mathcal{F}} \mathcal{S} \rightleftarrows \mathcal{S}_G.  \]
Here the left adjoint  $L$ 
sends a family of spaces $\left\{Y_H\right\}_{H \in \sF}$ to the $G$-space
$\bigsqcup_{H \in \sF} G/H \times Y_H$. 
The right adjoint $R$ sends a 
$G$-space $X$ to the family of spaces $\left\{X^H\right\}_{H \in \sF}$ given by
taking fixed points. 
The composite is given by $LR(X) \simeq M(X)$, and the map $M(X) \to X$ is the
counit. 

By general theory (compare the discussion in \cite[Example 3.15]{Dug09}), for any 
$X \in \mathcal{S}_G$, we have an augmented simplicial 
diagram in $\mathcal{S}_G$ (a form of the bar construction)
obtained by applying the unit and counits of the adjunction,
\begin{equation} \label{Mresolution}  \triplearrows \dots M(M(X))
\rightrightarrows M(X)   \to X.
\end{equation}
Furthermore, we have the following two standard properties of
\eqref{Mresolution}: 
\begin{enumerate}
\item  
The augmented simplicial diagram \eqref{Mresolution} admits a splitting after applying $H$-fixed points for any 
subgroup $H \in \sF$ (indeed, after applying the right adjoint $R$). 
\item If $X$ belongs to the image of $L$, then \eqref{Mresolution} admits a
splitting in $\mathcal{S}_G$. 
\end{enumerate}

\begin{prop} 
For any $G$-CW complex $X \in \mathcal{S}_G$ with isotropy in $\sF$, the augmented simplicial diagram 
\eqref{Mresolution}
gives a simplicial resolution of $X$ in $\mathcal{S}_G$, i.e., the map
$|M^{\circ \bullet + 1}(X)|\to X$ is a weak equivalence.
\end{prop} 
\begin{proof} 
To see that \eqref{Mresolution} is a simplicial resolution in $\mathcal{S}_G$, it suffices to show that 
it becomes a simplicial resolution after taking 
$H$-fixed points for each $H \leq G$. For $H \in \sF$, taking $H$-fixed points turns
\eqref{Mresolution} into a split augmented simplicial diagram which is therefore
a resolution. For $H \notin \sF$, taking $H$-fixed points gives $\emptyset$ on
both sides. 
\end{proof}

For the next result, we will need to use the notion of \emph{nilpotent}
augmented cosimplicial diagrams in $\GSpec$. 
Given an augmented cosimplicial diagram $X^\bullet$ in $\GSpec$, we say that $X^\bullet$ is
\emph{nilpotent} if $X^{-1} \simeq \mathrm{Tot}(X^\bullet)$ (i.e, it is a
limit diagram) and the associated $\mathrm{Tot}$ tower is quickly
converging \cite[Sec. 3]{Mat15}. 
Nilpotent augmented cosimplicial diagrams form a thick subcategory of all
cosimplicial diagrams containing the split ones. 

\begin{prop} 
\label{endnilpotent}
For $X$ a finite $G$-CW complex with isotropy in $\sF$ and for any $G$-spectrum
$E$, 
the augmented cosimplicial diagram in $\GSpec$
\begin{equation} \label{MEdiag} F(X_+, E) \to F( M X_+, E) \rightrightarrows F(
MMX_+, E) \triplearrows \dots \end{equation}
is nilpotent. 
\end{prop} 
\begin{proof} 
The construction 
\eqref{MEdiag} 
takes finite homotopy colimits in $X$ to finite homotopy limits of augmented
cosimplicial diagrams. 
When $X = G/H$ for $H \in \sF$, \eqref{Mresolution} is split (hence nilpotent), and therefore so
is \eqref{MEdiag}. Since every finite $G$-CW complex with isotropy in $\sF$ is a finite homotopy
colimit of copies of $G/H, H \in \sF$, the result now follows. 
\end{proof} 

\begin{prop} 
\label{invGsplit}
Let $X$ be a finite $G$-CW complex with isotropy in $\sF$. 
Then the induced augmented simplicial diagram $\Sigma^\infty_+ M^{\circ \bullet +
1}(X) [1/|G|]$ 
obtained by applying $\Sigma^\infty_+(-)[1/|G|]$ to \eqref{Mresolution}
admits a splitting in the homotopy category $\ho(\GSpec)$ (i.e., the associated
chain complex is chain contractible). 
\end{prop} 
\begin{proof} 
Recall that, if $X, Y$ are finite $G$-CW complexes, 
we have the basic formula (cf. \cite[Lemma 3.6]{HKR00})
\[
\pi_0 \Hom_{\GSpec}(\Sigma^\infty_+ X, \Sigma^\infty_+ Y)[1/|G|] \simeq 
\prod_{H } \pi_0 \Hom_{\mathrm{Sp}}( \Sigma^\infty_+ X^H, \Sigma^\infty_+
Y^H)[1/|G|]^{W_H},
\]
where in the product, $H$ ranges over a set of representatives for conjugacy classes of subgroups of
$G$ with $W_H$ the associated Weyl group. 
We need to show that for any finite $G$-CW complex $Y$, the augmented simplicial
abelian group $\pi_0 \Hom_{\GSpec}( \Sigma^\infty_+ Y, \Sigma^\infty_+ M^{\circ
\bullet + 1} X)[1/|G|]$ admits a chain contraction, functorially in $Y$. 
Using the formula above, and averaging a chain homotopy over the Weyl groups, it
suffices to see that for each $H \leq G$, 
$\pi_0 \Hom_{\Spec}( \Sigma^\infty_+ Y^H, (\Sigma^\infty_+ M^{\circ
\bullet + 1} X)^H)[1/|G|]$ admits a chain contraction, functorially in $Y$. 
It thus suffices to see that
after taking $H$-fixed points for any $H \leq G$, \eqref{Mresolution} admits a
splitting; when $H \in \sF$ we saw this above, and for $H \notin \sF$ everything
is empty. 
\end{proof}

We now return to the main result of this section.
We begin with an elementary remark. 
\begin{remark} 
\label{Fapproximable}
Let $M: \mathcal{S}_G^{op} \to \mathrm{Ab}$ be a functor which preserves finite
coproducts. Given a family $\sF$ of subgroups of $G$, we say that $M$ is \emph{$\sF$-approximable}
if $M(X) \to \varprojlim_{G/H \in \sOGF^{op}} M(X \times G/H)$ is an
isomorphism. Equivalently, if $S = \bigsqcup_{H \in \sF} G/H$, then 
$$M(X) \to M(X \times S) \rightrightarrows M(X \times S \times S)$$ is an
equalizer diagram. 

Let $F: \mathcal{S}^{op} \to \mathrm{Ab}$ be a functor which preserves finite
coproducts. For a subgroup $H \leq G$, let $F_H: \mathcal{S}_G^{op} \to \mathrm{Ab}$ be the functor which
sends $X \in \mathcal{S}_G$ to $F(X^H)$. Then for any family $\sF$
containing $H$, $F_H$ is $\sF$-approximable, as one sees from the above
equalizer diagram. 
\end{remark} 
\begin{proof}[Proof of \Cref{endformula}]
Let $X$ be a finite $G$-CW complex. Suppose first that $X$ has isotropy in
$\sF$ (but $E$ is arbitrary). 
Consider the augmented cosimplicial resolution \eqref{Mresolution}.
\[ F(X_+, E) \simeq \mathrm{Tot}( F( MX_+, E) \rightrightarrows F(MMX_+, E)
\triplearrows  \dots ) \]
which we saw is nilpotent in \Cref{endnilpotent}. 
If we take $G$-fixed points, we obtain an associated Tot-spectral sequence
for $E_G^*(X)$. 
Moreover, after inverting $|G|$ the augmented cosimplicial resolution 
admits a splitting (\Cref{invGsplit}).

Unwinding the definitions, we find 
that the map $M(X) \to X$ of $G$-spaces induces the map 
on $E_G^*$-cohomology considered above
\[ E_G^*(X) \to \prod_{H \in \sF} (E^H)^*( X^H).   \]
Furthermore, the lift of this map to a map $E_G^*(X) \to\int_{\sOGF} (E^H)^*(X^H)$ 
is exactly the edge-map of the $\mathrm{Tot}$-spectral sequence. 

We now argue similarly as in the proofs of \Cref{thm:gen-gen-artin} and
\Cref{thm:gen-gen-fiso}. 
Since the tower is nilpotent, we have a 
horizontal vanishing line at a finite stage (\cite[Prop. 3.12]{Mat15}). 
Furthermore, because the augmented cosimplicial diagram
admits a splitting in the homotopy category
after inverting $|G|$, the $E_2$-page of the spectral sequence satisfies
$E_2^{s, t}[1/|G|] =0$ for $s > 0$. 
We note that for each $s$, there must exist a uniform power of $|G|$, say
$|G|^N$, such that $|G|^N E_2^{s,\ast} =0$. If not, 
we could replace the $G$-spectrum $E$ by a product of suspensions of $E$ in such
a manner that would contradict 
$E_2^{s, t}[1/|G|] =0$ for $s > 0$. 

In view of the collapse of the spectral sequence at a finite stage and because
the terms $E_2^{s,t}$ are bounded $|G|$-power torsion for $s > 0$, 
it follows that 
the map $E_G^*(X) \to  \int_{\sOGF^{op}} (E^H)^*(X^H)$
(the edge homomorphism) has kernel and cokernel annihilated by a power of $|G|$, and that,
in the presence of a suitable multiplicative structure on $E$,
the map is an $\cN$-isomorphism integrally and an $\cF_p$-isomorphism after
localizing at a prime $p$. 
This completes the proof when $X$ is finite and has isotropy in $\sF$ (and does not use that
$E$ is $\sF$-nilpotent).

We now prove \Cref{endformula} for arbitrary finite $X$. 
For any $G$-space $X$, we define $T_{\sF}E_G^*(X)$ as the end on the
right-hand-side of \eqref{endcompmap}, so that
we have a natural map $E_G^*(X) \to T_{\sF}E_G^*(X)$ which we need to prove is
an isomorphism after inverting $|G|$ and an $\cF_p$-isomorphism after
localizing at $p$.

Consider the diagram:
\[ \xymatrix{
E_G^*(X) \ar[d]^{\pi_1}  \ar[r]^{\phi_1} &  T_{\sF} E_G^*(X) \ar[d]^{\pi_2} \\
\lim_{
\sOGF^{op}}E_G^*(X \times G/H)  \ar[r]^{\phi_2} &  \lim_{
\sOGF^{op}}T_{\sF} E_G^*(X \times G/H)  
}.\]
Here in the inverse limits, $G/H$ ranges over $\sOGF^{op}$.

We let $\mathcal{W}$ be the class of morphisms $f \colon A \to B$ of commutative rings 
such that $\ker f, \mathrm{coker} f $ are annihilated by a power of $|G|$ 
and such that for each prime number $p$, $f_{(p)}$ is a uniform $\cF_p$-isomorphism. 
 We note that this condition
 implies that $f$ is a uniform $\cN$-isomorphism. 
 We need to show that $\phi_1 \in \mathcal{W}$.

By considering each of the individual terms $(E^H)^*(X^H)$, we see easily that 
$\pi_2$ is actually an isomorphism (\Cref{Fapproximable}). 
Thus, it suffices to show that $\phi_2 \circ \pi_1 \in \mathcal{W}$.
Moreover, $\phi_2$ is a finite inverse limit of the maps 
$E_G^*(X \times G/H) \to T_{\sF} E_G^* ( X \times G/H)$, which belong to
$\mathcal{W}$, by the case of \Cref{endformula} that we have already proved.
As $\mathcal{W}$ is closed under finite limits by \Cref{Wlimit} below, it
follows that $\phi_2 \in \mathcal{W}$.  
Finally, 
$\pi_1 \in \mathcal{W}$ by \Cref{thm:gen-gen-artin} and
\Cref{thm:gen-gen-fiso}. Thus, $\phi_2 \circ \pi_1$ and therefore $\phi_1$
belong to $\mathcal{W}$ and we are done. 
\end{proof}

We used the following elementary algebraic lemma. 
\begin{lemma} 
\label{Wlimit}
Fix a positive integer $m$.
Consider  the class $\mathcal{W}$ of morphisms $f \colon A \to B$ of commutative rings which
have the property that: 
\begin{enumerate}
\item The kernel and cokernel $f$ are both annihilated by a power of $m$ (in
particular, $f \otimes \mathbb{Z}[1/m]$ is an isomorphism).  
\item For each prime number $p$, $f_{(p)}$ is a uniform $\cF_p$-isomorphism. 
\end{enumerate}
Then $\mathcal{W}$ is closed under finite limits. 
\end{lemma} 
\begin{proof} 
The terminal isomorphism $0=0$ is evidently in $\mathcal{W}$, so it suffices to show that if we 
have
fiber product diagrams
of commutative rings
\[ \xymatrix{
A \ar[d] \ar[r] &  A_0 \ar[d] \\
A_1 \ar[r] &  A_{01},
} \quad  \quad 
\xymatrix{
B \ar[d] \ar[r] &  B_0 \ar[d] \\
B_1 \ar[r] &  B_{01},
}
\]
and a natural map of diagrams between them 
such that $\phi_0 \colon A_0 \to B_0, \phi_1\colon A_1 \to B_1, \phi_{01}\colon A_{01} \to B_{01} $ belong to
$\mathcal{W}$, then $\phi\colon A \to B$ belongs to $\mathcal{W}$. It follows easily via a diagram-chase that the kernel and cokernel of $\phi $
are annihilated by a power of $m$, so it
suffices to show that for each $p$, $\phi_{(p)}$ is a uniform $\cF_p$-isomorphism. 

Without loss of generality, we can assume that we are already localized at $p$. 
It follows easily that $\ker(\phi)$ is nilpotent. 
Suppose given $y \in B$, with images $y_0, y_1, y_{01}$ in $B_0, B_1, B_{01}$,
respectively. After replacing $y$ by a suitable $p$th power (chosen uniformly for all $y$), we can assume that $y_0,
y_1$ belong to the image of $\phi_0, \phi_1$. That is, there exist  $x_0, x_1
\in A_0, A_1$ which map to $y_0, y_1$. However, $x_0, x_1$ need not have the
same image in $A_{01}$. Let $\overline{x}_0, \overline{x}_1$ be the images in
$A_{01}$. Then $z = \overline{x}_0 - \overline{x}_1$ is $p$-power torsion 
and
nilpotent, both of uniform exponent, as it maps to zero in $B_{01}$ under $\phi_{01}$. It follows that 
$$\overline{x}_1^{p^n} = (\overline{x}_0 + z )^{p^n} = \overline{x}_0^{p^n}$$
for $n \gg 0$, by the binomial theorem. In particular, $x_0^{p^n}, x_1^{p^n}$
have equal images in $A_{01}$, which implies that $y^{p^n}$ belongs to the
image of $\phi$ as desired. One sees that $n \gg 0$ can be chosen uniformly. 

\end{proof}

If we are only interested in the underlying \emph{variety} of $E_G^*(X)$, we
can make a further simplification. 
Consider the 
functor 
\[ K \colon \sOG \times \sOG^{op} \to \mathrm{Ring}, \quad 
K(G/H, G/L) = (E^L)^*( \pi_0 X^H),
\]
defined in a similar fashion as 
in \eqref{bifunctorformula}. 
Given a family of subgroups $\sF$, we have a similar natural map 
\begin{equation}
\label{endcompmap2}
E_G^0(X) \to \int_{\sOGF^{op}} (E^H)^0( \pi_0 X^H).
\end{equation}

We can also study the properties of this map. 
This recovers \cite[Theorem 2.4]{GS99} for complex-oriented theories. 
Taking $E$ to be 2-periodic mod $p$ cohomology, we can recover 
\cite[Theorem 6.2]{Qui71b} (at least for finite $G$-CW complexes).
\begin{cor} 
Suppose that $E$ is a homotopy commutative $A_\infty$-algebra in $\GSpec$
and $E$ is $\sF$-nilpotent. Suppose that $\pi_0^G E$ is a noetherian ring
and for any $H \leq G$ and $k\in\mathbb{Z}$, $\pi_k^H E$ is a finitely generated $\pi_0^G E$-module. Then for any finite $G$-CW complex $X$,
\eqref{endcompmap2} is a $\mathcal{V}$-isomorphism in the sense of \cite{GS99}. 
Moreover, the maps
\[ 
\int^{\sOGF}
\mathrm{Ring}\left((E^H)^0(X^H), \cdot \right)
\to 
\mathrm{Ring} \left( 
\int_{\sOGF^{op}} (E^H)^0( \pi_0 X^H), \cdot \right)
\leftarrow \mathrm{Ring}( E_G^0(X), \cdot)
\]
are isomorphisms when restricted to algebraically closed fields. 
\end{cor} 
\begin{proof} 
We observe that the edge maps
\[ (E^H)^0(X^H) \to (E^H)^0(\pi_0 X^H)  \]
in the associated Atiyah-Hirzebruch spectral sequences are $\mathcal{V}$-isomorphisms because each $X^H$ is a finite CW complex. 
Note that both sides are finitely generated $\pi_0^G E $-modules under our
hypotheses, and the map is surjective with nilpotent kernel. 
Then the result follows from Quillen's work as in 
\Cref{prop:stratification} in view of \Cref{endformula}. 
\end{proof}

\section{Defect bases and $\sF$-split spectra}\label{sec:split}
\subsection{Classical defect bases and $\sF$-split spectra}\label{sec:defect-base}
Classical induction theory centers around the notion of a defect base. To
define this, we will first need the following:
\begin{prop}[{Cf. \cite[Lemma 3.11]{Gre71}}]\label{prop:defect-base-exists}
  Let $R(-)$ be a Green functor for the group $G$, see \cite{Gre71,Web00}. 
  Then there is a unique minimal family $\sF$ such that  
	  the map \begin{equation}\label{eq:defect-base-surjective}
		 \Ind_{\sF}^G\colon \bigoplus_{H\in \sF }R(H)\rightarrow R(G)
	\end{equation} 
	is surjective.
\end{prop}
\begin{proof}
	It suffices to show that if $\sF_1, \sF_2$ are families such that 
$\Ind_{\sF_1}^G$ and $\Ind_{\sF_2}^G$ are surjective (equivalently, have image
containing the unit), then the same holds for
$\sF_1 \cap \sF_2$. This is a straightforward exercise with the double coset
formula \cite[Axiom G4, p. 44]{Gre71}. \end{proof}

\begin{defn}[{See \cite[Sec. 3]{Gre71}}]\label{def:defect-base}
	Let $R(-)$ be a Green functor for the group $G$. The \emph{defect base of
	$R$} is the minimal family $\sF$ of subgroups of $G$ such that 
	the map $\Ind_{\sF}^G$ above 
\eqref{eq:defect-base-surjective}
	is surjective. 
	The \emph{defect base} of a $G$-ring spectrum $R$ is the defect base of the Green functor $\pi_0^{(-)}R$. 
\end{defn}


To relate the notion of a defect base of a $G$-ring spectrum $R$, which only depends on $R$ through $\pi_0^{(-)} R$, to the derived defect base, we have the following: 
\begin{prop}\label{prop:characterize-split}
	Let $R$ be a $G$-ring spectrum. For a family of subgroups $\sF$ of $G$, consider the sum of the induction maps
	\( \Ind_{\sF}^G\colon \bigoplus_{H\in \sF} \pi_0^H R \rightarrow \pi_0^G R.\) 
	Then the following are equivalent:
	\begin{enumerate}
		\item\label{en:alg-1} The map $\Ind_\sF^G$ is surjective. 
		\item\label{en:alg-2} The product of the restriction maps \[
		\prod_{H\in \sF} \Res_H^G \colon R\rightarrow \prod_{H\in \sF} F(G/H_+,R)
		\]
		splits in $\SpG$.
		\item\label{en:alg-3} The $G$-spectrum $R$ is $\sF$-nilpotent and $\widehat{H}^*_\sF(\Gpi R)=0$.
		\item\label{en:alg-4} The $G$-spectrum $R$ is $\sF$-nilpotent and $\exp_{\sF}(R)\leq 1$.
	\end{enumerate}
\end{prop}
\begin{proof}
	First we will prove \eqref{en:alg-1}$\implies$\eqref{en:alg-4}. By
	assumption, there is for each $H\in \sF$  an element $m_H\in \pi_0^H R$ such that 
	\begin{equation}\label{eq:sum-of-inds}
		1=\sum_{H\in \sF}\Ind_H^G m_H \in \pi_0^G R.
	\end{equation} 
	The element $m_H\in \pi_0^H R$ is represented by a $G$-map $R\rightarrow
	F(G/H_+,R)$ and $\Ind_H^G m_H$ is obtained by postcomposing with the
	projection $F(G/H_+,R)\simeq R\wedge G/H_+\rightarrow R$. Assembling
	these together,
we find that the composite map
\[ R \xrightarrow{\prod_H m_H} \prod_{H \in \sF} F(G/H_+, R) \to  R   \]
is homotopic to the identity. This retraction implies that $R$ is
$\sF$-nilpotent with $\sF$-exponent $\leq 1$, proving (4).

The equivalence of \eqref{en:alg-2} and \eqref{en:alg-4} follows from
\Cref{prop:exponent} because $\prod_{H\in\sF}F (G/H_+,R)\simeq F(\sk_0 E\sF_+,R)$ with our preferred model for $E\sF$.

Now we will show \eqref{en:alg-2}$\implies$\eqref{en:alg-3}.  By assumption,
$R$ is a retract of the product spectrum $\prod_{H\in \sF}F (G/H_+,R)$, which
is $\sF$-nilpotent. Since $\sFNil$ is closed under finite products and retracts we see that $R$ is $\sF$-nilpotent.

	We will now show $\widehat{H}^*_\sF(\pi_*^{(-)} R)=0$. Since the
	Amitsur-Dress-Tate cohomology groups are bigraded modules over the algebra
	in bidegree $(0,0)$, it suffices to prove this claim when the bidegree is
	$(0,0)$. Furthermore, since $R$ is a retract of $\prod_{H\in \sF}
	F(G/H_+,R)$, it suffices to show that for each $H \in \sF$, $
	\widehat{H}^0_\sF(\pi_0^{(-)}  F(G/H_+,R)) = 0,$
and by naturality again it suffices to consider the case $R = S^0 $. 
However, in $\pi_0^G F(G/H_+, S^0)$, the unit is a norm from $\pi_0^H F(G/H_+,
S^0)$, so that it vanishes in the $\sF$-Tate cohomology, which is therefore
zero. 

Next, we prove \eqref{en:alg-3}$\implies$\eqref{en:alg-1}. Since $R$ is
$\sF$-nilpotent, the $\sF$-homotopy limit spectral sequence converges to
$\pi_*^G R$. The vanishing of the Amitsur-Dress-Tate cohomology groups implies
that this spectral sequence collapses at $E_2$ and the edge map induces an isomorphism $\pi_*^G R\cong \lim_{\sOGF^\op} \pi_*^H R$. Combining this with the identification of zeroth cohomology group from  \eqref{eq:tate} we obtain \[\widehat{H}^0_\sF(\pi_0^{(-)} R)\cong \left(\lim_{\sOGF^\op} \pi_0^H R\right) /(\Im \Ind_\sF^G)\cong \pi_0^GR/(\Im \Ind_\sF^G).\] Since these groups are zero by assumption, $\Ind_\sF^G$ is surjective.

\end{proof}

\begin{defn}
	We will say that a $G$-ring spectrum $R$ is $\sF$\emph{-split} if it
	satisfies any of the equivalent characterizations of
	\Cref{prop:characterize-split}. More generally, we will say a $G$-spectrum $M$ is $\sF$\emph{-split} if its endomorphism ring $\mathrm{End}(M)$ is $\sF$-split.
\end{defn}

\begin{remark}
	It follows from the definitions that the defect base of a $G$-ring spectrum $R$ is the smallest family $\sF$ such that $R$ is $\sF$-split. 
\end{remark}

\begin{remark}
	The $\sF$-split condition can be used to test for  projectivity and flatness
	(cf., \cite[Rem.~3.5.2]{HKR00}). For example, since $\KG$ is $\sF$-split for
	the family $\sF$ of Brauer elementary subgroups we know that for a
	$G$-spectrum $X$, $KU_*^G(X)$ is torsion-free if and only if $KU_*^H(X)$ is torsion-free for each $H\in \sF$. 
\end{remark}



\begin{prop}\label{prop:defect-base-for-rg}
  Suppose that $R$ is a $G$-ring spectrum such that $\pi_0^{(-)}R$ is isomorphic to the complex representation ring functor. Then the defect base of $R$ is the family $\sF$ of Brauer elementary subgroups of $G$, i.e., subgroups which are products of $p$-groups with cyclic subgroups of order prime to $p$ for some prime $p$.
\end{prop}
\begin{proof}
	This purely algebraic claim about the representation ring Green functor is a combination of Brauer's theorem and its converse due to J.~Green \cite[\S 11.3]{Ser77}.
\end{proof}

We now give two important cases in which the derived defect base and the defect
base automatically coincide.
Recall that a $G$-spectrum $M$ is \emph{connective} if, for every subgroup $H$ of $G$, $\pi_i^H M=0$ if $i<0$. 
\begin{prop}\label{prop:connective-theories}
	Suppose that $R$ is a connective $G$-ring spectrum and $\sF$ is a family of subgroups of $G$. Then $R$ is $\sF$-nilpotent  if and only if $R$ is $\sF$-split. In particular, the defect base and the derived defect base of $R$ coincide.
\end{prop}
\begin{proof}
	Clearly any $\sF$-split spectrum is $\sF$-nilpotent. For the other
	direction, suppose that $R$ is $\sF$-nilpotent, so the $\sF$-homotopy
	colimit spectral sequence converges to $\pi_*^G E\sF_+\wedge R\cong \pi_*^G R$
	by \Cref{prop:Fnil-implies-acyclic}. The connectivity assumption implies
	that $E^2_{0,0}$ is the only term contributing to $\pi_0^G R$ in this
	spectral sequence. Hence we obtain an isomorphism \[\pi_0^G R \cong E^2_{0,0}\cong H_0^G(E\sF_+;\pi_0^{(-)} R)\cong \colim_{\sOGF} \pi_0^H R.\] Since the $E_2$-edge map is an isomorphism on $\pi_0^G$, the $E_1$-edge map $\Ind^G_\sF$ is surjective and hence $R$ is $\sF$-split. 
\end{proof}

\begin{prop}\label{prop:split-after-inverting}
	Let $R$ be a $G$-ring spectrum. If $n(\sF)\in \pi_0^G R$ is a unit, then $R$ is $\sF$-nilpotent if and only if $R$ is $\sF$-split. In particular, the defect base and the derived defect base of $R$ coincide.
\end{prop}
\begin{proof}
	If $R$ is $\sF$-split, then it is $\sF$-nilpotent by definition. On the other
	hand, if $R$ is $\sF$-nilpotent then, by \Cref{prop:characterize-split}.\eqref{en:alg-3}, $R$ is $\sF$-split if and only if the Amitsur-Dress-Tate
	cohomology groups $\widehat{H}_{\sF}^*(\pi_*^{(-)} R)$ vanish. Now since
	$n(\sF)\in \pi_0^G R$ acts on these groups simultaneously by a unit and by zero by \Cref{prop:torsion}, we see that they must be zero.
\end{proof}


\subsection{Brauer induction theorems}
If we know the defect base of a $G$-ring spectrum we obtain an upper bound on the \emph{derived} defect base. We now include results
that enable us to go the other direction: if we know the derived defect base we
obtain an upper bound on the defect base.
\begin{prop}\label{prop:hyper-induction}
	Suppose that $R$ is an $\sF$-nilpotent $G$-ring spectrum. Let
	$\overline{\sF}\supset \sF$ be a family of subgroups satisfying the following
	condition: for each prime $p$,  
	if $H\leq G$ that fits into a short exact sequence
	\begin{equation} \label{nonsplitses} e\rightarrow N\rightarrow H\rightarrow H/N\rightarrow e
\end{equation} 
	where $N \in \sF$ and $H/N$ is a $p$-group,  then $H
	\in \overline{\sF}$. 
	Then $R$ is $\overline{\sF}$-split.
\end{prop}
\begin{proof}
	Since $\sF\subseteq \overline{\sF}$ and $R$ is $\sF$-nilpotent, $R$ is also $\overline\sF$-nilpotent. By \cite[Prop.~1.6]{Dre75b} we can find an $x\in \Im \Ind_{\overline{\sF}}^G$ and $y\in \ker\Res_\sF^G$ such than $x+y=1\in \pi_0^G R$. Since $y$ is nilpotent by \Cref{thm:gen-gen-fiso}, $x$ must be a unit and $\Ind_{\overline{\sF}}^G$ must be surjective as desired.  
\end{proof}

\begin{remark}
	Observe that any family $\overline{\sF}$ satisfying the assumption of
	\Cref{prop:hyper-induction} necessarily contains all $p$-Sylow subgroups of
	$G$. So when $R$ is a Borel-equivariant theory, the bound on the defect base
	of $R$ given by \Cref{prop:hyper-induction} will provide no information (cf.
	\Cref{prop:borel-sphere} below).
\end{remark}

When all of the subgroups in our given family $\sF$ are \emph{abelian}, and they often are, we can more explicitly identify a family $\overline{\sF}$ satisfying the assumption of \Cref{prop:hyper-induction}.
\begin{prop}\label{prop:hyper-ind-simplifies}
	Let $\sF\subset\sAb$ be a family of abelian subgroups of $G$. Then the set $\overline{\sF}$ of subgroups of the form \(G^\prime=H^\prime\rtimes P,\) where $P\in \bb{\sAll}$ is a $p$-group for some prime $p$ and $H^\prime\in \sF[p^{-1}]$ is a subgroup in $\sF$ of order prime to $p$, is a family of subgroups satisfying the assumption of \Cref{prop:hyper-induction}. 
\end{prop}

\begin{proof}
Consider the family $\overline{\sF'}$ of subgroups  $H \leq G$ that 
fit into a short exact sequence as in \eqref{nonsplitses} with $N \in \sF$ and
$H/N$ a $p$-group (for some $p$). We will show that any such $H$ belongs to
$\overline{\sF}$, so that $\overline{\sF} = \overline{\sF'}$. 
By \Cref{prop:hyper-induction}, this will suffice for the result.

Note first that $N = N_1 \times N_2$ where $N_2$ is a $p$-group and $p \nmid
|N_1|$, since $N$ is abelian by assumption. 
Now $N_1 \leq N$ is a characteristic subgroup and is therefore
normal in $H$. 
Therefore, we obtain a new short exact sequence
\[ e \to N_1 \to H \to H/N_1 \to e,  \]
where $N_1 \in \sF$ has order prime to $p$ and $H/N_1$ is a $p$-group. 
The Schur-Zassenhaus theorem now implies that this extension splits.
Consequently, $H \in \overline{\sF}$ as desired. 
\end{proof}

\subsection{Applications of $\sF$-split spectra}\label{sec:f-split-applications}
If a $G$-ring spectrum $R$ is $\sF$-split and $M$ is an $R$-module, then the Amitsur-Dress-Tate cohomology groups of $M$ vanish by \Cref{prop:characterize-split}. So under these hypotheses the $\sF$-homotopy limit and colimit spectral sequences collapse at $E_2$ onto the zero line and we obtain the following integral form of \Cref{thm:gen-gen-artin}:
\begin{thm}\label{thm:split-brauer}
	Let $R$ be an $\sF$-split $G$-ring spectrum and let $X$ be a $G$-spectrum.
	Then for each $R$-module $M$, each of the following maps
	\begin{equation}
	 \colim_{\sOGF} M^*_H(X)\xrightarrow{\Ind_\sF^G} M^*_G(X)\xrightarrow{\Res_\sF^G} \lim_{\sOGF^\op} M^*_H(X)
	\end{equation}
	\begin{equation}
	 \colim_{\sOGF} M_*^H(X)\xrightarrow{\Ind_\sF^G} M_*^G(X)\xrightarrow{\Res_\sF^G} \lim_{\sOGF^\op} M_*^H(X)
	\end{equation}
	is an isomorphism.
\end{thm}

We will now show how the $\sF$-split condition fits into Balmer's theory of
descent for triangulated categories with a monoidal product which is exact in
each variable \cite{Bal12}. For this we suppose that $R$ is an $E_2$-$G$-ring
spectrum so the $\infty$-category $\mod(R)$ of \emph{structured} $R$-modules is monoidal
\cite[Cor.~5.1.2.6]{Lur14}. Moreover, the monoidal product commutes with
homotopy colimits in each variable; in particular, tensoring with a fixed
module is an exact functor. It follows that the homotopy category
$\ho(\mod(R))$ of $R$-modules is an idempotent-complete triangulated category 
with a monoidal structure that is exact in each variable (cf.\  \cite{Man12}).

Now given a family of subgroups $\sF$, let $X=\coprod_{H\in \sF} G/H$ and
$A=F(X_+,R)$. The $R$-algebra structure on $A$ defines a monad $T$ on $\ho(\mod(R))$ where $TM=A\wedge_R M$. The forgetful functor $U_T$ from $T$-algebras in $\ho(\mod_R)$ to the underlying category $\ho(\mod_R)$ admits a right adjoint $F_T$. This defines a comonad $C=F_TU_T$ on the category of $T$-algebras and we let $\Desc_{R}(\sF)$ denote the category of $C$-coalgebras in $T$-algebras. The free algebra functor $F_T$ canonically lifts to a functor \[ Q_\sF\colon \ho( \mod(R)) \rightarrow \Desc_{R}(\sF)\] and we will say that $R$ \emph{effectively descends along $\sF$} if $Q_\sF$ is an equivalence of categories.

\begin{prop}\label{prop:balmer}
	Let $R$ be an $E_2$-$G$-ring spectrum and $\sF$ a family of subgroups. Then the following are equivalent:
	\begin{enumerate}
		\item $R$ effectively descends along $\sF$.
		\item $R$ is $\sF$-split.
	\end{enumerate}
\end{prop}
\begin{proof}
	 According to \cite[Cor.~3.1]{Bal12}, $R$ effectively descends along $\sF$
	 if and only if $A\wedge _R( -  )$ is faithful. Moreover, this condition is
	 equivalent to the unit $R\rightarrow A$ admitting a retraction in
	 $\ho(\mod(R))$. Indeed, if we have a retraction then clearly the functor is
	 faithful; the other implication is \cite[Prop.~2.12]{Bal12} but we also include a more direct argument for the
	 special case at hand: To see that $R\to A$ admits a retraction, we need to argue that the map
	 $\mathrm{hofib}(R\to A)\to R$ is zero. We can check this after smashing with $A$, and hence it suffices to see
	 that $A\cong R\wedge A\to A\wedge A$ admits a retraction; such is furnished by the multiplication map.
	 
	 Finally, since the unit
	 map \[R\rightarrow A\simeq \prod_{H\in \sF} F(G/H_+,R)\] is the product of the restriction maps, we see that $R$ is $\sF$-split if and only if $R$ effectively descends along $\sF$.
\end{proof}

\begin{example}\label{ex:ksigma_3}
	In \Cref{prop:defect-base-for-rg} we saw that $\KG$ is split for the family of Brauer elementary subgroups. When $G=\Sigma_3$, then the family of Brauer elementary subgroups is the family $\sC$ of cyclic subgroups and the category $\sO(\Sigma_3)_\sC$ has a very simple form. So, in this case, one can more explicitly identify the target of the restriction isomorphism in \Cref{thm:split-brauer}.

	For this purpose we fix Sylow subgroups $C_2$ and $C_3$ of $\Sigma_3$ and
	let $W_{\Sigma_3} C_3\cong \bZ/2$ denote the Weyl group of the 3-Sylow
	subgroup. We leave it to the reader to show (e.g., using
	\Cref{prop:pullbacks}) that for any $\Sigma_3$-spectrum $X$, $\KG_{\Sigma_3}^*(X)$ fits into the following pullback diagram:
	\begin{equation}\label{eq:kg-sigma3}
	\xymatrix{
	 \KG_{\Sigma_3}^*(X) \ar[d]_{\Res_{C_2}^{\Sigma_3}}\ar[r]^{\Res_{C_3}^{\Sigma_3}} & \KG_{C_3}^*(X)^{\bZ/2}\ar[d]^{\Res_e^{C_3}} \\ 
	 W_{C_2}^*(X) \ar[r]_{\Res_e^{C_2}} &  \KG_{e}^*(X)^{\Sigma_3}
	 }
	\end{equation}
	Here \[W_{C_2}^*(X):=\KG_{C_2}^*(X)\cap \left(\Res_e^{C_2}\right)^{-1}
	\left({\KG_e^*(X)}^{\Sigma_3}\right).\] 
\end{example}

\subsection{Examples of $\sF$-split spectra; the Borel-equivariant
sphere}\label{sec:f-split-examples}
We will now establish some more of the claims made in \Cref{fig:examples}. 
\begin{prop}\label{prop:sphere}
  Let $T$ be a multiplicatively closed subset of $\bZ \setminus \{0\}$. Then the derived defect base of $S[T^{-1}]$ is equal to its defect base, which is $\sAll$, the family of all subgroups.
\end{prop}
\begin{proof}
	 Since $S[T^{-1}]$ is connective, the derived defect base of $S[T^{-1}]$ is
	 the same as the defect base of $\pi_0^{(-)}S[T^{-1}]\cong A(-)[T^{-1}]$
	 (\Cref{prop:connective-theories}). It suffices to show that the family
	 $\sP$ of proper subgroups is not a defect base. The image of $\Ind_\sP^G$ is generated as a $\bZ[T^{-1}]$-module by those finite $G$-sets whose isotropy is a proper subgroup of $G$. In particular, they have trivial $G$-fixed points. It follows that $[G/G]\in A(G)\otimes \bZ[T^{-1}]$ is not in the image of $\Ind_\sP^G$ and hence $\sP$ is not a defect base.
\end{proof}

Recall from \Cref{fig:families} that for a family $\sF$, the subfamily $\bb{\sF}\subseteq \sF$ is
defined to be the subset of groups in $\sF$ of prime power order. Let
$\sF_{(0)}=\sTriv$ and, for each prime $p$, let  $\sF_{(p)}\subseteq \sF$
denote the subfamily of $\sF$ whose elements are $p$-groups. 
For any
multiplicative subset $T$ of $\bZ\setminus \{0\}$, we let
\(\bb{\sF}[T^{-1}]\) denote the subset of groups in $\bb{\sF}$ whose order is not invertible in $\mathbb{Z}[T^{-1} ]$. 

\begin{defn}
	The \emph{constant} Green functor $R(-)$ associated to a ring $R$ is defined as follows:
	\begin{itemize}
		 \item For each subgroup $H$ of $G$, $R(G/H)$ is the ring $R$.
		 \item The restriction and conjugation maps of $R(-)$ are all identities. 
		 \item For each chain of subgroup inclusions $H<K<G$, $\Ind_H^K$ is multiplication by $|K/H|$.
	\end{itemize}
\end{defn}

\begin{prop}\label{prop:constant-examples}
	Let $R\neq 0$ be a ring and let $T$ be the set of elements in $\bZ$ which are
	invertible in $R$. Then $\bb{\sAll}[T^{-1}]$ is the defect base of both $HR$
	and $\bb{HR}$; $\bb{\sAll}[T^{-1}]$ is also the derived defect base of $HR$.
\end{prop}
\begin{proof}
	Since $HR$ is connective, its derived defect base and its defect base are
	equal (\Cref{prop:connective-theories}). The remaining claims follow from
since $\pi_0^{(-)}HR\cong \pi_0^{(-)}\bb{HR}=R$ is the constant Green functor at $R$. Indeed, the image of $\Ind_\sF^G$ inside $R(G/G)=R$ is the principal ideal generated by $\gcd(\{|G/H|\}_{H\in \sF})$, so $\Ind_\sF^G$ is surjective if and only if this integer is a unit in $R$.
\end{proof}

We now give a particularly deep example of the determination of the derived
defect base of a $G$-ring spectrum. 
We do not know if it is possible to carry this out without the use of the Segal
conjecture. 
\begin{thm}\label{prop:borel-sphere}
	The derived defect base of $\bb{S}$ is equal to its defect base $\bb{\sAll}$. More generally if $T\subseteq \bZ\setminus\{0\}$ is a multiplicatively closed subset, then the derived defect base of $\bb{S[T^{-1}]}$ is equal to its defect base $\bb{\sAll}[T^{-1}]$. 
	Consequently, if $M$ is an $\bb{S[T^{-1}]}$-module, then $\bb{M}$ is $\bb{\sAll}[T^{-1}]$-nilpotent.
\end{thm} 
\begin{proof}

We first bound \emph{above} the derived defect base of $\underline{S[T^{-1}]}$ as
claimed. 
We will prove the following two items: 
\begin{enumerate}
\item 
$\underline{S}$ is nilpotent (even split) for the family $\bb{\sAll}$.
\item If $G$ is a $p$-group for a prime number $p $ invertible in $\mathbb{Z}[
T^{-1}]$, then 
$\underline{S[T^{-1}]}$  is nilpotent (even split) for the family consisting of the trivial
group.
\end{enumerate}
By 
\cite[Prop.~\ref{S-reducetopropersubgroups}]{MNNa}, to see that $\underline{S}$
is $\sAll[T^{-1}]$-nilpotent,
it suffices to show that if 
$G$ is not a $p$-group for some $p$ noninvertible in $\mathbb{Z}[T^{-1}]$, then
$\underline{S}$ is 
nilpotent for the family of proper subgroups. So the two items above are
certainly sufficient. 

Recall that the category of dualizable structured modules over 
$\underline{S[T^{-1}]}$ embeds into $\fun(BG, \mathrm{Perf}( S[T^{-1}]))$ (where the
latter $S[T^{-1}]$ belongs to the category of non-equivariant spectra) by
\cite[Cor.~\ref{S-perfborelequiv}]{MNNa}. 

We now treat the assertions above. 
We begin with the first. 
For each prime $p$ dividing the order of the group $G$, we let $G_p \leq G$
denote a $p$-Sylow subgroup. 
We consider the composite
map \[  \psi_p\colon \underline{S} \xrightarrow{\mathrm{Ind}_{G_p}^{G}}
\underline{(G/G_p)_+}  \xrightarrow{\mathrm{Res}_{G_p}^{G}} \underline{S}, \]
which induces multiplication by $|G/G_p|$ on the underlying spectrum. 
The orders of the $\{|G/G_p|\}$ generate the unit ideal in $\mathbb{Z}$, so
there is a linear combination $\sum_{p \mid |G|} n_p \psi_p$ of the $\psi_p$
which is a self-equivalence $\phi$ of
$\underline{S}$.
We thus get a retraction diagram
\[ \underline{S} \xrightarrow{\Sigma n_p \mathrm{Ind}_{G_p}^{G}} \bigoplus_p
\underline{(G/G_p)_+} \xrightarrow{\phi^{-1} \circ  \Sigma
\mathrm{Res}_{G_p}^{G}}   \underline{S},  \]
which shows that $\underline{S}$
is nilpotent (even split) for the family $\bb{\sAll}$. 
For the second claim, we observe that in this case 
the composite \[\underline{S[T^{-1}]} \xrightarrow{\mathrm{Ind}_e^G} \underline{ G_+[T^{-1}]} \xrightarrow{\mathrm{Res}_e^G}
\underline{S[T^{-1}]}\] is an equivalence.

Finally, we show that the families are minimal as claimed. 
In general, if a prime $p$ is not inverted in $\bb{S[T^{-1}]}$, then
$\bb{S_{(p)}}$ is an $\bb{S[T^{-1}]}$-module.  
It thus suffices to show that if $G $ is a $p$-group, then $\underline{S_{(p)}}$ is
\emph{not} nilpotent for the family of proper subgroups.   
In \Cref{prop:finite-spectra} below, we will show that $\underline{S/p}$ is not nilpotent
for the family of proper subgroups, which will prove the claim. The proof of
\Cref{prop:finite-spectra} relies on the Segal conjecture. 
Note also that the since the derived defect base of $\underline{S[T^{-1}]}$ is
$\bb{\sAll}[T^{-1}]$ and $\underline{S[T^{-1}]}$ is split for this family,
that is also the defect base. 
\end{proof}

\begin{prop}\label{prop:finite-spectra}
	Let $p$ be a prime and $X$ a nontrivial finite $p$-torsion spectrum. Then the derived defect base of $\bb{X}$ is $\sAllp$. 
\end{prop}
\begin{proof}
	By assumption, $\bb{X}$ is an $\bb{S_{(p)}}$-module and hence
	$\sAllp$-nilpotent, so we just need to prove the minimality of this family.
	For this claim, it suffices to show that if $G$ is a $p$-group then $\Phi^G
	\bb{X}\not \simeq *$; in fact, we recall that if $X \in \GSpec$ is nilpotent
	for the family of proper subgroups then $\Phi^G(X) = 0$ (cf. the discussion
	in \cite[Def. 6.12]{MNNa}). 

	Let $i_*: \Spe \to \Spe_G$ denote the left-adjoint functor that sends the sphere to the
	$G$-sphere. 
	The class of spectra $Y$ such that the canonical map $i_* Y\rightarrow
	\bb{Y}$  is an equivalence of $G$-spectra forms a
	thick subcategory $\cC\subseteq \Spe$. Since $G$ is a $p$-group, the Segal
	conjecture proved by Carlsson \cite{Car84d} (in the equivalent form given in \cite[Prop.
	B]{MayMc82}) implies $S/p \in \cC$.
That is, the Segal conjecture implies that the Borel-completion $\underline{S}$
is obtained by an algebraic completion (at the augmentation ideal) of the
$G$-sphere $S$; when we work mod $p$, the map $S/p \to \underline{S}/p$ is an
equivalence as $G$ is a $p$-group. 
	Hence every finite $p$-torsion spectrum
	belongs to $\cC$. Since $X$ is nontrivial and $\Phi^Gi_*X\simeq X$, it follows that \( \Phi^G \bb{X} \simeq \Phi^G i_* X\simeq X\not \simeq *.\)
\end{proof}

\section{Derived defect bases via orientations}\label{sec:oriented}

In this final section, we give the main examples of derived defect bases. All of
these will rely on the use of \emph{orientations} together with thick
subcategory arguments, and a reduction (following Quillen) to the family of
\emph{abelian} subgroups via consideration of the variety of complete flags of a
$G$-representation.
 
\subsection{On $G$-spectra with Thom isomorphisms}

We will now consider how our conditions on a homotopy commutative $G$-ring spectrum simplify when the spectrum is oriented in the following sense (Cf.~\cite[Defn.~2.1, Rem.~2.2, Defn.~3.7]{GrM97}) :
\begin{defn}\label{def:thom-isos}
	Let $R$ be a homotopy commutative $G$-ring spectrum and $V$ an orthogonal
	representation of $G$. Then a \emph{Thom class for
	$V$} with respect to $R$ is a map of $G$-spectra $\mu_{V}\colon S^{V-|V|}\rightarrow  R$
	such that its canonical extension to an $ R$-module map
	\[  R\wedge S^{V-|V|}\xrightarrow{ R \wedge \mu_V}
	 R\wedge  R\xrightarrow{\mu}  R\] is an equivalence.
	 If $V$ is a representation of a subgroup $H \leq G$, then we say that a
	 \emph{Thom class for $V$} with respect to $R$ is a Thom class for $V$ with
	 respect to the $H$-spectrum $\Res^G_H R$. 
	
	Let $\cI$ be a class of representations closed under finite direct sums,
	restriction, and conjugation (e.g., unitary, oriented, $8n$-dimensional
	spin). We will say that $R$ has \emph{multiplicative Thom classes} for $\cI$
	if, for each subgroup $H \leq G$, it has Thom classes for every $H$-representation $V$ in $\mathcal{I}$ and the Thom classes are multiplicative: $\mu_{V\oplus W}=\mu_V\cdot  \mu_W$.

	In this case we define the \emph{oriented Euler class of $V$}, $\chi(V)\in R^{|V|}_H(*)$, to be the following composition \[
	\chi(V)\colon S^{-|V|}\xrightarrow{e_V\wedge S^{-|V|}}
	S^{V-|V|}\xrightarrow{\mu_V} \Res^G_H R. \]
\end{defn}

\begin{remark}
	Thom classes often appear in another guise which we will now describe (cf.~\cite[Rem.~2.2]{GrM97}).	
	Given an orthogonal $H$-vector bundle $V$ on an $H$-space $X$ we have an associated Thom space $TV$. Suppose that we are given a family of isomorphisms of $R^*_H$-modules
	\[
		\phi_V\colon R^*_H(\Sigma^{|V|}(X_+))\rightarrow R^*_H(TV)
	\] 
	which are natural in $X$ and $H$. In the case that $X$ is a point, $TV\simeq S^V$ and we can rewrite $\phi_V$ as an isomorphism
	\[
	\pi_*^H R\cong  \pi_*^H R\wedge S^{|V|-V}
	\] of $\pi_*^H R$-modules. The unit in $\pi_0^H R$ corresponds to a map
	 of $H$-spectra $ \mu_{V}\colon S^{V-|V|}\rightarrow \Res^G_H R$ which extends to an
	 equivalence of $\Res^G_H R$-modules \[\Res^G_H R\wedge S^{V-|V|}\rightarrow 
	 \Res^G_H
	 R.\]  So we see that natural Thom isomorphisms give rise to such Thom classes.
\end{remark}

In terms of the $RO(G)$-graded groups $\pi_\star^G R$, the Euler classes are evidently related to the oriented Euler classes by the following formula:
\[ \chi(V)=e_V\cdot \mu_V.\] Since the Euler classes are in the Hurewicz image
of $\pi_\star^G S$, they are necessarily central and it follows that
$\chi(V^n)=\chi(V)^n$. Since $\mu_V$ is necessarily a unit in $\pi_\star^G R$, we immediately obtain the following:
\begin{lemma}\label{lem:euler-classes}
	Suppose that $R$ is a homotopy commutative $G$-ring spectrum with
	multiplicative Thom classes for a class $\cI$ of representations and $V$
	admits the structure of an $\cI$-representation of $G$. Then, for each positive integer $n$, $\chi(V)^n$ is zero if and only if $e_V^n$ is zero.
\end{lemma}

The following proposition will play a key role in proving the $\sF$-nilpotency of many $G$-spectra.
\begin{prop}\label{prop:oriented-nilpotence-criteria}
	Suppose that $R$ is a homotopy commutative $G$-ring spectrum with multiplicative Thom classes for either  unoriented, oriented, unitary, or $8n$-dimensional spin representations. Then the following are equivalent:
	\begin{enumerate}
		 \item The $G$-spectrum $R$ is $\sF$-nilpotent. 
		 \item For every subgroup $H \leq G$ with $H \notin \sF$, if 
	 		 \[  x\in \ker\left(\pi_*^H R\to  \prod_{K <
			 H} \pi_*^K R\right), \]
		 then $x$ is nilpotent. \label{it:oriented-nilpotence-criterion}
 	\end{enumerate}
\end{prop}
\begin{proof}
  The implication $(1)\implies (2)$ is an easy consequence of \Cref{thm:gen-gen-fiso},
  since for $H \notin \sF$, $\Res^G_H R$ is nilpotent for the family of proper
  subgroups of $H$.

  For the reverse implication, it suffices by \Cref{prop:ring-criteria-fnil} to
  show that for each $H\not\in \sF$, $e_{\tilde{\rho}_H}\in \pi_\star^H R$ is
  nilpotent. This class restricts to zero on all proper subgroups of $H$ by the
  first part of \Cref{prop:props-of-reps}, so if assumption
  \eqref{it:oriented-nilpotence-criterion} were stated in terms of
  $RO(H)$-graded groups we would already be done. Instead, we will use the Thom
  isomorphisms to reduce to the integer grading: 
  By \Cref{lem:cofinality}
  below there is an $n$ such that $n\tilde{\rho}_H$ has an Euler class
  $\chi(n\tilde{\rho}_H)\in \pi_*^H R$. This class is nilpotent by
  \Cref{lem:euler-classes} and assumption
  \eqref{it:oriented-nilpotence-criterion}. Since the nilpotency of
  $e_{\tilde{\rho}_H}$ is equivalent to the nilpotency of
  $\chi(n\tilde{\rho}_H)$ by \Cref{lem:euler-classes}, the result follows.
\end{proof}

\begin{lemma} \label{lem:cofinality}
	For any finite group $G$, 
	\begin{enumerate}
		\item $2\tilde{\rho}_G$ underlies a unitary and hence oriented representation and
		\item $8\tilde{\rho}_G$ underlies a spin representation whose dimension is divisible by 8.
	\end{enumerate}
\end{lemma}
\begin{proof}
	The real representation underlying the complex reduced regular representation of $G$ is $2\tilde{\rho}_G$, which proves the first claim. 

	Now $8\tilde{\rho}_G$ admits a spin structure if and only if the first two
	Stiefel-Whitney classes \[w_1(8\tilde{\rho}_G),w_2(8\tilde{\rho}_G)\in
	H^*(BG;\bF_2)\] vanish. 
By the Whitney sum formula $w_1(8\tilde{\rho}_G)=2w_1(4\tilde{\rho}_G)=0$ and hence $w_2(8\tilde{\rho}_G)=2w_2(4\tilde{\rho}_G)=0$ as desired.

  \end{proof}

\subsection{Equivariant topological $K$-theory}
In this subsection, we recall that we denote by $\sC$ the family of cyclic
subgroups of a group $G$. 

\begin{prop}\label{prop:kg-is-in-cnil}
	The derived defect base of both the complex and real equivariant $K$-theory spectra, $\KG$ and $\KOG$, is $\sC$.
\end{prop}
\begin{proof}

	First we show that $\KOG$ and $\KG$ are $\sC$-nilpotent. Since $\KG$ is a
	$\KOG$-module it suffices to prove this for $\KOG$. Now $\KOG$ admits
	multiplicative Thom classes for $8n$-dimensional spin representations and is
	$8$-periodic \cite[Thm.~6.1]{Ati68a}, so by 
	\Cref{prop:oriented-nilpotence-criteria},  it
	suffices to show that any element $x \in \pi_0^H \KOG=RO(H)$ which restricts
	to zero on all cyclic subgroups is zero. By elementary character theory, the complexification of $x$ in $R(H)$ is zero since it is zero on all cyclic subgroups. Now since the composite \(RO(H)\rightarrow R(H)\rightarrow RO(H)\) of the complexification and the forgetful map is multiplication by 2 and $RO(H)$ is torsion-free, $x$ is necessarily zero. 

	To see that this is a minimal family it suffices to prove this for the $\KOG$-module $\KG$. Now $\KG$ admits multiplicative Thom classes for unitary representations and is $2$-periodic \cite[Thm.~4.3]{Ati68a}. So it suffices, by \Cref{prop:oriented-nilpotence-criteria} again, to construct, for an arbitrary cyclic group $G$, a non-nilpotent element $x\in R(G)$ which restricts to zero on all proper subgroups. Since \[R(H_1\times H_2)\cong R(H_1)\otimes R(H_2),\] if we can do this in the case $G=C_{p^n}$ is a cyclic $p$-group, then we can tensor the classes together to obtain the desired element. 

	The character map $R(C_{p^n})\rightarrow \prod_{g\in C_{p^n}}\bC$
	is a ring map and an
	injection into a reduced ring, so any non-zero element
	of $R(C_{p^n})$ is non-nilpotent. Now $R(C_{p^n})=\bZ[z_n]/(z_n^{p^n}-1)$
	and $x=z_n^{p^{n-1}}-1$ is a nontrivial and hence non-nilpotent element of
	$R(C_{p^n})$. Under the inclusion $C_{p^{n-1}}\rightarrow C_{p^n}$ of the
	unique maximal proper subgroup, $z_n$ restricts to $z_{n-1}\in
	R(C_{p^{n-1}})$ and  we see that $x$ restricts to zero on this subgroup. It
	thus restricts to zero on all proper subgroups as desired.
\end{proof}

\begin{prop}\label{prop:borel-k-theory-is-in-cnil}
	The derived defect base of both the Borel-equivariant $K$-theory spectra $\KU$ and $\KO$ is $\bb{\sC}$.
\end{prop}
\begin{proof}
	First we will show that these spectra are $\bb{\sC}$-nilpotent. Since the
	arguments for the real and complex case are identical, we will just do the
	complex case. Since $\KG$ is split \cite[\S
	XVI.2]{May96}\cite[p.~458]{LMS86}, $\KU=F(EG_+,i_* \KG)\simeq F(EG_+,\KG)$. It follows that $\KU$ is a $\KG$-module and hence $\sC$-nilpotent by \Cref{prop:kg-is-in-cnil}. Since $\KU$ is also an $\bb{S}$-module it is $\sC\cap \bb{\sAll}=\bb{\sC}$-nilpotent by \Cref{prop:borel-sphere}.

	We will now prove that this is the minimal such family for $\KU$. Since $\KU$ is a $\KO$-module this will establish the minimality claim for $\KO$ as well. Now for each cyclic $p$-group $G$, we will construct a non-nilpotent element $x\in \pi_0^G\KU$, which restricts to zero on all proper subgroups. Since $\mathit{KU}$ is complex-orientable, it has Thom classes for unitary representations and the minimality claim will then follow from \Cref{prop:oriented-nilpotence-criteria}.

	We have already constructed such an $x$ in $\pi_0^G \KG\cong R(G)$ in \Cref{prop:kg-is-in-cnil}, so it suffices to show that the natural ring map \[ i\colon \pi_0^G\KG\rightarrow \pi_0^G \KU\] is an injection. By the Atiyah-Segal completion theorem \cite{AtS69}, $\pi_0^G\KU$ is $\widehat{R}(G)$, the completion of $R(G)$ at the ideal of virtual representations of dimension zero, and $i$ is the completion map. This map is an injection for all $p$-groups $G$ by \cite[Prop.~6.11]{Ati61}, so the claim follows. 
	Note that the use of the Atiyah-Segal completion theorem to bound below the
	derived defect base of $\underline{KO}$
	parallels the use of the Segal conjecture in \Cref{prop:borel-sphere}. 
\end{proof}

\begin{remark}

	One can also show that the derived defect bases of $\KG$ and $\KOG$
	(resp.~$\KU$ and $\KO$) agree with an independent argument using Galois descent
	\cite[Prop.~\ref{S-FnilKOKU}]{MNNa}.
\end{remark}

There are at least two standard notions of `connective' equivariant $K$-theory. The first is $\KG_{\tau\geq 0}$ which is the standard connective cover: it admits a canonical map $\KG_{\tau\geq 0}\rightarrow \KG$ such that $\pi_i^{(-)}$ is an isomorphism when $i\geq 0$ and $\pi_i^{(-)}\KG_{\tau\geq 0}=0$ for $i<0$. By \Cref{prop:connective-theories} the derived defect base of $\KG_{\tau\geq 0}$ is its defect base and this is the family of Brauer elementary subgroups of $G$ by \Cref{prop:defect-base-for-rg}.
The more interesting variant is the following:
\begin{defn}
	Let $\kG$ denote the equivariant connective\footnote{Which is not generally connective!} $K$-theory spectrum constructed by Greenlees \cite{Gre04,Gre05}. This is defined by the following homotopy pullback along the self-evident ring maps:
	\begin{equation}\label{eq:def-kG}
	\xymatrix{
		\kG \ar[d]  \ar[r] &  \KG \ar[d] \\
		\ku \ar[r] &  \KU\simeq F(EG_+,\KG)
}
	\end{equation}
	The real analogue, $\koG$ is defined similarly.
\end{defn}

\begin{prop}\label{prop:greenlees-kg-nil}
	The derived defect base of $\kG$ and $\koG$ is $\sC\cup \sE$. 
\end{prop}
\begin{proof}
	Since the arguments for $\kG$ and $\koG$ are essentially identical, we will
	prove the claim for $\kG$. We have already shown in
	\Cref{prop:kg-is-in-cnil,prop:borel-k-theory-is-in-cnil} that the derived
	defect bases of $\KG$ and $\KU$ are $\sC$ and $\bb{\sC}$ respectively. In
	\Cref{prop:ku-is-in-secnil} we will show that the derived defect base of $\bb{\mathit{ku}}$ is $\bb{\sC}\cup \sE$.  It follows that each of these spectra is $\sC\cup \sE$-nilpotent. Since the $\infty$-category of $\sC\cup \sE$-nilpotent $G$-spectra is closed under homotopy pullbacks, $\kG$ is $\sC\cup \sE$-nilpotent. The required results for $\KOG$, $\KO$, and $\ko$ are proven in \Cref{prop:kg-is-in-cnil,prop:KO-is-in-cnil}.

	Since $\KG$ and $\ku$ are $\kG$-algebras via the above maps, the minimality claim follows from the minimality of the families for $\ku$ and $\KG$. 
\end{proof}



Let $G$ be a compact Lie group with an involution $g \mapsto \overline{g}$,
i.e., a \emph{Real Lie group} in Atiyah's terminology.
Then one can form a split extension of groups (the semidirect product) 
\[ 1 \to G \to \widetilde{G} \to C_2 \to 1.  \]
A $\widetilde{G}$-space is then a \emph{Real $G$-space}
in the sense of \cite[\S 6]{AtS69}. There is an equivariant cohomology
theory
$K\mathbb{R}_G^*$ on
$\widetilde{G}$-spaces $X$ such that, for a finite $\widetilde{G}$-CW complex,
$K\mathbb{R}_G^0(X)$
is the Grothendieck group of \emph{Real $G$-vector bundles on $X$.}
In \cite[Ch.\ 14]{HJJS08}, the Thom isomorphism theorem is proved for Real
$G$-vector bundles on compact $\widetilde{G}$-spaces. 
We let $K\mathbb{R}_G $ be a ring $\widetilde{G}$-spectrum representing this cohomology
theory. 

We have the following generalization of 
\Cref{prop:kr-defect-base}. 
\begin{prop} 
\label{rem:equi-kr}
Suppose $G$ (and therefore $\widetilde{G}$) is finite.
The derived defect base of the $\widetilde{G}$-spectrum $K\mathbb{R}_G$ is given by the family of
cyclic subgroups of $G$. 
\end{prop} 
\begin{proof} 
We will need the two following observations: 
\begin{enumerate}
\item  
Let $X$ be a finite $\widetilde{G}$-CW complex on which $G$ acts trivially, so
that it arises from a finite $C_2$-CW complex. Then we have 
$K\mathbb{R}^{*}_G(X) = K\mathbb{R}^\ast(X) \otimes_{\mathbb{Z}}
K\mathbb{R}_G^0(\ast)$.
\item We have $\mathrm{Res}^{\widetilde{G}}_G K\mathbb{R}_G  = KU_G = KU $. 
\end{enumerate}
By the second item, it suffices to show that $K\mathbb{R}_G$ is nilpotent for the family of 
subgroups of $G$. 
To show this, we first let $\sigma$ denote the real sign representation of $C_2$, regarded as a $\widetilde{G}$-representation. Then the first item
together with the calculation used in \Cref{prop:kr-defect-base}
shows that the Euler class $S^0 \to S^{3 \sigma}$ becomes null-homotopic after
smashing with $K\mathbb{R}_G$. This means that $K\mathbb{R}_G$ is a retract of $S(3 \sigma)_+
\wedge K\mathbb{R}_G$ and is therefore nilpotent for the family of subgroups of $G$. 
\end{proof} 

\begin{remark}
	We do not know of an extension of \Cref{prop:do-not-descend} along the lines of \Cref{rem:equi-kr} as of this time.
\end{remark}

\subsection{Complex-oriented Borel-equivariant theories}

The following is fundamental to all of the following calculations
of derived defect bases of Borel-equivariant spectra. 
\begin{thm}[{cf.~\cite{HKR00} and
\cite[Cor.~\ref{S-complexorBorequivabelian}]{MNNa}}]\label{prop:mu-is-anil}
 	  The derived defect base of the Borel-equivariant $G$-spectrum $\bb{MU}$ is $\bb{\sAb}$.
\end{thm}
\begin{proof}
	We begin by showing that $\bb{MU}$ is $\bb{\sAb}=\bb{\sAll}\cap\sAb$-nilpotent. Since $\bb{MU}$ is an $\bb{S}$-module and the latter is $\bb{\sAll}$-nilpotent, it suffices to show that $\bb{MU}$ is $\sAb$-nilpotent. Since $\bb{MU}$ has Thom isomorphisms for unitary representations, it suffices by \Cref{prop:oriented-nilpotence-criteria} to show that any $x$ in $MU^*(BG)$ which restricts to zero on each abelian subgroup is nilpotent.

	The following nilpotence argument is standard (see \cite[\S 4]{GS99}) and
	dates back to \cite{Qui71b}.  Let $F$ be the variety of complete flags
	associated to a faithful representation of $G$. This is a compact
	$G$-manifold with abelian isotropy, so it admits the structure of a finite
	$G$-CW complex, whose cells have abelian isotropy which we will now fix. By
	\cite[Prop.~2.6]{HKR00} we have an inclusion $MU^*(BG)\rightarrow
	MU^*(EG\times_G F)$. So it suffices to show that $x$ is nilpotent in the target ring. Filtering $F$ by its $G$-CW structure, there is a multiplicative spectral sequence \[ E_2^{s,t}=H^s_G(F;\pi_{-t}^{(-)}\bb{MU})\Longrightarrow MU^{t+s}(EG\times_G F)\] with the following properties:
	\begin{enumerate}
		\item\label{it:ker-edge} Any class $y\in MU^*(BG)$ which restricts to
		zero in $MU^*(BA)$ for each abelian subgroup $A$ belongs to the kernel of the edge homomorphism \[MU^*(EG\times_G F)\rightarrow E_2^{0,*}\subseteq E_1^{0,*}.\] This is a consequence of the following two facts: 
		\begin{enumerate}
			\item The flag variety $F$ has abelian isotropy and hence $E_1^{0,*}$ is a product of terms of the form $MU^*(EG\times_G G/A)\cong MU^*(BA)$.
			\item The $E_1$-edge homomorphism is the product of the restriction homomorphisms induced by a coproduct of projections of the form $G/A\rightarrow G/G$.
		\end{enumerate}
		\item\label{it:nil-bound} $E_2^{s,*}=0$ for $s>\dim F$ by definition of the spectral sequence.
	\end{enumerate} 
	Property \eqref{it:ker-edge} shows that $x$ must be detected in positive
	filtration, while the property \eqref{it:nil-bound} shows that $x$ is nilpotent.

	In \Cref{prop:ln-local} below, we will show that for every prime $p$ and integer $n$, the $\bb{MU}$-module $\bb{E_n}$ has $\sAbpn$ as its derived defect base. Since $n$ and $p$ are arbitrary this forces the minimality of the family for $\bb{MU}$. 
\end{proof}
\begin{remark}\label{rem:flag-bound}
	The argument with the flag variety above plays a key role in the unipotence
	results of \cite[\S\ref{S-unipsec}]{MNNa}. A consequence of those results
	(combined with \Cref{prop:exponent} above) is
	that if $G$ admits an $n$-dimensional faithful complex representation then
	we obtain an explicit upper bound \[ \exp_{\sAb}(\bb{MU})\leq n(n-1)+1\] on the $\sAb$-exponent of $G$-equivariant $\bb{MU}$.
\end{remark}

\subsection{Ordinary Borel-equivariant cohomology}
We will now discuss the further reduction one can make when one is over the
integers. 
\begin{prop}[cf.~\cite{Qui71b}]\label{prop:hfp-is-epnil}
	The derived defect base of $\bb{H\bF_p}$ is $\sEp$. 
\end{prop}
\begin{proof}
	We first prove that $\bb{H\bF_p}$ is $\sEp$-nilpotent. Since $\bb{H\bF_p}$
	is an $\bb{MU}$-module and an $\bb{S_{(p)}}$-module, $\bb{H\bF_p}$ is
	$\sAb\cap \sAllp=\sAbp$-nilpotent by
	\Cref{prop:mu-is-anil,prop:borel-sphere}. Moreover, $\bb{H\bF_p}$ has Thom
	isomorphisms for oriented representations. So by
	\Cref{prop:oriented-nilpotence-criteria}, it suffices to show that if $G=A$ is an abelian $p$-group and $x\in H^*(BA;\bF_p)$ restricts to 0 on each elementary abelian subgroup then $x$ is nilpotent. 

	The remainder of the argument follows from elementary group cohomology
	calculations: If $A=C_{p^{i_1}}\times \cdots \times C_{p^{i_n}}$, then there
	is a polynomial subalgebra $R=\bF_p[x_1,\ldots,x_n]\subset H^*(BA;\bF_p)$
	whose generators are in degree 2 and such that for any element $x\in
	H^*(BA;\bF_p)$, $x^p \in R$. Moreover, there is a maximal elementary abelian subgroup $E$ of $A$ such that the composite \[ R \rightarrow H^*(BA;\bF_p)\xrightarrow{\Res_E^A} H^*(BE;\bF_p)\] is an injection. It follows that if $x$ restricts to zero on $E$ then $x$ is nilpotent.

	To prove minimality of the family $\sEp$, we suppose that $G=E$ is an
	elementary abelian group. To see that $\bb{H\bF_p}$ is not
	$\sP$-nilpotent we will construct an element $z\in H^*(BE;\bF_p)$, which
	restricts to zero on each proper subgroup of $E$ and belongs to the polynomial
	subalgebra $R$ of $H^*(BE;\bF_p)$ and hence is non-nilpotent. Let $y\in
	H^1(BC_p;\bF_p)=\bF_p$ be non-zero. For each nontrivial homomorphism
	$\phi\colon E\rightarrow C_p$, we obtain a nontrivial element
	\(y_\phi=\beta\phi^*(y)\in R\cap H^2(BE;\bF_p).\) By construction $y_\phi$
	restricts to zero on the maximal proper subgroup $\ker \phi$ of $E$. Since
	any proper subgroup is contained in the kernel of such a map, the element
	\[z=\prod_{\phi\in \mathrm{Gp}(E,C_p)\setminus \{0\}} y_\phi\] restricts to
	zero on any proper subgroup of $E$ and is non-nilpotent as desired because
	$z \in R$.
\end{proof}
\begin{cor}\label{cor:mo-is-e2-nil}
	The derived defect base of $\bb{MO}$ is $\sE_{(2)}$.
\end{cor}
\begin{proof}
Recall first that 
	$MO$ admits the structure of an $H \bF_2$-module via the work of Thom, cf.
	\cite[Theorem IV.6.2]{Rud98}.
 It follows that $\bb{MO}$ is $\sE_{(2)}$-nilpotent.
Since $\bb{H \mathbb{F}_2}$ is a $\bb{MO}$-algebra via the 	zeroth Postnikov
section, the
	minimality claim follows from the minimality claim for $\bb{H\bF_2}$ in
	\Cref{prop:hfp-is-epnil}. 
\end{proof}


\begin{example}\label{ex:quaternions}

We now examine the $\sE_{(2)}$-homotopy limit spectral sequence for $\bb{H\bF_2}$ when $G=Q_8$ is the quaternion group of order $8$. The edge homomorphism of this spectral sequence was first analyzed by Quillen \cite[Ex.~7.4]{Qui71b} and provides an example where this map is neither an injection nor a surjection, but is evidently an $\cF_2$-isomorphism. We will now calculate the rest of the spectral sequence and verify Quillen's result.

 In this case, the only nontrivial elementary abelian subgroup is the center
 $Z(Q_8)=C_2$. Since this is normal with quotient $C_2\times C_2$, by
 \Cref{lem:normal-subgroups} the $\sE_{(2)}$-homotopy limit spectral sequence
 (which is also the Lyndon-Hochschild-Serre spectral sequence) takes the form 
\[ H^s(B(C_2\times C_2);H^t(BC_2;\bF_2))\Longrightarrow H^{t+s}(BQ_8;\bF_2).\]
Since the action of $C_2\times C_2$ on the center is trivial, the local coefficient system is trivial.

Hence the $E_2$-page is isomorphic to $\bF_2[e_1,e_2]\otimes \bF_2[e_3]$ where $e_3$ generates the cohomology of the center and is in bidegree $(0,1)$. Now $e_1$ and $e_2$ are both in bidegree $(1,0)$ and for degree reasons they are permanent cycles. Since the spectral sequence does not have a horizontal vanishing line at the $E_2$-page we know that the last remaining indecomposable $e_3$, must support a differential. For degree reasons this must be a $d_2$.

To identify this differential we note that the $\sE_{(2)}$-homotopy limit
spectral sequence is acted on by $\mathrm{Aut}(Q_8)$. This follows from  the
observation that the family $\sE_{(2)}$ of elementary abelian 2-groups is invariant under automorphisms of $Q_8$, and all resolutions in question can therefore be
carried out respecting the $\mathrm{Aut}(Q_8)$-action. Since $\mathrm{Aut}(Q_8)$ fixes the center and acts transitively on the non-zero elements of the quotient group $C_2\times C_2$  \cite[Lem.~IV.6.9]{AdM04} we see that $d_2(e_3)$ must land in the invariants \[ H^2(BC_2\times C_2;\bF_2)^{\mathrm{Aut(Q_8)}}\cong \bF_2\{e_1^2+e_1 e_2 +e_2^2\}.\] This forces $d_2(e_3)=e_1^2+e_1 e_2 +e_2^2$. 
\begin{figure}
\centering 
\includegraphics[scale=0.75]{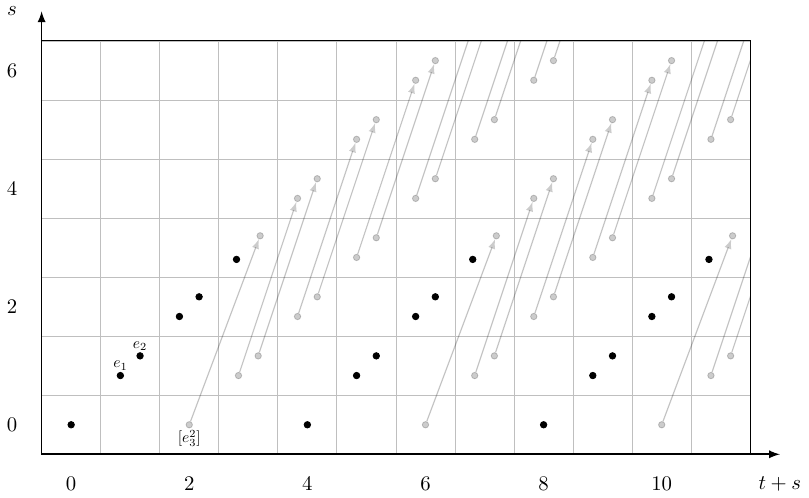}
\caption{The $E_3$-page of the $\sE_{(2)}$-homotopy limit spectral converging to $H^{t+s}(BQ_8;\bF_2)$.}\label{fig:q8-ss}
\end{figure}

The $E_3$-page shown in \Cref{fig:q8-ss} does not yet have a horizontal
vanishing line so there must be a further differential. By the same reasoning
as above we see that $[e_3^2]$ must support a differential and this must be a
$d_3$ which lands in the invariants of the $\mathrm{Aut}(Q_8)$-action. This forces $d_3([e_3^2])=e_1^2e_2+e_1e_2^2$. At this point there is no room for
further differentials and  the spectral sequence collapses at $E_4$. There are no additive or multiplicative extensions for degree reasons.  So we obtain: 
\[H^*(BQ_8;\bF_2)\cong \bF_2[e_1, e_2, [e_3^4]]/(e_1^2 +e_1e_2+e_2^2, e_1^2e_2+e_1e_2^2)\] (cf.~\cite[Lem.~IV.2.10]{AdM04}). 
Since there are elements of filtration $3$ at $E_\infty$, we find that 
$\exp_{\sE_{(2)}}( \underline{H \mathbb{F}_2}) \geq 4$.

We can in fact show that this is an equality, equivalently, that there is a
$3$-dimensional finite $Q_8$-CW-complex $X$ with isotropy in $\sE_{(2)}$
such that $\bb{H\bF_2}$ splits off $\bb{H\bF_2}\wedge X_+$. For this, we choose $X=\mathbb{P}(\mathbb{\mathbb{H}})$, the projective
space of the $4$-dimensional real representation of $Q_8$ afforded by quaternion multiplication on $\mathbb{H}\cong\mathbb{R}^4$. 
The required splitting
follows from the projective bundle theorem in mod $2$-cohomology (cf.\ \cite[Sec.~17,~Thm.~2.5]{Hus94}).
In fact, this  produces a map
\begin{equation} \label{projbundlemap} \bigvee_{i=0}^3 \Sigma^{-i}\underline{H \mathbb{F}_2} \to \underline{H
\mathbb{F}_2} \wedge \mathbb{D} X_+,  \end{equation}
classifying the generators of the free $H^*(B Q_8; \mathbb{F}_2)$-module
$\pi_*^{Q_8} ( \underline{H \mathbb{F}_2} \wedge \mathbb{D} X_+) \simeq
H_{Q_8}^*( X; \mathbb{F}_2)$. Since 
the projective bundle formula implies that \eqref{projbundlemap} is an equivalence on
$H$-fixed points for any $H \leq Q_8$, we get that \eqref{projbundlemap} is an equivalence and we
have the desired splitting after dualizing. 
\end{example}
\begin{remark}
 	In \cite{Qui71b}, Quillen actually considers a smaller indexing category than $\sOGEp$. The objects of this category $\cA$ are the elementary abelian subgroups of $G$ and the morphisms are the group homomorphisms $A\to B$ of the form $c_g\colon a\mapsto gag^{-1}$ for some $g\in G$. 

 	To relate these two notions we construct a functor $J\colon \sOGEp\to \cA$ sending $G/A$ to $A$. Given a $G$-map $f\colon G/A_1\to G/A_2$ satisfying $f(A_1)=gA_2$, we set $J(f)=c_{g^{-1}}$. Since the subgroups involved are abelian, this functor is well-defined. 

 	Now $J$ is a cofinal functor and hence the induced map $ \lim_{\cA^{\op}} F \to \lim_{\sOGEp^{\op}} J^*F$ is an isomorphism for every contravariant functor $F$ indexed on $\cA$. Now for a $G$-space $X$, the functor on $\sOGEp^{\op}$ sending $G/A$ to $H^*(EG\times_A X;\bF_p)$ extends over $J$, so Quillen's limit is isomorphic to the one considered here. 

 	However $J$ is not \emph{homotopy cofinal}; the higher limit terms are generally quite different. For example, in the case $G=Q_8$ just considered we have \[\sideset{}{^*}\lim_{\cA^{\op}} H^*(BA;\bF_2)\cong \sideset{}{^0}\lim_{\cA^{\op}} H^*(BA;\bF_2)\cong H^*(BZ(Q_8);\bF_2)\cong \bF_2[e_3].\]
 	Since the higher limit functors vanish we see that the homotopy limit spectral sequence using Quillen's indexing category \emph{will not} converge to $H^*(BQ_8;\bF_2)$.

Note also that 
the higher limit functors over the category $\cA$ (and its generalization for
arbitrary collections of $p$-subgroups) have been extensively studied in relation to the
theory of centralizer sharp homology decompositions \cite{Dw97, Dw98, JM92}, cf.
\cite{GS06} for an account and many examples. 
\end{remark}

 \begin{example}\label{ex:exponent}
	 We will now calculate the $\sF$-exponent of $Q_8$-equivariant $\bb{H\bF_2}$
	 for a slightly larger family than $\sE_{(2)}$. Let $f$ be one of the
	 nontrivial classes in $H^1(Q_8;\bF_2)$ and let $\sigma$ be the pullback of
	 the sign representation along $f$, so $\chi(\sigma) = f$. Now $f^3=0$ by
	 the calculation above so $\bb{H\bF_2}$ is a retract of $\bb{H\bF_2}\wedge
	 S(3\sigma)_+$. If we set $\sF$ to be the family of subgroups contained in
	 the kernel of $f$, then we see that the $\sF$-exponent of $\bb{H\bF_2}$ for $G=Q_8$ is at most 3. 
	Moreover, $\exp_{\sF}( \underline{H \mathbb{F}_2}) \geq 3$ because $f^2
	\neq 0$ (cf. \Cref{rem:lowerboundexp}), so we have in fact equality. 
\end{example}

We now prove the integral version of the above result. 
We will frequently use the following.  

\begin{lemma} 
\label{lem:arith-fracture}
Fix a finite group $G$.
For a spectrum $E$, we let $\sF_E$ be the derived defect base of
$\underline{E} \in \SpG$.
If $R$ is a ring spectrum, then $\sF_{R} = \bigcup_{p \mid  |G|} \sF_{R_p} = 
\bigcup_{p \mid |G|} \sF_{R_{(p)}}
$ where $R_p$ (resp. $R_{(p)}$)
denotes the $p$-completion (resp. $p$-localization) of $R$. 
\end{lemma} 
\begin{proof} 
We give the argument for the completions; the argument for the localizations is
similar. 
Since $\underline{R_p}$ is an algebra over $\underline{R}$, we have 
$\sF_{R} \supset \bigcup_p \sF_{R_p}$.
To 
obtain the opposite inclusion, we use the arithmetic square
\begin{equation}\label{eq:arith-sq}
\xymatrix{
R \ar[d]  \ar[r] &    \prod_{p\mid |G|} R_{p} \ar[d] \\
R[1/|G|] \ar[r] &   \left(\prod_{p\mid |G|}R_{p}\right)[1/|G|].
}
\end{equation}
This induces a pullback square upon taking Borel-equivariant theories. 
The Borel-equivariant forms of 
$R[1/|G|]$ and 
$\left(\prod_{p\mid |G|}R_{p}\right)[1/|G|]$ have trivial derived defect base
since they are $|G|^{-1}$-local (\Cref{prop:borel-sphere}). 
As a result, we obtain 
$\sF_{R} \subset  \sF_{\prod_{p \mid |G| }R_p} = \bigcup_{p \mid |G|}
\sF_{R_p}$ as desired. 
\end{proof} 

\begin{prop}[cf.~\cite{carlson}] \label{prop:HZ-is-enil}
	The derived defect base of $\bb{H\bZ}$ is $\sE$.
\end{prop}
We note that this result is essentially equivalent to \cite[Thm. 2.1]{carlson}. 
See also \cite[Thm.\ 4.3]{Bal} for another equivalent statement stated in a
language closer to ours. 
\begin{proof}
We first prove that $\bb{H\bZ}$ is $\sE$-nilpotent. By \Cref{lem:arith-fracture}, it suffices to show that $\bb{H\bZ_p}$ is $\sE_{(p)}$-nilpotent. Since $\bb{H\bZ_p}$ is both an $\bb{MU}$-module and an $\bb{S_{(p)}}$-module, we already know that $\bb{H\bZ_p}$ is $\bb{\sAb}_{(p)}$-nilpotent by \Cref{prop:mu-is-anil,prop:borel-sphere}.
Moreover, $\bb{H\bZ_p}$ has Thom isomorphisms for oriented representations. So by \Cref{prop:oriented-nilpotence-criteria} it suffices to show that if $A$ is an abelian $p$-group and $x\in H^*(BA;\bZ_p)$ restricts to 0 on each elementary abelian subgroup, then $x$ is nilpotent. 

Suppose we have such an $x\in H^*(BA;\bZ_p)$. Note that $|x|$ is necessarily greater than zero and, by \Cref{prop:hfp-is-epnil}, the mod-$p$ reduction of $x$ is nilpotent. In other words, a power of $x$ is divisible by $p$. It follows that there exists $k\geq 1$ and $z\in H^*(BA;\bZ_p)$ such that $x^k=|A|z$. Since $|A|\cdot H^*(BA;\bZ_p)=0$ for $*>0$, we see that $x^k=0$ as desired.
Thus $\bb{H \bZ}$ is $\sE$-nilpotent. 


Finally, to see that the derived defect base is precisely $\sE$, we note that
since $H\bF_p$ is an $H\bZ$-module, the derived defect base of $\bb{H\bZ}$ must
contain $\sE_{(p)}$ by \Cref{prop:hfp-is-epnil}. Varying $p$, we find that the
derived defect base must contain $\sE$, and therefore is equal to $\sE$s.
\end{proof}

\subsection{$L_n$-local Borel-equivariant theories}

Using Hopkins-Kuhn-Ravenel character theory, we now determine the minimal family for
Borel-equivariant Morava
$E$-theory and some related spectra. 

\begin{prop}[cf.~\cite{GS99,HKR00}] \label{prop:JW-theory} 
	Suppose that $E$ is a complex-oriented homotopy commutative ring spectrum
	with associated formal group $\bG$. Suppose further that:
	\begin{itemize}
		\item The coefficient ring $\pi_*E$ is a complete local ring with maximal ideal $m$.
		\item The graded residue field $\pi_*E/m$ has characteristic $p>0$.
		\item The localization $\pi_*E[p^{-1}]$ is non-zero.
		\item The mod $m$ reduction of $\bG$ has height $n<\infty$.
	\end{itemize}
	Then the derived defect base of $\bb{E}$ is $\sAbpn$.
\end{prop} 
\begin{proof} 
 	First we show that $\bb{E}$ is $\sAbpn$-nilpotent. Since $E$ is complex
	oriented and $p$-local we already know $\bb{E}$ is $\sAb\cap
	\sAllp=\sAbp$-nilpotent. So by \Cref{prop:oriented-nilpotence-criteria} it
	suffices to show that if $A$ is an abelian $p$-group and $x \in E^*(BA)$
	restricts to zero on $E^*(BA')$ for any $A' \leq A$ of rank $\leq n$, then
	$x$ is nilpotent. 

 	The results of \cite{HKR00} show that, under the given hypotheses, there is a natural injection 
 	\begin{equation}\label{eq:characters}
 	 E^*(BA)\hookrightarrow L(E^*)\otimes_{E^*}E^*(BA)\cong
	 \Hom_{\mathrm{Set}}(\mathrm{Gp}(\bZ_p^n, A), L(E^*))
 	\end{equation}
 	of $E^*(BA)$ into a ring of generalized characters valued in some particular nontrivial $E^*[p^{-1}]$-algebra $L(E^*)$. By assumption, $x\in E^*(BA)$ is trivial on all of the subgroups of $A$ which appear as images of some homomorphisms $\bZ_p^n\rightarrow A$. It follows that $x$ has trivial image in the character ring. Since $E^*(BA)$ injects into the character ring, $x$ must be zero.

	To see the minimality of this family, we will suppose that $A$ is a product of $n$ cyclic $p$-groups and find a non-nilpotent element $x\in E^*(BA)$ which restricts to zero on all proper subgroups
As in \cite[Thm. C]{HKR00}, there is an $\mathrm{Aut}(\mathbb{Z}_p^n)$-action
on the right-hand-side of \eqref{eq:characters} such that $p^{-1} E^*(BA)$ is
the $	\mathrm{Aut}(\mathbb{Z}_p^n)$-invariants.
	 Let $z\in 	 \Hom_{\mathrm{Set}}(\mathrm{Gp}(\bZ_p^n, A), L(E^*))$  be the generalized character which sends each surjective homomorphism $\bZ_p^n\rightarrow A$ to $1\in L(E^*)$ and all other homomorphisms to zero.
	 The element $z$ is $\mathrm{Aut}(\mathbb{Z}_p^n)$-invariant and therefore
	 belongs to $p^{-1} E^*(BA)$. Clearly $z$ is idempotent and restricts to
	 zero on all proper subgroups. Since the map in \eqref{eq:characters} is an
	 isomorphism after inverting $p$, there is an $x\in E^*(BA)$ and a natural
	 number $k$ such that $p^k z=x$. By construction, $x$ is non-nilpotent and restricts to zero on all proper subgroups.

\end{proof} 

The derived defect base shrinks if one quotients by an invariant ideal in
$\pi_0 E$. 
 For each positive integer $n$, let $\widehat{E}(n)$ denote the $I_n$-adically
 completed Johnson-Wilson theory. This is a complex oriented $p$-local
 cohomology theory whose coefficients $\pi_* \widehat{E}(n)$ are obtained by
 completing $\pi_*E(n)\cong \bZ_{(p)}[v_1,\dots,v_n,v_n^{-1}]$ at the ideal
 $I_n=(p,v_1\dots,v_{n-1})$ (here $v_0 = p$ conventionally).

\begin{prop} \label{prop:lubin-tate-quotients}
For $0 \leq k \leq n$, 
let $E=\widehat{E}(n)$ and $E/I_k = E/(p, v_1, \dots, v_{k-1})$.
The derived defect base of $\underline{E/I_k}$ is $\mathscr{A}_{(p)}^{n-k}$.
\end{prop} 
\begin{proof} 
Using the $v_n$-periodicity of $E$ it suffices to study the nilpotence of elements in degree 0. Let $\mathbb{G}$ be the formal group over $\pi_0(E/I_k)$ associated to the
complex-oriented ring spectrum $E/I_k$. 
Note that this is the reduction modulo $I_k$ of the formal group
associated to $E$. 
Let $A$ be an abelian $p$-group and let $A^{\vee}$ denote the Pontryagin dual. 

Recall that $\spec\left( (E/I_k)^0( BA)\right)$ is the formal scheme that classifies
homomorphisms $A^{\vee} \to \mathbb{G}$. 
Since $E^0(BA)$ is a finite free module over $\pi_0 E$, we have
\[ (E/I_k)^0(BA) \simeq \pi_0 (E/I_k) \otimes_{\pi_0E}
E^0(BA).  \]
By applying \cite[Thm.~2.3]{GS99} to $E$ and then base-changing along $\pi_0E \to \pi_0(E/I_k)$,
one has a morphism of schemes
\[  
\bigsqcup_{H \leq A} \mathrm{Level}(H^{\vee}, \mathbb{G})
\to
\spec\left( (E/I_k)^0( BA)\right), 
\]
which induces an isomorphism on underlying \emph{reduced schemes.}
Here $\mathrm{Level}(H^{\vee}, \mathbb{G})$ is the closed subscheme of $\spec\left(
(E/I_k)^0(BH)\right)$ classifying \emph{level structures} $H^{\vee} \to
\mathbb{G}$ and the above map factors through the map induced by the restriction homomorphisms.  
Moreover, 
$\mathrm{Level}(H^{\vee}, \mathbb{G})  $  is empty if and only if
$\mathrm{rank}(H) > n-k$. 

It follows now that the map of schemes
\begin{equation} \label{mapofschemes} \bigsqcup_{H<  A} \spec \left((E/I_k)^0(BH)  \right)
\to
\spec\left( (E/I_k)^0( BA) \right) \end{equation}
is surjective on geometric points if and only if $\mathrm{rank}(A) > n-k$. If
$\mathrm{rank}(A) > n-k$, it follows that any element in 
$(E/I_k)^0( BA)$ which restricts to zero on proper subgroups is nilpotent. This
proves that the derived defect base of $\underline{E/I_k}$ is
at most $\mathscr{A}_{(p)}^{n-k}$. 
Similarly, the analysis of \eqref{mapofschemes}
combined with \Cref{prop:stratification} shows that the derived defect base can
be no smaller. 
\end{proof} 
\begin{example} 
We show explicitly that the derived defect base of $\underline{K(n)}$ is
$\sTriv$. 
	Since $\sTriv$ is the minimal family, we only need to show that these
	$G$-spectra are $\sTriv$-nilpotent. Using \Cref{rem:borel-tate}, this can
	also be deduced from 
	\cite[Thm.~1.1]{GrS96}  (i.e., the vanishing of Tate spectra). 
	
	Since $K(n)$ is complex orientable and $p$-local we already know that $\bb{K(n)}$ is $\sAb\cap \sAllp=\sAbp$-nilpotent. Now since $K(n)$ admits Thom isomorphisms for unitary representations, it suffices to show that if $A=C_{p^{i_1}}\times\cdots \times C_{p^{i_k}}$ is an arbitrary abelian $p$-group and $x\in K(n)^*(BA)$ restricts to zero on the trivial subgroup then $x$ is nilpotent by \Cref{prop:oriented-nilpotence-criteria}. 
	By the K\"unneth isomorphism for Morava $K$-theory and the well-known
	calculations of the complex-oriented cohomology of cyclic groups, \[ K(n)^*(BA)\cong K(n)^*\otimes \bF_p[x_1,\dots, x_k]/(x_1^{p^{i_1n}},\dots,x_k^{p^{i_kn}}) \] and the kernel of the restriction map \( \Res_e^A\colon K(n)^*(BA)\rightarrow K(n)^*(Be)\) is the ideal $(x_1,\dots, x_k)$. This ideal is evidently nilpotent and hence so is $x$, proving the claim.

\end{example} 

We can also obtain a variant for the telescopic replacement for $K(n)$. 
\begin{prop}[Cf. \cite{Kuh04a}]\label{prop:kn-defect-base}
	Let $X$ be a type $n$-finite complex and let $T(n)$ be its $v_n$-periodic localization. Then the derived defect base of $\bb{T(n)}$ is $\sTriv$.
\end{prop}
\begin{proof}
	The spectrum $T(n)$ is a module over the $v_n$-periodic localization of the type $n$, $A_\infty$-ring spectrum $\mathrm{End}(X)$. So it suffices to consider the case $T(n):=\mathrm{End}(X)[v_n^{-1}]$. Since this spectrum is obtained from an $A_\infty$-ring by  inverting a central element \cite[Lem.~6.1.2]{Rav92} it is $A_\infty$ \cite[\S 7.2.4]{Lur14}. Now by \Cref{rem:borel-tate} it suffices to show that the associated Tate object $\wt{E}\sTriv\wedge \bb{T(n)}$ is contractible. This is \cite[Cor.~1.6]{Kuh04a}.
\end{proof}

\subsection{Hybrids of $L_n$-local theories and $H\mathbb{Z}$-algebras}

We now include examples of Borel-equivariant theories where there are two
contributions to the derived defect base: an $L_n$-local piece and a
$H\mathbb{Z}$-algebra piece. 

\begin{prop}\label{prop:BPn}
	The derived defect base of $\bb{BP\langle n\rangle}$ is $\sEp\cup \sAbpn$.
\end{prop} 
\begin{proof} 
	Since both $H\bF_p$ and the completed Johnson-Wilson theory $\widehat{E}(n)$ are $BP\langle n\rangle$-modules, the minimality claim will follow from the minimality results for these module spectra proven in \Cref{prop:hfp-is-epnil,prop:JW-theory}.

	To show that $\bb{BP\langle n\rangle}$ is $\sEp\cup \sAbpn$-nilpotent, we argue by induction on $n$. The base case $n=0$ follows from \Cref{prop:HZ-is-enil}. So suppose $n > 0$. Since $\bb{BP\langle n\rangle}$ has Thom isomorphisms for unitary representations, by \Cref{prop:oriented-nilpotence-criteria} it suffices to show that if $x\in BP\langle n\rangle^*(BG)$ restricts to zero in $BP\langle n\rangle^*(BA)$ for each $A\in \sEp\cup\sAbpn$ then $x$ is nilpotent. 

	First observe that $x$ maps to a nilpotent class in
	$(L_n BP \langle n\rangle)^*(BG)$ by \Cref{prop:ln-local} below. So by raising $x$ to a power, we may assume that $x$ is already zero in $(L_n BP \langle n\rangle)^*(BG)$.
	Moreover, by the inductive assumption, and raising $x$ to a sufficiently high power, we
	may assume that $x$ maps to zero under $r$ in the long exact sequence
	\[ \cdots \rightarrow (BP \langle n\rangle)^{* - |v_n|}(BG)\xrightarrow{v_n} (BP \langle n\rangle)^*(BG) \xrightarrow{r} (BP \langle n-1\rangle)^*(BG)\rightarrow \cdots.\] This means that $x=v_n y$ for some $y\in(BP \langle n\rangle)^*(BG)$.

	The $L_n$-localization map fits into the following fiber sequence of $BP \langle n\rangle$-modules
	\[ \Gamma_n BP \langle n\rangle  \to BP \langle n\rangle \to L_n BP \langle n\rangle.  \] Mapping $BG$ into this sequence we obtain another fiber sequence of $BP \langle n\rangle$-modules
	\[F(BG_+,\Gamma_n BP \langle n\rangle) \rightarrow  F(BG_+,BP \langle
	n\rangle) \rightarrow  F(BG_+,L_n BP\langle n \rangle).  \]
	By the long exact sequence in homotopy, $x$ lifts to $(\Gamma_n BP \langle n\rangle)^*(BG)$. 
	Now by \cite[Thm.~2.3, \S 3, and Thm.~6.1]{GrM95a} we see that $\Gamma_n BP
	\langle n  \rangle$ and hence $F(BG_+,\Gamma_n BP\langle n\rangle)$ are bounded above $BP\langle n\rangle$-modules. It follows that there is a power of $v_n$ such that
	\[ v_n^r x = 0\in (\Gamma_n BP \langle n\rangle)^*(BG), \quad r \gg 0.  \]
	Examining the long exact sequence we see that $0=v_n^r x\in BP\langle n \rangle^*(BG)$. Moreover, since $x=v_n y$, \[x^{r+1}=(v_n y)^{r+1}=v_n^r x y^r=0\] as desired.
\end{proof} 
The key properties of $BP\left \langle n\right\rangle$ that
are used above are that it is a ring spectrum with the desired homotopy groups,
that $BP\left \langle n\right\rangle \to L_n BP\left \langle n\right\rangle$ is
an equivalence on connective covers, and that it admits the standard cofiber sequence relating $BP\left \langle
n\right\rangle$ to $BP\left \langle n-1\right\rangle$. As such, the argument is
quite robust. 
We give another example of this argument below. 
\begin{prop} 
\label{regseqarg}
Let $R$ be a connective $E_\infty$-ring. 
Suppose that $\pi_*(R) \simeq \pi_0(R)[x_1, \dots, x_n]$, where $x_i \in
\pi_*(R)$ is in positive even degrees.
Consider the finite localization $R'$ of $R$ away from $(x_1,\dots, x_n)$
  (cf.\ \cite{GrM95a}).
For a spectrum $X$, let $\mathcal{F}_X$ denote the derived defect base of
$\underline{X}$ with respect to a finite group $G$.
Then  we have
\[ \mathcal{F}_R = \mathcal{F}_{R'} \cup \mathcal{F}_{H \pi_0 R} . \]
\end{prop} 
\begin{proof} 
Since $R'$and $H \pi_0 R$ are $R$-algebras, the inclusion $\mathcal{F}_R \supset \mathcal{F}_{R'} \cup \mathcal{F}_{H \pi_0 R}$ is evident. 

Let $G$ be a finite group such that $\underline{R'}$, $\underline{H \pi_0 R} \in
\GSpec$ are
nilpotent 
for the family of proper subgroups. It suffices to show that $\underline{R}$ is
too.
We will show that 
$\underline{R/(x_1, \dots, x_k)}$ is nilpotent for the family of proper
subgroups by descending induction on $k$. 
When $k = n$, this iterated cofiber is $H \pi_0 R$ and the induction hypothesis holds by assumption. 

Suppose now that 
$\underline{R/(x_1, \dots, x_{k+1})}$ is nilpotent for the family of proper
subgroups. We want to prove the analog with $k+1$ replaced by $k$.
Note that each $R/(x_1, \dots, x_i)$ admits an $A_\infty$-algebra structure in $R$-modules by
\cite[Cor.\ 3.2]{Ang08a}.
Since $\pi_*(R)$ is concentrated in even degrees, $R$ is complex orientable. 
It therefore suffices to show that if 
\[ u \in (R/(x_1, \dots, x_k))^*(BG)  \]
restricts to zero on proper subgroups, it is nilpotent. 
The inductive hypothesis shows that a power $u^k$ of $u$ is a multiple of
$x_{k+1}$, so it suffices to show that some power of $u$ is annihilated by a power of
$x_{k+1}$. 

Let $\Gamma_n R$ denote the fiber of $R \to R'$, so that $\Gamma_n R$ has
bounded-above homotopy groups via the spectral sequence of \cite[(3.2)]{GrM95a}. 
We consider similarly the cofiber sequence
\[ \Gamma_n R/(x_1, \dots, x_k) \to R/(x_1, \dots, x_k) \to  R'/(x_1, \dots,
x_k),   \]
where 
$\Gamma_n R/(x_1, \dots, x_k)$ has bounded-above homotopy groups. 
Replacing $u$ by a power of itself, we may assume that $u$ maps to zero in
$(R'/(x_1, \dots, x_k))^*(BG)$, so that it lifts to the module $(\Gamma_n R/(x_1, \dots,
x_k))^*(BG)$. However, we see as before that every element of this (as a bounded above object) is
annihilated by a power of $x_{k+1}$. 

\end{proof} 

\begin{cor}\label{prop:ku-is-in-secnil}
	The derived defect base of $\bb{\mathit{ku}}$ is $\sE\cup \bb{\sC}$.
\end{cor}
\begin{proof}
	By \Cref{lem:arith-fracture}, it suffices to check the derived defect base
	of $\bb{\mathit{ku}_{(p)}}$ is $\sE_{(p)}\cup \sC_{(p)}$ for each prime $p$
	dividing the group order. Now since $\mathit{ku}_{(p)}$ splits as a wedge of
	suspensions of $BP\langle 1\rangle$ and $BP\langle 1\rangle$ is a
	$\mathit{ku}_{(p)}$-module, the derived defect base of $\bb{\mathit{ku}_{(p)}}$ is the derived defect base of $\bb{BP\langle 1\rangle}$. The claim now follows from \Cref{prop:BPn}.
\end{proof}

\begin{prop} \label{prop:connective-kn}
The derived defect base of $\underline{k(n)}$ is $\mathcal{E}_{(p)}$.
\end{prop} 
\begin{proof} 
This is deduced similarly from the derived defect bases of $\underline{K(n)}$ and
$\underline{H \mathbb{F}_p}$. We leave the details to the reader. 

\end{proof} 
\subsection{Thick subcategory arguments}\label{sec:thick-examples} 
We will now show how to apply thick subcategory arguments, e.g. the
thick subcategory theorem of Hopkins-Smith \cite[Thm.~7]{HoS98}, to extend the results of the previous section to non-orientable Borel-equivariant 
 theories such as $\tmf$ and $\bb{L_n S^0}$.

\begin{prop}\label{prop:ln-local}
	The derived defect base of $\bb{L_n S^0}$ is $\sAbpn$. 
\end{prop}
\begin{proof}
	By the Hopkins-Ravenel smash product theorem, there exists $k\geq 0$ such
	that  $L_n S^0$ is a retract of $\mathrm{Tot}_k E_n^{\wedge \bul+1}$, the
	$k$th stage of the $E_n$-cobar construction \cite[\S 8]{Rav92}. Since
	$\bb{E_n}$ is $\sAbpn$-nilpotent (\Cref{prop:JW-theory}), for each positive
	integer $k$, the module spectrum $\bb{E_n^{\wedge k}}$ is
	$\sAbpn$-nilpotent. Taking finite homotopy limits, we see that the
	Borel-equivariant theories $\bb{\mathrm{Tot}_k E_n^{\wedge \bul+1}}$
	associated to the finite stages of the $E_n$-cobar construction are
	$\sAbpn$-nilpotent. Finally, since $\sAbpnNil$ is closed under retracts, we see that $\bb{L_n S^0}$ is $\sAbpn$-nilpotent.
	Conversely, since $\bb{E_n}$ is an $\bb{L_n S^0}$-module, the minimality claim follows from \Cref{prop:JW-theory}.
\end{proof}

\begin{lemma}\label{lem:thick-subcat}
	Suppose that $p$ is a prime and $X$ is a $p$-local finite spectrum of type zero, i.e., $H_*(X;\bQ)\neq 0$. Then $\bb{X\wedge M}$ is $\sF$-nilpotent if and only if $\bb{M_{(p)}}$ is $\sF$-nilpotent. 
\end{lemma}
\begin{proof}
Note that the functor $X \mapsto \underline{X}$ 
preserves finite limits and colimits, so that $\underline{X \wedge M} \simeq
i_* X \wedge \bb{M}$ for a finite spectrum $X$. It thus follows that the thick
subcategories of $\SpG$ generated by $\underline{X \wedge M}$ and
$\underline{M_{(p)}}$ are equal, so their derived defect bases are equal.
\end{proof}
\begin{prop}\label{prop:KO-is-in-cnil}
	The derived defect base of $\bb{KO}$ is $\bb{\sC}$, while the derived defect base of $\bb{ko}$ is $\bb{\sC}\cup \sE$.
\end{prop}
\begin{proof}
	Both of these statements are consequences of Wood's theorem \cite[Thm.~3.2]{Mat13} which gives equivalences $ C\eta\wedge \mathit{KO} \simeq \mathit{KU}$ and $ C\eta \wedge\mathit{ko}\simeq \mathit{ku}$. Since $H_*(C\eta;\bQ)\neq 0$ we can apply \Cref{lem:thick-subcat} and \Cref{prop:kg-is-in-cnil,prop:ku-is-in-secnil} to see that $\bb{KO}$ is $\bb{\sC}$-nilpotent and $\bb{ko}$ is $\bb{\sC}\cup \sE$-nilpotent. 
	The minimality of these families follows from the minimality results in \Cref{prop:kg-is-in-cnil,prop:ku-is-in-secnil} for their respective module spectra $\bb{KU}$ and $\bb{ku}$.
\end{proof}

\begin{defn}\label{def:tmf}
	Let $\cO^{\mathrm{top}}$ be the Goerss-Hopkins-Miller sheaf of
	$E_\infty$-ring spectra on $\overline{\cM}_{ell}$, the compactified moduli
	stack of elliptic curves (cf.~\cite{Beh12b}). Let
	$\cM_{ell}\subset \overline{\cM}_{ell}$ denote the locus parametrizing smooth elliptic curves.
	\begin{itemize}
		\item Let $\mathit{TMF}:=\Gamma(\cM_{ell};\cO^{\mathrm{top}})$ denote the derived global sections of $\cO^{\mathrm{top}}$ over $\cM_{ell}$.
		\item Let $\mathit{Tmf}:=\Gamma(\overline{\cM}_{ell};\cO^{\mathrm{top}})$ denote the derived global sections of $\cO^{\mathrm{top}}$.
		\item Let $\mathit{tmf}$ denote the connective cover of $\mathit{Tmf}$.
	\end{itemize}
\end{defn}

Note that by construction, we have a sequence \[ \mathit{tmf}\rightarrow \mathit{Tmf}\rightarrow \mathit{TMF}\] of $E_\infty$-ring maps, where the first map is the connective covering and the second map is induced by the restriction map on structure sheaves.

\begin{prop}\label{prop:non-connective-variants-of-tmf}
	The family $\bb{\sAb^2}$ is the derived defect base of both $\bb{\mathit{Tmf}}$ and $\bb{\mathit{TMF}}$.
\end{prop}
\begin{proof}
	To show that the derived defect bases of $\bb{\mathit{Tmf}}$ and
	$\bb{\mathit{TMF}}$ must contain $\bb{\sAb^2}$, it suffices to show this is true for the $\bb{\mathit{Tmf}}$-module $\bb{\mathit{TMF}}$. We will do this by constructing, for every prime $p$, a map $\mathit{Tmf}\rightarrow \widehat{E}$ of ring spectra such that the derived defect base of $\bb{\widehat{E}}$ is $\sAb^2_{(p)}$. By varying $p$ we see that the derived defect base of $\bb{\mathit{Tmf}}$ must contain all of $\bb{\sAb^2}$.

Fix a supersingular elliptic curve $C$ over
$\overline{\mathbb{F}}_p$ (recall that the existence of such a curve for every $p$
is classical and follows from the Eichler-Deuring mass formula
\cite[Exer.~V.5.9]{Sil86}).  It determines a geometric point $x\colon
\Spec(\overline{\mathbb{F}}_p)\to\cM_{ell}$. 
Choose an affine \'etale neighborhood $\spec(R) \to \cM_{ell}$ of $x$. Note
that $R$ is finitely generated over $\mathbb{Z}$ (hence noetherian) and
torsion-free. 
Let $E$ denote the localization of   $\cO^{\mathrm{top}}( R)$ at the prime
ideal corresponding to the point $x$. 
There is a canonical map of $E_\infty$-rings $\mathit{TMF}\to E$. Furthermore,
$E$ is even periodic with trivial $\pi_1$, hence complex orientable, $\pi_0 E$
is a local ring, and the corresponding formal group $\mathbb{G}$ on $\pi_0
E\simeq\mathcal{O}_{\cM_{ell},x}$ is the base-change of the formal group of the
elliptic curve. In particular, the reduction of $\mathbb{G}$ modulo
the maximal ideal of $\pi_0 E$ is the formal group of $C$, hence of height $2$.
It now follows from \Cref{prop:JW-theory} applied to the completion \cite[\S
4]{DAGXII} $\widehat{E}$ of $E$ at
the maximal ideal that the derived defect base of $\bb{\widehat{E}}$ is $\sAbp^2$, as desired.

	
	We will now prove that $\bb{\mathit{Tmf}}$ is $\bb{\sAb^2}$-nilpotent; the
	claim for $\bb{\mathit{TMF}}$ will then follow by the module structure.
	By \Cref{lem:arith-fracture}, it
	suffices to show $\bb{\mathit{Tmf}_{(p)}}$ is $\sAbp^2$-nilpotent for each
	prime $p$ (dividing $|G|$). Since $\mathit{Tmf}_{(p)}$ is $L_2$-local
	\cite{Beh12b}, and hence an $L_2S^0$-module, the result now follows from \Cref{prop:ln-local}.
\end{proof}

\begin{prop}\label{prop:connective-tmf}
	The derived defect base of $\bb{\mathit{tmf}}$ is $\bb{\sAb^2}\cup \sE$.
\end{prop}
\begin{proof}
	For the minimality claim, we note that $H\bZ$ is a $\mathit{tmf}$-module and hence the derived defect base of $\bb{\mathit{tmf}}$ must contain $\sE$ by \Cref{prop:HZ-is-enil}. Since $\mathit{Tmf}$ is also a $\mathit{tmf}$-module,  the derived defect base of $\bb{\mathit{tmf}}$ must also contain $\bb{\sAb^2}$ by \Cref{prop:non-connective-variants-of-tmf}.

	To prove that $\bb{\mathit{tmf}}$ is $\bb{\sAb^2}\cup\sE$-nilpotent, we will
	use \Cref{lem:arith-fracture} and check this locally
	at every prime. 
	
	\begin{enumerate}
	\item  
At the prime 2 we recall that there is a finite $2$-local
	spectrum $DA(1)$ of type zero such that $DA(1)\wedge \mathit{tmf}_{(2)}\simeq
	BP\langle 2\rangle$ \cite[Thm.\ 5.8]{Mat13}. 
	It now follows from
	\Cref{prop:BPn,lem:thick-subcat} that $\bb{\mathit{tmf}_{(2)}}$ is
	$\sE_{(2)}\cup\sAb^2_{(2)}$-nilpotent.
	\item
  A similar argument at the prime $3$
	applies using a finite $3$-local complex $F$ of type zero such that $
	F\wedge\mathit{tmf}_{(3)}\simeq \mathit{tmf}_1(2)_{(3)}$ (cf.\ \cite[Thm.\ 4.13]{Mat13}).
	One now applies 
	\Cref{regseqarg}
	 to determine the derived
	defect base of $\underline{\mathit{tmf}_1(2)_{(3)}}$ (whose homotopy
	groups are polynomial on classes $a_2, a_4$ in $\pi_4, \pi_8$) and hence that of
	$\underline{\mathit{tmf}_{(3)}}$.
	Note that the finite localization of $\mathit{tmf}_1(2)_{(3)}$ away from the
	ideal $(a_2, a_4)$ is $Tmf_1(2)_{(3)}$  because
	the compactified moduli stack $(M_{\overline{ell}, 1}(2))_{(3)}$ is $(\mathrm{Spec}
	\mathbb{Z}_{(3)}[a_2, a_4] \setminus V(a_2, a_4))/\mathbb{G}_m$.  This is
	in particular $L_2$-local by construction, so that 
the Borel-equivariant theory $\underline{Tmf_1(2)_{(3)}}$ is $\bb{\sAb^2}$-nilpotent 
as before. 
\item 
	At $p \geq 5$, one applies \Cref{regseqarg} directly to
	$\mathit{tmf}_{(p)}$, which is now complex orientable and whose homotopy
	groups are $\mathbb{Z}_{(p)}[c_4, c_6]$.
	Similarly, the finite localization away from the ideal generated by $(c_4,
	c_6)$ is $Tmf_{(p)}$ and is therefore $L_2$-local by construction.
	Therefore, one can conclude as before.  
	\end{enumerate}
	
\end{proof}

\subsection{Additional bordism theories}

Finally, we determine the derived defect bases for the 
Borel-equivariant forms of a few additional bordism theories. 
\begin{prop}\label{prop:mso-nilpotence}
	The derived defect base of $\bb{MSO}$ is $\sE_{(2)}\cup \bb{\sAb}[1/2]$.
\end{prop}
\begin{proof}
	By \Cref{lem:arith-fracture}, it suffices to show that the derived defect bases of $\bb{MSO_{(2)}}$ and
	$\bb{MSO[1/2]}$  are $\sE_{(2)}$ and $\bb{\sAb}[1/2]$, respectively. 

	Using a result of Wall \cite[Thm.~5]{Wal60}, $MSO_{(2)}$ admits an $H\bZ_{(2)}$-module structure \cite[p.~209]{Sto68} and hence $\bb{MSO_{(2)}}$ is $\sE_{(2)}$-nilpotent by \Cref{prop:HZ-is-enil}. This family is minimal since $H\bZ_{(2)}$ is an $MSO_{(2)}$-algebra via the zeroth Postnikov section. 

	It is well known that the evident composite 
	\begin{equation}\label{eq:MU-splits}
	MSp\rightarrow MU\rightarrow MSO
	\end{equation} 
	of ring maps induces an isomorphism in $\bZ[1/2]$-homology. Since these
	spectra are connective the composite in \eqref{eq:MU-splits} is a homotopy
	equivalence after inverting 2.
	It follows that the
	derived defect base of $\bb{MU[1/2]}$ is bounded above by the derived defect
	base for $\bb{MSp[1/2]}$ and bounded below by the derived defect base for
	$\bb{MSO[1/2]}$ and that all of these derived defect bases are equal.  Now,
	by \Cref{prop:borel-sphere,prop:mu-is-anil}, each of these derived defect bases is $\bb{\sAll}[1/2]\cap \bb{\sAb}=\bb{\sAb}[1/2]$.
\end{proof}
As shown in the course of the previous proof we have:
\begin{cor}\label{cor:msp-nilpotence}
	The derived defect base of $\bb{MSp[1/2]}$ is $\bb{\sAb}[1/2]$.
\end{cor}
	In general, the map of ring spectra $MSp\rightarrow MU$ induced by forgetting structure and \Cref{prop:mu-is-anil} show that any defect base of $\bb{MSp}$ must contain $\bb{\sAb}$.

We now consider the family of Borel equivariant bordism theories associated to $MU\langle n\rangle$ and $MO\langle n\rangle$ for $n\geq 0$. These commutative ring spectra are constructed by applying the generalized Thom construction to the $n-1$st connective covers of $BU\times \bZ$ and $BO\times \bZ$ respectively.  

By construction, there are maps of ring spectra $MU\langle n\rangle \to MU\langle n-1\rangle$ and $MO\langle n\rangle \to MO\langle n-1\rangle$ for $n\geq 1$. So the following proposition gives lower bounds on the derived bases of the associated Borel equivariant theories:
\begin{prop}\label{prop:MUP}
	The derived defect base of $\underline{MU\langle 0\rangle}$ is $\bb{\sAb}$. The derived defect base of $\underline{MO\langle0\rangle}$ is $\sE_{(2)}$. 
\end{prop}
\begin{proof}
	Since $MU\langle 0\rangle$ is complex oriented, $\underline{MU\langle 0\rangle}$ is $\bb{\sAb}$-nilpotent. As an $MU$-module, there is a well-known splitting $MU\langle 0\rangle\simeq \bigvee_{n\in \bZ} \Sigma^{2n} MU$. So the derived defect base of $\underline{MU\langle 0\rangle}$ must contain $\bb{\sAb}$ as well.

	The argument for $MO\langle 0\rangle$ is deduced similarly from \Cref{cor:mo-is-e2-nil}.
\end{proof}

\begin{prop}\label{prop:MUn}
	For any $n\geq 0$ the derived defect base of $\bb{MU\langle n\rangle}$ is $\bb{\sAb}$. 
\end{prop}
\begin{proof}
	By \Cref{lem:arith-fracture} and \Cref{prop:MUP}, it suffices to show, for each prime $p$, $\bb{MU\langle n\rangle_{(p)}}$ is $\sAb_{(p)}$-nilpotent. Using the results of  \cite[\S 5]{HoR95} and the notation therein, there is a map of ring spectra $MBP(r, tq)\to MU\langle n\rangle_{(p)}$. So it suffices to show $\bb{MBP(r,tq)}$ is $\sAb_{(p)}$-nilpotent. 

	The ring spectrum $MBP(r,tq)$ has the property that there exists a finite
	type 0 complex $X$ and a splitting of $MBP(r,tq)\wedge X$ into a wedge of
	suspensions of $BP$ by \cite[Cor.~5.2]{HoR95}.
We note that the existence of desired finite complexes follows from the work
of Smith \cite[Thm.~1.5]{Smith}; note that Smith's construction produces
finite even complexes.	
	The claim now follows from \Cref{lem:thick-subcat}.
\end{proof}

\begin{prop}\label{prop:MOn}
	For any $n\geq 0$, $\bb{MO\langle n\rangle}$ is $\bb{\sAb}$-nilpotent and the derived defect base of $\bb{MO\langle n\rangle [1/2]}$ is $\bb{\sAb}[1/2]$ for $n\geq 2$.
\end{prop}
\begin{proof}
	Using the orientations $MU\langle n\rangle \to MO\langle n\rangle$ we see the first claim is a corollary of \Cref{prop:MUn}.  For $n\geq 2$ the derived defect base of $\bb{MO\langle n\rangle[1/2]}$ contains that of $\bb{MSO[1/2]}=\bb{MO\langle 2\rangle}$. So the second claim follows from the first and \Cref{prop:mso-nilpotence}.
\end{proof}

We can obtain a sharp lower bound on the derived defect bases of $\bb{MO\langle n\rangle}$ for small $n$.
\begin{prop}\label{prop:spin-bordism}
	The derived defect base of $\MSpin=\bb{MO\langle 4\rangle}$ is $\sC_{(2)}\cup \sE_{(2)}\cup \bb{\sAb}[1/2]$.
\end{prop}
\begin{proof}
By \Cref{lem:arith-fracture} and \Cref{prop:MOn}, it suffices to show that the derived defect base of $\bb{\mathit{MSpin}_{(2)}}$ is $\sC_{(2)}\cup \sE_{(2)}$.


	Now $2$-locally $\mathit{MSpin}$ additively splits as a wedge of $\mathit{ko}_{(2)}$-modules \cite{ABP67} and $\mathit{ko}_{(2)}$ is an $\mathit{MSpin}_{(2)}$-module via the Atiyah-Bott-Shapiro orientation. It follows that the derived defect base of $\bb{\mathit{MSpin}_{(2)}}$ is equal to the derived defect base of $\bb{ko_{(2)}}$ which is $\sC_{(2)}\cup \sE_{(2)}$ by \Cref{prop:KO-is-in-cnil,prop:borel-sphere}.
\end{proof}

\begin{prop}\label{prop:string-bordism}
	The derived defect base $\sF$ of $\bb{\mathit{MString}}$ satisfies $\bb{\sAb}\supseteq \sF\supseteq \bb{\sAb}[1/2]\cup\sE_{(2)}\cup\sAb^2_{(2)}$.
\end{prop}
\begin{proof}
	The first containment is a special case of \Cref{prop:MOn}. The second containment follows from the $\mathit{String}$-orientation on $\mathit{tmf}$ \cite{AHR06} and \Cref{prop:connective-tmf}.
\end{proof}


	It is an open problem to determine if there is an analogue of the Anderson-Brown-Peterson splitting of $\mathit{MSpin}_{(2)}$ for $\mathit{MString}_{(2)}$. 
	If $\mathit{MString}_{(2)}$ split additively into a wedge of $\mathit{tmf}_{(2)}$-modules then our methods would show that the derived defect base of $\MString$ is $\bb{\sAb}[1/2]\cup\sE_{(2)}\cup\sAb^2_{(2)}$.

\appendix
\section{A toolbox for calculations}
Below we provide a few technical results for working with $\sF$-homotopy (co)limit spectral sequences.

\subsection{The classifying space $E\sF$}\label{sec:esf}
We will now verify some claims about the classifying space $E\sF$ which were used in the body of the paper. 

Let $i\colon \sOGF\rightarrow \Top_G$ be the inclusion of the full subcategory
spanned by the transitive $G$-sets with isotropy in $\sF$. We have defined $E\sF$ to be $\hocolim_{\sOGF} i$. We can model this $G$-space by the standard two-sided bar construction \cite[\S XII.5]{BoK72}\cite[\S V.2]{May96}:
\begin{equation}
	E\sF:=\hocolim_{\sOGF} i\simeq |B_\bul(*,\sOGF, i)|,
\end{equation}
where $B_\bul(*,\sOGF, i)$ is the simplicial $G$-space which in degree $n$ is 
\[ \coprod_{(G/H_0,\dots,G/H_n)\in \sOGF^{\times n+1}} *\times \sOGF(G/H_n,G/H_{n-1})\times \cdots \times \sOGF(G/H_1, G/H_0)\times i(G/H_0). \] 
The zeroth face map is the projection which sends $\sOGF(G/H_n, G/H_{n-1})$ to a point and is the identity on the other factors. Using the functoriality of $i$ we obtain a map \[\sOGF(G/H_1, G/H_0)\times i(G/H_0)\rightarrow i(G/H_1).\] The last face map is the product of this with map with the identity on the remaining factors. The remaining face maps come from the composition in $\sOGF$ and the degeneracies come from including identities into the hom-factors. 

\begin{prop}\label{prop:sf-universality}
	The $G$-space $E\sF$ has the following properties:
	\begin{enumerate}
	\item \label{en:esf-univ-1} The fixed points of $E\sF$ have the following homotopy types: \[	E\sF^K \simeq \begin{cases}
		* & \mbox{if } K\in \sF\\
		\emptyset & \mbox{otherwise.}
		\end{cases}
		\]
	\item \label{en:esf-univ-2} Let $\Top_\sF\subseteq \Top_G$ denote the full
	subcategory spanned by those $G$-spaces which admit a $G$-CW structure with
	cells having isotropy only in $\sF$. Then $E\sF$ is a homotopically terminal object in $\Top_\sF$. 
	\item \label{en:esf-univ-3} The $G$-space $E\sF$ is determined up to equivalence by Condition \eqref{en:esf-univ-1}.
	\end{enumerate}
\end{prop}

\begin{proof} 
We only give the proof of the first assertion; for the others, see for instance
\cite[Sec. 1.2]{Luck05}. 
Let $K \leq G$ be such that $K \notin \sF$. Since $K$-fixed points commute with
homotopy colimits, it follows easily that $(E\sF)^K = \emptyset$. 
Suppose now $K \in \sF$; then we have
\[ \hocolim_{G/H \in \mathcal{O}_{\sF}(G)} (G/H)^K = 
 \hocolim_{G/H \in \mathcal{O}_{\sF}(G)} \Hom_{\Top_G}( G/K, G/H ) \simeq
\ast
\]
because the homotopy colimit of a corepresentable functor is contractible. 
\end{proof} 

\subsection{Cofinality results}
The following cofinality results aid in the calculation of $\sF$-homotopy (co)limit spectral sequences.

\begin{lemma}\label{lem:normal-subgroups}
	Let $N$ be a normal subgroup of $G$. If $\sF$ is the family of all subgroups of $N$ then the inclusion $i\colon BG/N^{\op}\rightarrow \sOGF$ is homotopy cofinal.
	In particular, the derived functors of colimits and limits over $\sOGF$ for
	$\mathcal{F}$ the family of subgroups contained in $N$ are identified with
	group (co)homology for $G/N$.
\end{lemma}
\begin{proof} 
This is a special case of \cite[Prop.~\ref{S-furthercofinalcollection}]{MNNa}.
\end{proof} 
\begin{prop}\label{prop:pushouts}
  Let $p$ and $q$ be two distinct primes and $\sF_1$ (resp.~$\sF_2$) the family
  of $p$-subgroups (resp.~$q$-subgroups) of the finite group $G$. Then the
  following commutative square of categories
  \begin{equation}\label{eq:pushout-of-cats} 
 \xymatrix{
BG\ar[d] \ar[r] &  \mathcal{O}_{\sF_1}(G)^\op  \ar[d]  \\
  \mathcal{O}_{\sF_2}(G)^\op  \ar[r] &  \mathcal{O}_{\sF_1\cup \sF_2}(G)^\op 
 }
 \end{equation} 
 induces a pushout of simplicial sets upon applying 
the nerve functor. \end{prop}
\begin{proof}
It suffices to prove the statement above for the opposite categories. That is,
we show \[N\sO_{\sF_1\cup \sF_2}(G)\cong N\sO_{\sF_1}(G)\cup_{NBG^\op}
N\sO_{\sF_2}(G)\] is the pushout of the nerves. Note that the pushout
simplicial set $N\sO_{\sF_1}(G)\cup_{NBG^\op} N\sO_{\sF_2}(G)$ is just a set-theoretic union in each degree. 

	So we need to show that  any $n$-simplex in the nerve of $\sO_{\sF_1\cup
	\sF_2}(G)$ lies entirely in the nerve of
	$\sO_{\sF_1}(G)$ or entirely in the nerve of
	$\sO_{\sF_2}(G)$; and if the $n$-simplex lies in both, then it must lie
	entirely in their intersection $NBG^\op$. When $n$ is $0$, the $n$-simplices
	of a nerve correspond to the objects of the category and this claim is
	obvious. When $n$ is positive the $n$-simplices correspond to chains of
	morphisms
	of length $(n-1)$. 
	
	Now suppose that $H_p$ is a $p$-subgroup of $G$ and $H_q$ is a $q$-subgroup of $G$, then
	 \[\sO_{\sF_1\cup \sF_2}(G)(G/H_p,G/H_q)\neq \emptyset \iff H_p \textrm{ is subconjugate to }H_q\iff H_p=e.\] 
	 This argument is symmetric in $p$ and $q$, so any chain of morphisms in
	 $\sO_{\sF_1\cup \sF_2}(G)$ is either in the image of $\sO_{\sF_1}(G)$ or in
	 the image of $\sO_{\sF_2}(G)$ under the embeddings in \eqref{eq:pushout-of-cats}. If the chain of morphisms is in both categories, then it is a sequence of endomorphisms of $G/e$ as desired. 
\end{proof}

\begin{prop}\label{prop:pullbacks}
  Let $p$ and $q$ be two distinct primes and $\sF_1$ (resp.~$\sF_2$) the family
  of $p$-subgroups (resp.~$q$-subgroups) of the finite group $G$. Let $\cC$ be a
  complete $\infty$-category. Then for any functor \[F\colon \sO_{\sF_1\cup
  \sF_2}(G)^\op\rightarrow \sC\] the decomposition in \Cref{prop:pushouts} induces a homotopy pullback diagram in $\sC$:
  \[ 
  \xymatrix{
  \holim_{\sO_{\sF_1\cup \sF_2}(G)^\op} F \ar[d]  \ar[r] &
  \holim_{\sO_{\sF_1}(G)^\op} \left. F\right|_{\sO_{\sF_1}(G)^\op} \ar[d]  \\
  \holim_{\sO_{\sF_2}(G)^\op} \left. F\right|_{\sO_{\sF_2}(G)^\op} \ar[r] &  \holim_{BG}\left. F\right|_{BG}   
  }
  \]
\end{prop}
\begin{proof}
	Applying the nerve functor to the pushout diagram from \Cref{prop:pushouts}
	we obtain a pushout diagram of $\infty$-categories where, since the
	inclusions are fully faithful, each map is a monomorphism. The claim now
	follows from \cite[Prop.~4.4.2.2]{Lur09}, after taking opposite
	$\infty$-categories.
\end{proof}


\section{A sample calculation in equivariant $K$-theory}
In this section we analyze the $\sC$-homotopy limit spectral sequence
converging to $\pi_*^G \KG$ when $G=C_2\times C_2$ is the Klein 4-group and
$\sC=\sP$ is the family of all cyclic subgroups of $G$. Of course we know the
target groups of this spectral sequence and we will use this knowledge below.
Nonetheless, this calculation does
illustrate some standard techniques for calculating derived functors and for
evaluating differentials in these spectral sequences. Moreover, this determines
the `stable' portion of the $\sC$-homotopy limit spectral sequence converging to
the homotopy groups of the Picard spectrum of the category of $G$-equivariant
$\KG$-modules (cf.\ \cite{MaS}). We hope to return to this topic later.

Even this most elementary case is still nontrivial. We will leave minor details to the reader.

We first fix some notation for the various subgroups and quotient groups:
\begin{align*}
	H_1 &= C_2\times e <G &\  F_1=G/H_1\\
	H_2 &= e\times C_2 <G &\  F_2=G/H_2\\
	H_3 &= \Delta(C_2) <G &\ F_3=G/H_3
\end{align*}
The quotient maps induce ring homomorphisms $R(F_i)\rightarrow R(G)$ such that the induced map \[R(F_1)\otimes R(F_2)\rightarrow R(G)\] is an isomorphism. 
Let $\sigma_i$ denote both the complex sign representation of $F_i\cong C_2$ and the representation of $G$ obtained by pulling back along the quotient map. 

The $\sC$-homotopy limit spectral sequence for $\KG$ takes the following form:
 \[\sideset{}{^s}\lim_{\sOGC^\op}\pi_{t}^{(-)} \KG\Longrightarrow \pi_{t-s}^G \KG\]
 The abutment is \[R(G)[\beta^{\pm
 1}]=\bZ\{1,\sigma_1,\sigma_2,\sigma_3=\sigma_1\sigma_2\}[\beta^{\pm 1}]\]
 where $\beta$ is the Bott periodicity generator in degree 2 and $R(G)$ is the
 complex representation ring in degree 0. Since $\pi_*^{(-)}\KG$ is 2-periodic
 with respect to this generator, the $E_2$-page is 2-periodic as well.

The map sending a virtual representation to its virtual dimension defines a map
$R(-)\rightarrow \underline{\bZ}$ of Green functors with kernel the
augmentation ideal functor $I(-)$. Although this map does not split as
Mackey functors, it does split as coefficient systems. From this splitting we obtain:
\begin{prop}\label{prop:klein}
	The $E_2$-term of the $\sC$-homotopy limit spectral sequence has the following form:
\[ \sideset{}{^*}\lim_{\sOGC^\op}\pi_*^{(-)}\KG\cong \sideset{}{^*}\lim_{\sOGC^\op}(\underline{\bZ})[\beta^{\pm}]\oplus\sideset{}{^*}\lim_{\sOGC^\op}(I(-))[\beta^{\pm}] \]
\end{prop}

To calculate these summands  we will use the identification, for coefficient
systems $M$, \[\sideset{}{^*}\lim_{\sOGF^\op} (M)\cong \Ext_{\ZOGF}^*(\underline{\bZ},M)\] of \Cref{sec:bredon}. One could calculate this directly from the definition by taking a projective resolution of $\underline{\bZ}$ in coefficient systems. We will instead use a less direct method that can be be applied to a wider class of problems. 

We will perform the analogous calculation for various subfamilies $\sF\subseteq \sC$ of subgroups, starting with the trivial family and gradually working our way up. For such a family let $\bZ[\sF]$ be the coefficient system obtained by restricting $\underline{\bZ}$ to $\sOGF^\op$ and then left Kan extending to a functor on $\sOGC^\op$. We then define the coefficient system $\bZ[\tilde{\sF}]$ by the following short exact sequence:
\[ 0\rightarrow \bZ[\sF]\xrightarrow{i} \bZ[\sC]\xrightarrow{\pi} \bZ[\tilde{\sF}]\rightarrow 0\] where $i$ is the counit of the left Kan extension/restriction adjunction.

From this short exact sequence we obtain the following long exact sequence of $\Ext$-groups:
\begin{equation}\label{eq:les}
 \cdots \Ext_{\ZOGC}^s(\bZ[\sF], M) \xleftarrow{i^*} \Ext_{\ZOGC}^s(\bZ[\sC], M) \xleftarrow{\pi^*} \Ext_{\ZOGC}^s(\bZ[\tilde{\sF}], M) \xleftarrow{\partial} \Ext_{\ZOGC}^{s-1}(\bZ[\sF], M)\cdots
\end{equation}
Just as in the proof of  \Cref{prop:e2-identification} we have an adjunction isomorphism 
\[\Ext_{\ZOGC}^s(\bZ[\sF], M)\cong \Ext_{\ZOGF}^s(\underline{\bZ}, M).\] We will use this isomorphism and the long exact sequence of \eqref{eq:les} repeatedly to calculate the $E_2$-term in \Cref{prop:klein} by gradually increasing the size of the family under consideration.

\subsection{The trivial family of subgroups}\label{sec:trivial-family}
We begin by considering the trivial family of subgroups. In this case $\sOGF^\op$ is the category with one object $G/e$ and whose morphisms are the elements of $G$. The composition law is obtained from the group multiplication and a projective resolution of $\underline{\bZ}$ in $\ZOGF$ is just a projective resolution of the trivial module $\bZ$ in $\bZ[G]$-modules. Under this identification the free module $\bZ[G]$ corresponds to the restriction of the projective functor $\bZ\{\sOG(-, G/e)\}$ to the trivial family. This leads easily to the following identification (when $\sF=\sTriv=\{e\}$)
\[ \Ext_{\ZOGF}^s(\underline{\bZ}, M)\cong H^s(G;M(G/e)). \]
In the case $M=I$, $I(G/e)=0$ so these groups vanish. To simplify the notation we will write $I(H):=I(G/H)$ below. When $M=\underline{\bZ}$ this is just the integral cohomology of $G$, which we will denote by $H^*(G;\bZ)$ throughout this section.

We will now recall the well-known calculation of $H^*(G;\mathbb{Z})$ in order to fix notation and to relate it to the cohomology of the subgroups $H_i$ and the quotient groups $F_i$. We will use the Bockstein spectral sequence from the cohomology with $\mathbb{F}_2$-coefficients. Recall that $H^*(F_i;\bF_2)$ is a polynomial algebra on a generator $x_i$ in degree 1. This element supports a nontrivial Bockstein $\beta x_i=\mathrm{Sq}^1 x_i=x_i^2$. By the K\"unneth theorem the quotient maps induce an isomorphism \[H^*(F_1;\bF_2)\otimes H^*(F_2;\bF_2)\cong \bF_2[x_1, x_2] \cong H^*(C_2\times C_2;\bF_2).\]
The Bockstein spectral sequence collapses at $E_2$. There is only simple $2$-torsion and no exotic multiplicative extensions:
\[ H^*(C_2\times C_2;\bZ)\cong \bZ[y_1, y_2, z]/(2y_1, 2y_2, 2z, z^2-y_1y_2^2-y_1^2y_2).\] Here $y_i=\beta x_i$ is in degree 2 and $z=\beta(x_1 x_2)$ is in degree 3. 

\subsection{The nearly trivial family of subgroups}
We now consider the case $\sF=\{e,H_i\}$. By the Yoneda lemma, we can identify a map of coefficient systems \[f\colon \bZ\{\sOGF(-,G/H)\}\rightarrow M\] with an element $f\in M(G/H)$. In particular we obtain an augmentation map
\[\varepsilon\colon \bZ\{\sOGF(-,G/H)\}\rightarrow \underline{\bZ}\]
corresponding to the unit. Similarly, for every element in the Weyl group $g\in N_G H/H=\sOG(G/H,G/H)$ we obtain a map
\[g\colon \bZ\{\sOGF(-,G/H)\}\rightarrow \bZ\{\sOGF(-,G/H)\}\]
\begin{lemma}\label{lem:short-res}
	Let $\underline{\bZ}$ denote the constant $G=C_2\times C_2$-Green functor at the integers restricted to the family $\sF=\{e, H_i\}$. Let $g$ be a generator of the quotient group $F_i= N_GH_i\cong C_2$. Then the following sequence of functors is exact:
	\[ 0\rightarrow \underline{\bZ}\rightarrow \bZ\{\sOGF(-, F_i)\} \xrightarrow{e+g}\bZ\{\sOGF(-,F_i)\}\xrightarrow{e-g}\bZ\{\sOGF(-,F_i)\}\xrightarrow{\varepsilon}\underline{\bZ}\rightarrow 0.\]
\end{lemma}
Before proceeding to the proof, we note that we can concatenate these exact sequences together to obtain a projective resolution of $\bb{\mathbb{Z}}$. This immediately yields:
\begin{cor}\label{cor:nearly-triv}
	For the family $\sF=\{e,H_i\}$, we have the following identification: \[\Ext_{\ZOGF}^*(\underline{\bZ}, M)\cong H^*(F_i;M(G/H_i)).\]
\end{cor}

\begin{proof}[Proof of \Cref{lem:short-res}]
	Although this is a special case of \Cref{lem:normal-subgroups}, we include an alternative, and perhaps more explicit, argument. 

	Since kernels and cokernels in $\ZOGF$ are calculated object-wise, the exactness of a sequence of natural transformations is equivalent to the exactness of the sequence of maps obtained by  evaluating at $G/H_i=F_i$ and $G/e$. In both cases we obtain the beginning of the standard 2-periodic $\bZ[F_i]$-resolution of the trivial module $\bZ$. 
\end{proof}

We will now calculate the terms in \Cref{cor:nearly-triv} when $M$ is either of the summands $\underline{\bZ}$ or $I(-)$ of $R(-)$. From the discussion in \Cref{sec:trivial-family} we know that $H^*(F_i;\bZ)\cong \bZ[y_i]/(2y_i)$. The action of $F_i$ on $I(H_i)=\bZ\{1-\overline{\sigma}_i\}$ is via the conjugation action on $H_i$ which is trivial since $G$ is abelian. Regarding $H^*(F_i;I(H_i))$ as a module over $H^*(F_i,\bZ)$, we obtain: \[ H^*(F_i;I(H_i))\cong \bZ[y_i]/(2y_i)\otimes (1-\overline{\sigma}_i).\] Here $\overline{\sigma}_i$ is the sign representation of $H_i$, not $F_i=G/H_i$, so $\sigma_j$ restricts to $\overline{\sigma}_i$ if and only if $j$ is \emph{not} $i$.

Note that the relations $(1-\overline{\sigma}_i)^2=2(1-\overline{\sigma}_i)$ and $2y_i=0$ force all products of positive degree elements in $H^*(F_i;I(H_i))$ to vanish. 

We will need to understand the behavior of the restriction map induced by the natural transformation \[i\colon \bZ[\{ e \} ]\rightarrow \bZ[\{ e,H_i\} ]\] of coefficient systems. Topologically this corresponds to the nontrivial map $EG_+\rightarrow E\{e,H_i\}_+\simeq {EF_i}_+$ of pointed $G$-spaces. Using the preferred models $EG=|G^{\bul +1}|$ and $EF_i=|F_i^{\bul +1}|$, we see that this map is induced by the quotient map $G\rightarrow F_i$. It follows that 
\[
i^*\colon H^*(F_i;\bZ)\rightarrow H^*(G;\bZ)
\] is induced by the quotient $G\rightarrow F_i$. 
Of course
\[
i^*\colon H^*(F_i;I(H_i))\rightarrow H^*(G;I(e))=0
\] is the zero map.

\subsection{The family $\sC=\sP$}
We can now assemble the above results to calculate the $E_2$-term from \Cref{prop:klein}. The sum of the counit maps \[\bigoplus_{i=1}^3 \bZ[ \{ e, H_i \} ]\rightarrow \underline{\bZ}\] is evidently surjective and yields the following short exact sequence of coefficient systems:
\begin{equation}\label{eqn:ses-k-theory-ex}
 0\rightarrow \bZ[ \{ e \} ]\oplus \bZ[\{ e \} ]\xrightarrow{j} \bigoplus_{i=1}^3 \bZ[ \{ e, H_i \} ]\rightarrow \underline{\bZ}=\bZ[ \{ e,H_1,H_2,H_3 \} ]\rightarrow 0.
\end{equation}
 Here $j$ is the inclusion of the kernel, which is adjoint to the linear map between trivial $\bZ[G]$-modules given by the following matrix:
\[ j=\left( \begin{array}{cc}
1 & 0 \\
-1 & 1 \\
0 & -1 \end{array}\right).\] 
The short exact sequence in \eqref{eqn:ses-k-theory-ex} induces the following long exact sequence in $\Ext$-groups:
\begin{equation}\label{eqn:ses-k-theory-ex-2}
	\cdots \xleftarrow{\partial} \bigoplus_{i=1}^2 H^*(C_2\times C_2;\bZ[\beta^\pm]) \xleftarrow{j^*} \bigoplus_{i=1}^3 H^*(F_i; R(H_i)[\beta^\pm]) \leftarrow H^*_{C_2\times C_2}(E\sC; \Gpi \KG) \xleftarrow{\partial}\cdots.
\end{equation}

\begin{remark}\label{rem:top-lifting}
We can find $G$-spectra whose integral Bredon homology realizes the short exact sequence of coefficient systems in \eqref{eqn:ses-k-theory-ex}. Moreover, we can also lift the maps of a coefficient systems to maps of $G$-spectra:
\begin{equation}
	\Sigma^\infty_+EG\vee \Sigma^\infty_+ EG \xrightarrow{j} \bigvee_{i=1}^3 \Sigma^\infty_+ E\sC(H_i)\rightarrow \Sigma^\infty_+ E\sC.
\end{equation}
Mapping this sequence into $\KG$ and taking the associated Atiyah-Hirzebruch spectral sequences one can see that the maps in \eqref{eqn:ses-k-theory-ex} are morphisms between the $E_2$-terms of these spectral sequences. 
\end{remark}


\begin{thm}\label{thm:example-e2}
	The $E_2$-term from \Cref{prop:klein} can be explicitly identified as follows:
	\[
		\sideset{}{^*}\lim_{\sOGC^\op} (\Gpi \KG)\cong \bZ[\beta^\pm]\otimes\left(A\oplus B\right)	\]
	where 
	\begin{equation} A= \mathrm{Im}\partial=\bigoplus_{i=1}^2\tilde{H}^{*-1}(C_2\times C_2;\bZ)/\left(\bZ/2\{(y_1^k,0),(y_2^k,y_2^k),(0,y_3^k)\}_{k\geq 1}\right)\label{eq:line-1}
	\end{equation}
	and 
	\begin{equation}
	 B=\ker j^*= \bZ\oplus \left(\bigoplus_{i=1}^3 H^*(F_i; I(H_i))\right)\label{eq:line-2}.
	\end{equation}
\end{thm}
\begin{proof}
Plugging the calculations of the previous sections into the associated long exact sequence from \eqref{eq:les} we obtain an exact sequence
\[ 0\rightarrow A\rightarrow \sideset{}{^*}\lim_{\sOGC^\op} R(-) \rightarrow B\rightarrow 0.\] 
In the zeroth cohomological degree all of the terms are in $B$, so we have to check this sequence splits in positive degrees.
In positive degrees the splitting of the coefficient system $R(-)\cong \bZ\oplus I(-)$ splits this sequence.
\end{proof}

\begin{remark}
We will now perform some simple consistency checks. Note that all of the positive filtration terms in this spectral sequence are 2-torsion in accordance with \Cref{prop:torsion}. 
One can independently verify the correctness of the 0-line by analyzing the representation rings.

Finally we note that if we restrict to any proper subgroup of $C_2\times C_2$, then for some $i$ and all positive $k$, $y_i^k$ is sent to zero as are all the terms divisible by the $z$ classes. The terms $I(H_i)$ restrict to zero on all of the subgroups except $H_i$, in which case the higher cohomology groups map to zero. It follows that all of the positive degree terms restrict to zero on any proper subgroup as expected. 
\end{remark}

\subsection{Analysis of the $\sC$-homotopy limit spectral sequence.}\label{sec:C-homotopy-lim-analysis}
In this section we will complete this calculation and prove:
\begin{thm}\label{thm:ss-conclusion}
	Let $G=C_2\times C_2$. The $\sC$-homotopy limit spectral sequence
	\[ E_2^{s,t}=\sideset{}{^s}\lim_{\sOGC^\op}(\pi_t^{(-)} \KG)\Longrightarrow \pi_{t-s}^G\KG \cong R(G)[\beta^\pm]\]
	collapses at $E_4$ onto the zero line. Moreover the $E_2$-edge homomorphism \[ R(G)\rightarrow \lim_{\sOGC^\op} R(C)\] is injective with cokernel $\bZ/2$. A generator of the cokernel supports a nontrivial $d_3$. 
\end{thm}

We will break up the analysis of this spectral sequence using the splitting in \Cref{thm:example-e2}. To analyze the $A$-summand in \eqref{eq:line-1} we will first determine the behavior of the classical Atiyah-Hirzebruch spectral sequence
\begin{gather}
 H^s(G;\bZ[\beta^\pm])\cong \bZ[y_1,y_2, z]/(2y_1,2y_2,
 2z,z^2-y_1^2y_2-y_1y_2^2)[\beta^{\pm}]\\ \Longrightarrow
 \mathit{KU}^{s-t}(BG)\cong \pi_{t-s}^G F(EG_+;\mathit{KU}).\label{eq:classical-ahss}
\end{gather} 
This spectral sequence arises from a multiplicative filtration on the ring
spectrum $R:=F(EG_+,\mathit{KU})$, which is compatible with the similarly
defined filtration on the free module \[\Sigma^{-1}R\vee \Sigma^{-1}R\simeq
\Sigma^{-1}(R\times R)\simeq \Sigma^{-1}F(EG_+\vee EG_+,\mathit{KU})\simeq
F(\Sigma(EG_+\vee EG_+),\mathit{KU}).\] We can now identify the spectral
sequence converging to $\pi_{t-s}^G F(\Sigma(EG_+\vee EG_+), \mathit{KU})$ as
two shifted copies of the spectral sequence in \eqref{eq:classical-ahss}. As
discussed in \Cref{rem:top-lifting}, the spectral sequence converging to $\pi_{t-s}^G F(\Sigma(EG_+\vee EG_+, \mathit{KU})$ maps to the $\sC$-homotopy limit spectral sequence and by \Cref{thm:example-e2} the image of this morphism is the $A$-summand. Through this comparison we can determine the differentials emanating from the $A$-summand from the differentials in \eqref{eq:classical-ahss}.

\begin{figure}
\centering 
\includegraphics[scale=0.65]{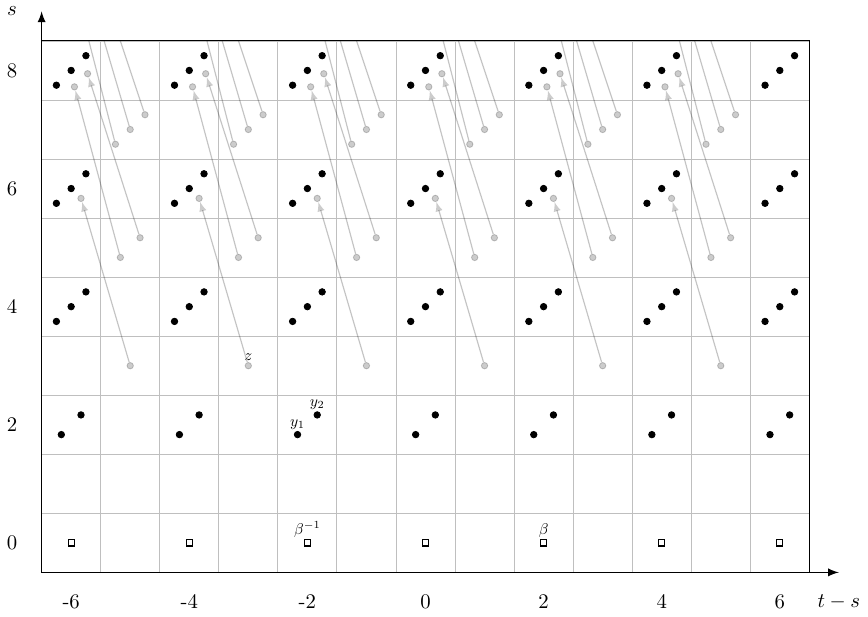}
\caption{The $E_3$-page of the Atiyah-Hirzebruch spectral sequence converging to $\pi_{t-s}^G\KU\cong \mathit{KU}^{s-t}(BG)$, where $G=C_2\times C_2$.}\label{fig:ahss}
\end{figure}
Now by \cite[Prop.~2.4]{Ati61}, the first differential in \eqref{eq:classical-ahss} is a $d_3$ given by the operation $\mathrm{Sq}^3_\bZ$. This operation is defined to be the following 
composition 
\[\mathrm{Sq}^3_\bZ\colon H\bZ \xrightarrow{-\otimes \bZ/2} H\bZ/2 \xrightarrow{\mathrm{Sq}^2} \Sigma^2 H\bZ/2\xrightarrow{\beta_\bZ} \Sigma^3 H\bZ. \] 
Here $\beta_\bZ$ is the boundary map induced by the short exact sequence of coefficients: \[0\rightarrow \bZ\xrightarrow{\cdot 2} \bZ\rightarrow \bZ/2\rightarrow 0. \]

Now the mod-2 reduction of $y_i^k$ is $x_i^{2k}$ and an easy inductive argument shows that, for $k$ odd, $\mathrm{Sq}^2 x_i^{2k}=x_i^{2k+2}$ and, for $k$ even, $\mathrm{Sq}^2 x_i^{2k}=0$. Now by our calculations from \Cref{sec:trivial-family}, $\beta_\bZ$ is zero on $x_i^{2\ell}$, for $\ell$ positive. It follows that $d_3(y_i^k)=\mathrm{Sq}^3_{\bZ}y_i^k=0$. Similarly
\[\mathrm{Sq}_\bZ^3 z=\beta_\bZ \mathrm{Sq}^2 (x_1^2 x_2+x_1 x_2^2)= \beta_{\bZ} (x_1^4 x_2 + x_1 x_2^4)= y_1^2 y_2 +y_1 y_2^2=z^2,\] 
so $d_3(z)=z^2$. Using the Leibniz rule, one generates all other differentials in this spectral sequence. The $E_4$-page is concentrated in even degrees, so the spectral sequence of \Cref{fig:ahss} collapses at this stage.

\begin{proof}[Proof of \Cref{thm:ss-conclusion}]
Since the homotopy groups $\Gpi \KG$ are concentrated in even degrees the first possible differential in the $\sC$-homotopy limit spectral sequence is a $d_3$. We will first calculate this differential on the $A$-summand from \eqref{eq:line-1}:
\begin{equation}\label{eq:boundary} 
\partial\colon \bZ[\beta^\pm]\otimes A=\bigoplus_{i=1}^2\tilde{H}^{*-1}(C_2\times C_2;\bZ[\beta^\pm])\rightarrow H^*_{C_2\times C_2}(E\sC;\Gpi \KG).
\end{equation}
\begin{figure}
\centering 
\includegraphics[scale=0.85]{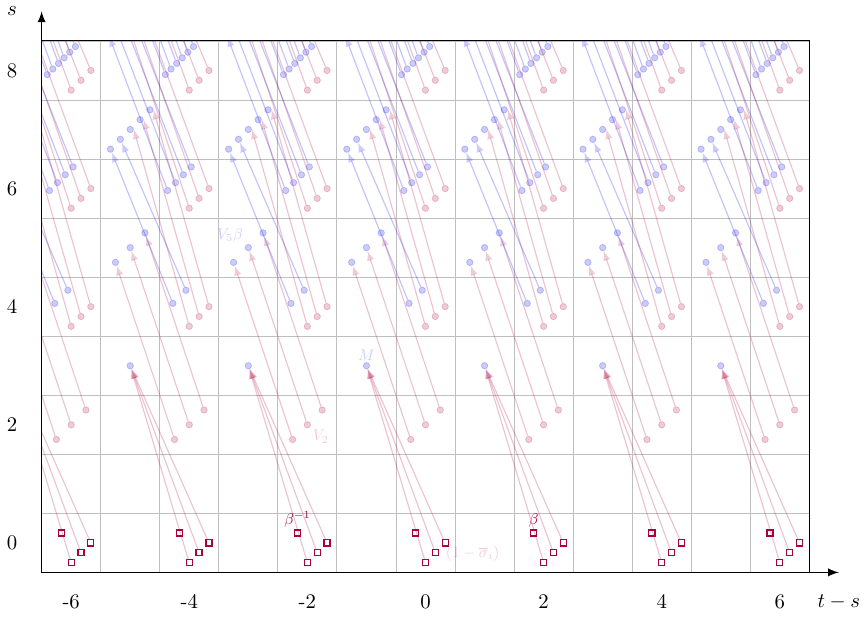}
\caption{The $E_3$-page of the $\sC$-homotopy limit spectral sequence
converging to $\pi_{t-s}^G\KG$, where $G=C_2\times C_2$. The $A$-terms from \eqref{eq:line-1} and the differentials emanating from them are tinted blue. The $B$-terms from \eqref{eq:line-2} and the differentials emanating from them are tinted red.}\label{fig:sc-holim-ss}
\end{figure}

Having determined the behavior of the spectral sequence in
\eqref{eq:classical-ahss} we can now determine the differentials emanating from
the $A$-summand. We see that $d_3(z,0)=(z^2,0)$ and $d_3(0,z)=(0,z^2)$ and that
these generate all $d_3$ differentials emanating from the $A$-term (see
\Cref{fig:sc-holim-ss}). Moreover all of the remaining classes from $A$ are
permanent cycles from $E_4$-onward, since they come from permanent cycles in the
Atiyah-Hirzebruch spectral sequence converging to $\KG^*_G(\Sigma(EG_+\vee~EG_+))$. 

Since the $B$-summand \eqref{eq:line-2} is concentrated in even degrees, any possible $d_3$ emanating from it must land in the $A$-summand. Let us now calculate the differentials coming out of the zero line. An elementary analysis of the restriction maps $R(C_2\times C_2)\rightarrow R(H_i)$ shows that the degree 0 part of the $E_2$-edge homomorphism
\[ R(C_2\times C_2)\rightarrow \lim_{\sOGC^\op} R(-)\cong \bZ\oplus \bigoplus_{i=1}^3 \bZ\{1-\overline{\sigma}_i\}\] sends the unit summand isomorphically to itself, while sending $(1-\sigma_i)$ to $\sum_{j\neq i}(1-\overline{\sigma}_j)$. It follows that the restriction map is injective with cokernel $\bZ/2$ generated by $\Delta$. We can choose $(1-\overline{\sigma}_i)$ as a generator of $\Delta$ for any $i$. Since the spectral sequence converges we know that that all terms in positive filtration must die and that $\Delta$ must support a differential, i.e., $d_i(\Delta)\neq 0$ for some $i\geq 3$.

We will now show that $d_3(\Delta)\neq 0$. Examining the $E_2$-term from \Cref{thm:example-e2} we see that \[d_3(\Delta)\in H^3_{G}(E\sC;\pi_2^{(-)} \KG)\cong \bZ/2\] which is generated by \[M=(y_2,0)\beta\equiv (y_3,y_3)\beta\equiv (0,y_1)\beta.\] Now $M\in A$ is a permanent cycle. Since the positive filtration terms can not survive the spectral sequence, $M$ must be hit by a differential emanating from the zero line. It follows that \[d_3(\Delta)=d_3(1-\overline{\sigma}_i)=M\] for each $i$.

The remaining terms in $E_4^{0,*}$ are the free abelian groups generated by \[1,(1-\overline{\sigma}_1)+(1-\overline{\sigma}_3),(1-\overline{\sigma}_2)+(1-\overline{\sigma}_3), \textrm{ and }2(1-\overline{\sigma}_3).\] These are in the image of the restriction map and hence survive to the $E_\infty$-page.
\begin{figure}
\begin{tabu}{ c   c   }
\hline
 Term & Mod-$(2, \beta-1)$ Poincar\'e Series \\
\hline
 $A$ & $t\left(\frac{2(1+t^3)}{(1-t^2)^2}-\frac{3t^2}{1-t^2}-2\right)$\\
 $B$ & $1+\frac{3}{1-t^2}$\\
 $d_3\Delta$ & $1+t^3$\\
 $d_3(z,0)y_1^*y_2^*$ & $t\left(\frac{t^3(1+t^3)}{(1-t^2)^2}\right)$\\
 $d_3(0,z)y_1^*y_2^*$ & $t\left(\frac{t^3(1+t^3)}{(1-t^2)^2}\right)$\\
 $d_3((1-\overline{\sigma}_i)y_i)y_i^*$ & $\frac{t^2(1+t^3)}{1-t^2}$\\
\hline
\end{tabu}
\caption{Poincar\'e series calculations. The Poincar\'e series of a differential is defined to be the series for the dimension of the image plus the dimension of the vector space mapping injectively to the image.\label{fig:poincare-series}}
\end{figure}

For the remaining terms we examine the Poincar\'e series for $H^*_G(E\sC;R(-))$ (see \Cref{fig:sc-holim-ss,fig:poincare-series}). We can argue inductively on the filtration degree to see that all of the $A$-terms which do not support a differential must be the target of a $d_3$ and that $d_3$ is injective on the positive degree terms in $B$. For example, one can see that the 3-dimensional vector space $V_5 \beta$ in filtration degree 5 must be in the image of a differential. Since the only possible differential out of the zero line is the $d_3$ we just calculated, we see that $V_5\beta$ must be the image of a $d_3$ coming from the 3-dimensional vector space \[V_2=\bF_2\{(1-\overline{\sigma}_i)y_i\}_{1\leq i\leq 3}\] in filtration degree 2. This pattern continues with $d_3$-differentials yielding isomorphisms between the remaining pairs of 3-dimensional vector spaces.

It follows that the $\sC$-homotopy limit spectral sequence collapses at  $E_4$ onto the zero line. 
\end{proof}

\bibliographystyle{alpha}
\bibliography{biblio}

\end{document}